\documentclass[11pt,reqno,a4paper]{amsart}

\usepackage{amsmath,amsthm,amsfonts,amssymb,mathrsfs}
\usepackage{inputenc,mathrsfs}

\numberwithin{equation}{section}
\usepackage[dvips]{graphicx}
\usepackage[colorlinks]{hyperref}

\newtheorem{theorem}{Theorem}[section]
\newtheorem{definition}{Definition}[section]
\newtheorem{conjecture}{Conjecture}[section]

\newtheorem{proposition}{Proposition}[section]
\newtheorem{lemma}{Lemma}[section]
\newtheorem{corollary}{Corollary}[section]

\newtheorem{remark}{Remark}[section]

\DeclareMathOperator{\diam}{\mathrm{diam}}
\DeclareMathOperator{\riem}{\mathrm{Rm}}
\DeclareMathOperator{\ric}{\mathrm{Ric}}
\DeclareMathOperator{\scal}{\mathrm{Scal}}
\DeclareMathOperator{\hess}{\mathrm{Hess}}

\DeclareMathOperator{\vol}{\mathrm{vol}}

\newcommand{\img}{\mathrm{Im}}

\title[Diameter bounds in 3d Type I Ricci flows]{Diameter bounds in 3d Type I Ricci flows}

\author[]{Panagiotis Gianniotis}
\address{Department of Mathematics\\
National and Kapodistrian University of Athens}
\email{pgianniotis@math.uoa.gr}

\begin{document}

\begin{abstract}
We prove that a three dimensional compact Ricci flow that encounters a Type I singularity has uniformly bounded diameter up to the singular time, thus giving an affirmative answer - for Type I singularities - to a conjecture of Perelman. To achieve this, we introduce a concept of a neck-region for a Ricci flow, analogous to the neck-regions introduced by Jiang--Naber and Cheeger--Jiang--Naber, in the study of Ricci limit spaces. We then prove that the associated packing measure is, in a certain sense, Ahlfors regular, a result that holds in any dimension.
\end{abstract}

\maketitle
\tableofcontents

\section{Introduction}
A Ricci flow $g(t)$, $t\in [0,T]$ on a compact Riemannian manifold $M$ is a smooth one parameter family of smooth Riemannian metrics on $M$ such that 
\begin{equation*}
\frac{\partial g}{\partial t}= -2\ric_{g(t)},
\end{equation*}
where $\ric_{g(t)}$ is the Ricci tensor of $g(t)$. The Ricci flow has proven to be a powerful tool in the study of the interaction between Geometry and Topology, most notably after the proof of Thurston's Geometrization Conjecture of $3$-manifolds, by Perelman in \cite{Per02,Per03a,Per03b}. 

Understanding the singularity formation of Ricci flow is a crucial step towards its potential applications. By now there is a quite complete understanding of the behaviour of Ricci flow in dimension three, starting from the work of Hamilton \cite{Ham82,Ham93} and Perelman \cite{Per02,Per03a,Per03b}, until the work of Kleiner--Lott \cite{KL17,KL20} and Bamler--Kleiner \cite{BamK22}, which allows the construction of a unique 3d Ricci flow through singularities.

From the work of Perelman, high curvature regions of a 3d (orientable) Ricci flow either resemble tubes, diffeomorphic to $\mathbb S^2\times \mathbb R$ or tubes capped on one of their ends. Subsequently, at the singular time, these regions pinch into horns, namely regions diffeomorphic to $\mathbb S^2\times \mathbb R$ whose curvature becomes unbounded at least at one of their ends. Although the local, at the curvature scale, geometry of all these regions resembles that of a long piece of the round cylinder $\mathbb S^2\times\mathbb R$ their global geometry is less understood. In fact, Perelman in Section 13.2 of \cite{Per02} states that ``the horns most likely have finite diameter, but at the moment I don't have a proof of that".

The main result of this paper confirms this conjecture for 3d Ricci flows that develop Type I singularities, namely their curvature tensor satisfies \eqref{eqn:intro_type_i}.

\begin{theorem}\label{thm:intro_bounded_diameter}
Let $M$ be a compact three dimensional manifold and $g(t)$, $t\in [0,T)$ be a smooth Ricci flow on $M$, which may become singular at $t=T$. Suppose that there is a constant $C_I<+\infty$ such that
\begin{equation}\label{eqn:intro_type_i}
|\riem|\leq \frac{C_I}{T-t}
\end{equation}
on $M\times [0,T)$. Then, there is a constant $C=C(g(0))<+\infty$ such that
\begin{align*}
\diam_{g(t)}(M) \leq C,\\
\int_M |\riem|(\cdot,t) d\vol_{g(t)} \leq C,
\end{align*}
for every $t\in [0,T)$.
\end{theorem}

Clearly, one can hope that Theorem \ref{thm:intro_bounded_diameter} remains valid in higher dimensions as well.
\begin{conjecture}\label{conj}
Let $M$ be a compact manifold and $g(t)$, $t\in [0,T)$ be a smooth Ricci flow on $M$, which may become singular at $t=T$. Then, possibly under \eqref{eqn:intro_type_i}, the diameter of $g(t)$ is uniformly bounded for every $0\leq t<T$.
\end{conjecture}

In fact, the main tool developed in this paper, the neck structure Theorem \ref{thm:neck_structure}, holds in any dimension, under some a priori assumptions which are natural in the setting of Type I singularities. 

\subsection{Discussion of the problem}
It has long been understood, by the work of Topping in \cite{Top05}, that the diameter under a Ricci flow on a compact $n$-dimensional manifold is controlled by the $L^{\frac{n-1}{2}}$ norm of the scalar curvature, which exhibits the same scaling behaviour. In dimension three, in particular, this is just the $L^1$ norm of the scalar curvature.

Kleiner--Lott in \cite{KL20} proved that in a three dimensional Ricci flow - and in fact even along the three dimensional singular Ricci flow through singularities they construct in \cite{KL17} - for any $p\in (0,1)$ the $L^p$-norm of the scalar curvature is uniformly bounded. One can construct examples, however, in which the scalar curvature is uniformly bounded in $L^1$. Such example is the ``shrinking ring" Ricci flow on $\mathbb S^2\times \mathbb S^1$, which is the quotient of the shrinking cylinder Ricci flow on $\mathbb S^2\times \mathbb R$ by a $\mathbb Z$-action of translations along the $\mathbb R$ factor. It becomes singular in finite time and the $L^p$ norm of the Riemann tensor is uniformly bounded if and only if $p\leq 1$.

Thus it is reasonable to expect that for a general compact 3d Ricci flow this is also true, which is what Theorem \ref{thm:intro_bounded_diameter} confirms for Type I singularities. However, in higher dimensions, the $L^{\frac{n-1}{2}}$ norm of the curvature is not necessarily uniformly bounded, as it fails to do so even in the simple example of the product Ricci flow on $\mathbb S^2\times T^{n-2}$, where $T^{n-2}$ is the flat $(n-2)$-dimensional torus and the Ricci flow on $\mathbb S^2$ is just the shrinking sphere. In this example, the curvature is only uniformly bounded in $L^1$, and this is the best one can expect in general. For Type I singularities, it was shown by the author in \cite{G19} that for any $p<1$ the curvature is uniformly bounded in $L^p$. Recently, this estimate was shown to hold in general by Fang--Li in \cite{FangLi25}.

Although the Ricci curvature along the ``shrinking ring" Ricci flow is non-negative, this is no longer true when a 3d Ricci flow encounters a neck-pinch singularity, say at $t=0$. In fact, considering times $t_i=-4^{-i}\rightarrow 0$, the flow, in high curvature regions and in the duration the time intervals  $[t_i,t_{i+1}]$, is only $\varepsilon$-close to the shrinking cylinder. Thus, the length of the neck might be increasing by a factor $1+O(\varepsilon)$ in the duration of each interval $[t_i,t_{i+1}]$, becoming unbounded as $t\rightarrow 0$. On the other hand, $\varepsilon\rightarrow 0$ as the curvature blows up, so in this toy example we actually know that $\varepsilon_i\rightarrow 0$. The main issue here is the precise rate of this convergence, since controlling the total expansion of the neck requires that $\sum_i \varepsilon_i <+\infty$.

Similar issues, regarding the rate of convergence to a singularity model, are common in the study of the singular set of solutions to many other geometric pde, and are related to the rectifiability of the singular set and the uniqueness of pointed blow up limits at singular points. The influential work of L. Simon \cite{Simon83, Simon96} on energy minimizing maps with analytic targets established \L ojasiewicz inequalities as a powerful tool to understand these issues. For the mean curvature flow, using an appropriate \L ojasiewicz inequality Schulze \cite{Schulze14} proved the uniqueness of compact tangent flows,  and later Colding--Minicozzi \cite{CM15,CM16} established a \L ojasiewicz inequality for the mean curvature flow close to cylindrical shrinkers, and proved uniqueness of cylindrical tangent flows and the rectifiability of the singular set. Regarding the behaviour of the mean curvature flow exhibiting asymptotically conical singularities, Chodosh--Schulze \cite{CS21} established the uniqueness of asymptotically conical tangent flows, also by proving an appropriate \L ojasiewicz inequality near asymptotically cylindrical shrinkers. The \L ojasiewicz inequality of \cite{CM15} was also exploited by the author and Haslhofer in \cite{GH20} to prove that the intrinsic diameter of a $2$-convex mean curvature flow of hypersurfaces in Euclidean space is uniformly bounded, thus answering the analogue of Conjecture \ref{conj} in this setting. Later, Wenkui Du in \cite{Du21}, taking advantage of the \L ojasiewicz inequality of \cite{CS21} as well as the proof of the mean convex neighbourhood conjecture by Choi--Haslhofer--Hershkovits \cite{CHH22} and Choi--Haslhofer--Herschkovits--White \cite{CHHW22}, extended the results in \cite{GH20},  to mean curvature flows that only exhibit singularities of either neck or conical type. A \L ojasiewicz-type inequality for non-compact singularity models of Ricci flow remains an interesting open problem, although Colding--Minicozzi \cite{CM25} have recently managed to prove the strong rigidity of cylinders, as shrinking Ricci solitons, see also the work of Li--Wang \cite{LiWang24} for the particular case of cylinders splitting one Euclidean factor.

Unfortunately, there are many situations in which an appropriate \L ojasiewicz inequality is not available, and might even not  hold. Such is the case, for example, of stationary harmonic maps, or even energy minimizing maps to a non-analytic target. For the mean curvature flow, the \L ojasiewicz inequalities in \cite{CM15,CS21} depend strongly on the geometry of the singularity models (cylindrical and asymptotically conical, respectively), and although it is expected that these results have analogues for the Ricci flow, it is known \cite{BCCD24} that there are shrinking Ricci solitons that belong to none of these categories.

However, it has recently been possible to establish rectifiability results for singular sets in a series of such situations, using weaker methods. For instance, the singular set of stationary or energy minimizing maps is indeed rectifiable, by the work of Naber--Valtorta in \cite{NV}. The same is true for the singular set of non-collapsed Gromov--Hausdorff limits of sequences of Riemannian manifolds with bounded Ricci curvature, by Jiang--Naber \cite{JN}, and with merely lower bounded Ricci curvature, by Cheeger--Jiang--Naber \cite{CJN}.

In particular, to prove the rectifiability of the singular sets in \cite{JN, CJN}, the authors not only need to establish a Reifenberg type property for the singular set, but also to control the linear (splitting) behaviour of the geometry of the space itself, as it may degenerate in small scales. This requires them to perform a very fine analysis of \textit{harmonic almost splitting maps}, which quantify the extend that the local geometry splits Euclidean directions, since these may degenerate in small scales. 

Clearly, one should expect similar issues when dealing with a Ricci flow $(M,g(t))_{-A\leq t<0}$ that encounters a finite time singularity at $t=0$. However, in contrast to \cite{JN,CJN} where the underlying geometry is well behaved in the Gromov--Hausdorff sense, the geometry of the evolving metric spaces $(M,d_{g(t)})$  may (potentially) degenerate in such a way that the limit as $t\rightarrow 0$ doesn't even exist.

\subsection{Discussion of the proof} Our approach in proving Theorem \ref{thm:intro_bounded_diameter} is inspired by \cite{JN,CJN}. Let $M$ be a compact 3-manifold and $g(t)$, $t\in [-A,0]$ a Ricci flow satisfying \eqref{eqn:intro_type_i}, for $t<0$. For $A>>1$,  we prove a neck decomposition theorem, Theorem \ref{thm:neck_decomp}, that covers a manifold by balls $B(x,-1,R)=B_{g(-1)}(x,R)$, where $R>>1$ is a large constant, and separate them into two kinds. The first kind are these balls whose points have some lower bound on the regularity scale, see Definition \ref{def:curvature_scale}, while the second kind are these balls that contain at least a point with very low regularity scale, i.e. high curvature. These are the regions of the manifold that require fine control, since the curvature is not a priori bounded there.

We then show that any ball $B(p,-1,2R)$, where $p$ is the center of a ball of the second kind, carries the structure of a \textit{neck region}, as in Definition \ref{def:neck_region}. This is similar to the neck regions in \cite{JN,CJN}, however some key modifications are needed in order to account for the lack of a fixed background metric structure. Our definition of a neck region, in rough terms, consists of a ball $B(p,-1,2R)$ of large radius, a finite subset $C\subset B(p,-1,2R)$ and a scale function $r_x>0$ for any $x\in C$, such that the flow is close to a shrinking cylinder Ricci flow at every scale $r\in [r_x,1]$. In the setting of Theorem \ref{thm:intro_bounded_diameter} we also achieve that $r_x$ is comparable to the curvature scale around $x$, in other words, the neck regions we construct are able to detect the curvature scale. Moreover, the balls $B\left(x,-r_x^2,10^{-4}Rr_x\right)$, $x\in C$, are mutually disjoint and
$$B(p,-1,R)\subset \bigcup_{x\in C\cap B\left(p,-1,\frac{3R}{2}\right)} B(x,-r_x^2,Rr_x).$$ 
It is also crucial to construct $C$ so that the function $r_y$ does not vary much around any $x\in C$, i.e. it is in some sense $\delta$-Lipschitz, again as in \cite{JN,CJN}.

However, it is unnatural to define this $\delta$-Lipschitz property with respect to any of the metrics $g(t)$, $t \in [-1,0]$, since there is no a priori control on their geometry, which is in fact what we aim to control. It turns out that the appropriate notion of ``distance" between any two points $x,y\in C$ is what we call the $R$-scale, $D_R(x,y)$, in Definition \ref{def:r_scale}. This is, essentially, the infimum of the  scales $r>0$ such that $y\in B(x,-r^2,Rr)$. We discover, in Proposition \ref{prop:D_metric}, that the $R$-scale almost satisfies the triangle inequality, in the sense that
\begin{equation}\label{eqn:intro_almost_triangle}
D_R(x,z)\leq \sigma_R (D_R(x,y)+D_R(y,z))
\end{equation}
for any $x,y,z\in M$, with $1<\sigma_R \rightarrow 1$, as $R\rightarrow +\infty$. 
This is a crucial property of $D_R$ that has implications throughout this work. It allows for a reasonable definition of a neck region and is used crucially in the construction of a neck region in Lemma \ref{lemma:3d_neck_region}, in particular in the definition of the function $r_x$. The triangle inequality \eqref{eqn:intro_almost_triangle} is  a consequence of the lower distance distortion estimate available under \eqref{eqn:intro_type_i}, and is established in Section \ref{sec:apriori}. Note that a similar notion of distance has recently appeared in \cite{FangLi25}, involving however the $W^1$-Wasserstein distance of conjugate heat kernel measures. In fact, the authors in \cite{FangLi25} establish that the usual triangle inequality holds.

Having proved that any such ball $B(p,-1,2R)$ has a neck structure, we can the main technical tool of this paper, the neck structure Theorem \ref{thm:neck_structure}, to obtain an estimate of the form
\begin{equation}\label{eqn:intro_sum}
\sum_{x\in C} r_x \leq L.
\end{equation}
This suffices to obtain an $L^1$ bound on the norm of the curvature tensor and prove Theorem \ref{thm:intro_bounded_diameter}, using Topping's result from \cite{Top05}. 

The neck structure Theorem \ref{thm:neck_structure} holds in any dimension and it does not depend on the particular gradient shrinking Ricci soliton that appears as a singularity model at each scale, as long as it splits a fixed number of $k$ Euclidean factors.

In simple terms, in the setting of 3d Ricci flow, the key ingredient in the proof Theorem \ref{thm:neck_structure} to obtain the estimate \eqref{eqn:intro_sum}, is the construction of a function $\tilde v: C\rightarrow \mathbb R$ that is, in most of $C$, bi-Lipschitz in the sense that 
\begin{equation}\label{eqn:intro_biLip}
(1-\varepsilon) \sigma D_\sigma (x,y)\leq |\tilde v(x)-\tilde v(y)|\leq (1+\varepsilon) \sigma D_\sigma(x,y)
\end{equation}
for any $x,y\in C_\varepsilon\subset C$ and some $\sigma=\sigma_{x,y}\sim R$. Here $C_\varepsilon$ is constructed to satisfy, essentially, the estimate
 $$\sum_{x\in C} r_x \leq (1+\varepsilon)\sum_{x\in C_\varepsilon} r_x.$$
 The function $\tilde v$ is defined as $\tilde v(x) = v(x,0)$,
 where $v: M\times [-1,0]\rightarrow \mathbb R$ is a solution to the heat equation which is a \textit{parabolic} almost splitting map, see Definition \ref{def:splitting_map}. 

 To prove \eqref{eqn:intro_biLip} we have to exclude the degeneration of such an almost splitting map in small scales, around points in $C_\varepsilon$, namely that its gradient does not become arbitrarily small around $x\in C_\varepsilon$, in a weighted $L^2$ sense. To achieve this, we use recent results of the author in \cite{G25}. We refer the reader in Subsection \ref{sec:apriori} for more details.
 
 {\bf Acknowledgments:}  The author would like to thank Andrea Mondino for bringing \cite{CJN} to his attention. This research was supported by the Hellenic Foundation for Research
and Innovation (H.F.R.I.) under the “2nd Call for H.F.R.I. Research Projects to support
Faculty Members \& Researchers” (Project Number: HFRI-FM20-2958). The author would also like to thank Konstantinos Leskas for his helpful comments on earlier versions of the paper.

\section{Preliminaries}

\subsection{The heat and conjugate heat equations}
Let $I\subset \mathbb R$ be an interval, and let $(M,g(t))_{t\in I}$ be a smooth compact Ricci flow. We will say that $v\in C^\infty(M\times I)$ is a solution to the heat equation if 
\begin{equation}\label{eqn:heat_equation}
\left(\frac{\partial}{\partial t} - \Delta \right)v=0
\end{equation}
 on $M\times I$, and that $u\in C^\infty(M\times I)$ satisfies the conjugate heat equation if
\begin{equation}\label{eqn:conjugate_heat_equation}
\left(\frac{\partial}{\partial t} +\Delta + \scal\right) u =0,
\end{equation}
where $\scal$ denotes the scalar curvature. 

It is well known that under \eqref{eqn:heat_equation} and \eqref{eqn:conjugate_heat_equation} 
the function $t\mapsto \int_M v(\cdot,t) u(\cdot,t) d\vol_{g(t)}$ is constant. In particular if $\int_M u(\cdot,t_0) d\vol_{g(t_0)}=1$, for some $t_0\in I$, then
$$\int_M u(\cdot,t) d\vol_{g(t)} =1$$
for every $t\in I$. It follows that for every $t\in I$, $d\nu_t = u(\cdot,t) d\vol_{g(t)}$ defines a probability measure on $M$. We will call the family $\nu=(\nu_t)_{t\in I}$ a conjugate heat flow on $(M,g(t))_{t\in I}$.

Given $(x,t)\in M\times I$ we denote by $u_{(x,t)}$ the conjugate heat kernel starting at $(x,t)$, namely $u_{(x,t)}$ solves the conjugate heat equation and for every $\varphi\in C^\infty(M)$
$$\int_M \varphi(y) u_{(x,t)}(y,s) d\vol_{g(s)} \rightarrow \varphi(x)$$
as $s\rightarrow t$. Conjugate heat kernels always exist and are unique, see for instance \cite{RFpartII}. Moreover, we will denote by $\nu_{(x,t)}=(\nu_{(x,t),s})_{s<t, s\in I}$ the associated conjugate heat flow, where $d\nu_{(x,t),s}=u_{(x,t)}(\cdot,s) d\vol_{g(s)}$ for every $s<t$.

Fix a large positive constant $H<+\infty$. Given any $p\in M$ and assuming that $\sup I=0$, we will say that a conjugate heat flow $\nu_t$, $d\nu_t = (4\pi |t|)^{-n/2} e^{-f(\cdot,t)} d\vol_{g(t)}$ satisfies the condition (CHF) with respect to $p$ if 
\begin{equation}\label{eqn:CHF}
\frac{d_{g(t)}(p,x)^2}{H|t|} - H \leq f(x,t) \leq H\left( \frac{d_{g(t)}(p,x)^2}{|t|} +1\right),\tag{CHF}
\end{equation}
for every $(x,t) \in M\times I$.

\subsection{Monotone quantities} Let $M$ be a compact manifold, $n=\dim M$, $f\in C^\infty(M)$ and $\tau>0$. Perelman's $\mathcal W$-entropy \cite{Per02} is given by 
\begin{equation*}
\mathcal W(g,f,\tau)=\int_M \left(\tau(\scal+|\nabla f|^2) +f - n\right) u d\vol,
\end{equation*}
where $u=(4\pi \tau)^{-n/2} e^{-f}$, while the $\mu$-entropy is given by
$$\mu(g,\tau) = \inf \left\{\mathcal W(g,f,\tau), f\in C^\infty(M), \int_M (4\pi \tau)^{-n/2} e^{-f} d\vol=1\right\},$$
and the $\nu$-entropy by
$$\nu(g,\tau) = \inf\{\mu(g,\tau'), 0<\tau'\leq \tau\}.$$
By \cite{Per02}, if $(M,g(t))_{t\in I}$ is smooth Ricci flow and for $t_0\in I$ $u=(4\pi |t_0-t|)^{-n/2} e^{-f}$ evolves by the conjugate heat equation then if $t_1<t_2<t_0$
$$\mathcal W(g(t_1),f(\cdot,t_1), |t_0-t_1|) \leq \mathcal W(g(t_2),f(\cdot,t_2), |t_0-t_2|),$$
and for every $\tau>0$
\begin{align*}
\mu(g(t_1),\tau+t_2-t_1)\leq \mu(g(t_2),\tau),\\
\nu(g(t_1),\tau+t_2-t_1)\leq \nu(g(t_2),\tau).
\end{align*}
Moreover, if $\max I=0$ we defined the pointed entropy $\mathcal W_p(\tau)$ at scale $\tau>0$ by
\begin{equation}\label{eqn:def_pointed_entropy}
\mathcal W_p(\tau) = \mathcal W(g(-\tau),f_{(p,0)}(\cdot,-\tau),\tau),
\end{equation}
where $u_{(p,0)}=(4\pi \tau)^{-n/2} e^{-f_{(p,0)}}$ is the conjugate heat kernel starting at $(p,0)$.
\subsection{Gradient shrinking Ricci solitons} Let $(M,\bar g)$ be a complete Riemannian manifold and $\bar f\in C^\infty (M)$. We will say that $(M,\bar g,\bar f)$ is a gradient shrinking Ricci soliton at scale $\tau>0$ if 
$$\ric_{\bar g}+\hess_{\bar g} \bar f = \frac{\bar g}{2\tau}.$$
We call $\bar f$ a soliton function. Note that soliton functions are uniquely defined, modulo linear functions, namely a functions with vanishing Hessian.

We will say that $(M,\bar g,\bar f)$ is normalized if
\begin{equation}\label{eqn:gsrs_n}
\int_M (4\pi \tau)^{-n/2} e^{-\bar f} d\vol_{\bar g}=1.
\end{equation}
The integral on the right-hand side of \eqref{eqn:gsrs_n} is finite, since $\bar f$ has quadratic behaviour while the volume grows polynomially, see \cite{HasMu}. When $(M,\bar g,\bar f)$ is normalized, the entropy 
\begin{equation}\label{eqn:entropy_soliton}
\mu(\bar g)= \mathcal W(\bar g,\bar f,\tau)
\end{equation}
is well defined and depends only on the metric $\bar g$, by \cite{MM15}.

A gradient shrinking Ricci soliton $(M,\bar g,\bar f)$ at scale $1$ induces a Ricci flow $g(t)=|t| \varphi_t^*\bar g$, $t<0$, on $M$, where $\varphi_t: M\rightarrow M$ is the smooth one parameter family of diffeomorphisms that satisfies
\begin{align*}
\frac{d}{dt} \varphi_t &= \frac{1}{|t|} \nabla^{\bar g} \bar f \circ\varphi_t,\\
\varphi_{-1}&=id_M.
\end{align*}
We will say that $g(t)$ is the selfsimilar Ricci flow induced by the gradient shrinking Ricci soliton $(M,\bar g,\bar f)$. Then, setting $f_t(x):=f(x,t)=\bar f(\varphi_t(x))$, it is standard that $(M,g(t),f_t)$ is a gradient shrinking Ricci soliton at scale $|t|$ with $\mu(g(t))=\mu(\bar g)$, for every $t<0$. Moreover, $u=(4\pi |t|)^{-n/2} e^{-f}$ evolves by the conjugate heat equation. Denote by $\nu_f$ the associated conjugate heat flow on $(M,g(t))_{t<0}$. By \cite{HasMu}, there is a $p\in M$ such that $f_t$ attains a minimum for every $t<0$ and $u$ satisfies \eqref{eqn:CHF} for $H=H(n,\mu(\bar g))$. 

Given a selfsimilar Ricci flow $(M,g(t))_{t<0}$ induced by a gradient shrinking Ricci soliton we define its spine $\mathcal S$ as
$$\mathcal S=\{\nu_f,  (M,g(t),f_t) \;\textrm{is a normalized gradient shrinking Ricci soliton at scale}\; |t| \}.$$
Moreover, its point-spine $\mathcal S_{\textrm{point}}$ is defined as
$$\mathcal S_{\textrm{point}} = \{x\in M, \textrm{ there is $\nu_f\in\mathcal S$ such that $f_t$ attains its minimum at $p$ for every $t<0$}\}.$$

The following proposition elaborates on the structure of $\mathcal S$ and $\mathcal S_{\textrm{point}}$, and is crucial for what follows.

\begin{proposition}[Proposition 4.1 in \cite{G25}] \label{prop:spine}
 Let $(M,g(t))_{t<0}$, $\dim M=n$, be the selfsimilar Ricci flow induced by a gradient shrinking Ricci soliton, with spine $\mathcal S$. There is a maximal $0\leq k\leq n$ such that, up to isometry, $(M,g(t))_{t<0}=(M'\times \mathbb R^k, g'(t)\oplus g_{\mathbb R^k})_{t<0}$, where $(M',g'(t))_{t<0}$ is the Ricci flow induced by a gradient shrinking Ricci soliton on $M'$, and $g_{\mathbb R^k}$ is the Euclidean metric in $\mathbb R^k$. Moreover,
\begin{enumerate}
\item The spine $\mathcal S'$ of $(M',g'(t))_{t<0}$ consists of a unique conjugate heat flow $\nu_{f'}$ and the elements $\nu_f$ of $\mathcal S$ are exactly of the form 
$$f((x,a), t) =\frac{|a-a_f|^2}{4|t|} +f'(x,t)$$
for some $a_f\in\mathbb R^k$.
\item If $\nu_f\in\mathcal S$ satisfies (CHF) with respect to $p$ for some constant $H<+\infty$ then $p=(q,a_f)$ for some $q\in M'$.
\item There is $\mathcal K\subset M'$ such that $\mathcal S_{\textrm{point}}=\mathcal K\times \mathbb R^k$ and for every $t<0$
$$\diam_{g'(t)} (\mathcal K) \leq A(n) \sqrt{|t|}.$$
\end{enumerate}
\end{proposition}

\subsection{$k$-independent subsets}

\begin{definition}\label{def:k_ind}
Let $(X,d,p)$ be a complete pointed metric space, $\alpha>0$ and $R<+\infty$. We will call as subset $\{x_i\}_{i=0}^k$,
\begin{enumerate}
\item $(k,\alpha)$-independent in $B(p,R)$ if there is no $0\leq l<k$ and $\{x_i'\}_{i=0}^k\subset\mathbb R^k$ such that
$$|d(x_i,x_j) - |x_i'-x_j'||<\alpha R.$$
\item $(k,\alpha, D)$-independent in $B(p,R)$, where $D>0$, if there is no $0\leq l <k$, a compact metric space $(K,d_K)$ with $\diam(K)\leq D$, and $\{x_i'\}_{i=0}^k \subset K\times \mathbb R^l$ such that for any $i,j$
$$|d(x_i,x_j) - d_{K\times \mathbb R^l}(x_i',x_j')| < D+\alpha R,$$
where $K\times \mathbb R^l$ is endowed with the product metric.
\end{enumerate}
\end{definition}
Recall the following remark form \cite{G25}.

\begin{remark}[Remark 8.1 in \cite{G25}]\label{rmk:k_ind}
If $X\subset B(0,R)\subset \mathbb R^n$ does not contain any $(k,\delta)$-independent subsets and $0<\delta\leq \delta(n|\varepsilon)$, then there is a $k-1$ plane $L$ in $\mathbb R^n$ such that $X\subset B_{\varepsilon R} (L)$.
\end{remark}

The following lemma gives the relation between the two notions of $k$-independence of Definition \ref{def:k_ind}.

\begin{lemma}[Lemma 8.2 in \cite{G25}]\label{lemma:k_ind}
Let $(K,d_K)$ be a metric space with $\diam(K)\leq D$, and let $\{x_i=(q_i,a_i)\}_{i=0}^k \subset B((q,0) , R) \subset K\times \mathbb R^k$, where $K\times \mathbb R^k$ is endowed with the product metric. Then,
\begin{enumerate}
\item If $\{x_i\}$ is $(k,\alpha,D)$-independent then $\{a_i\}\subset B(0,R)$ is $(k,\alpha)$-independent.
\item If $\{a_i\}\subset B(0,R)$  is $(k,3\alpha)$-independent and $R\geq \frac{D}{\alpha}$ then $\{x_i\}$ is $(k,\alpha,D)$-independent.
\end{enumerate}
\end{lemma}

\section{Ricci flows under a priori assumptions}\label{sec:apriori}
\subsection{A priori assumptions}
Let $I\subset\mathbb R$ be an interval with $\sup I=0$ and $(M,g(t))_{t\in I}$ be a smooth   Ricci flow. We will say that $(M,g(t))_{t\in I}$ satisfies
\begin{enumerate}
\item[(RF1)]  if $(M,g(t))_{t\in I}$ satisfies
$$\sup_M |\riem|(\cdot,t) \leq \frac{C_I}{|t|},$$
for every $t\in I$, $t<0$.
\item[(RF2)] if $M$ is compact and $\nu(g(t), |t|) \geq -\Lambda$, for every $t\in I$, $t<0$.
\item[(RF3)] if for every $x,y\in M$ and $s,t\in I$, $s<t$ the conjugate heat kernel $$u_{(x,t)}(y,s)=(4\pi |t-s|)^{-n/2} e^{-f_{(x,t)}(y,s)}$$ satisfies
\begin{equation*}
\frac{d_{g(s)}(x,y)^2}{H|t-s|} - H\leq f_{(x,t)}(y,s) \leq H\left(\frac{d_{g(s)}(x,y)^2}{|t-s|} + 1\right).
\end{equation*}
\item[(RF4)] if for every $x,y\in M$ and $s,t\in I$, $s<t$,  
\begin{equation*}
d_{g(s)}(x,y) \leq d_{g(t)}(x,y) + K\left(\sqrt{|s|}-\sqrt{|t|}\right).
\end{equation*}
\end{enumerate}

We will always assume that the constants $C_I,H,K<+\infty$ are large, in particular 
$$C_I\geq 1, H\geq 1, K\geq 1.$$

\begin{remark}\label{rmk:nc}
Note that under condition (RF2) and Perelman's no local collapsing theorem we know that there is a constant $\kappa=\kappa(n,\Lambda)$ such that for every $x\in M$, $t\in I$, $t<0$, and $0<r\leq \sqrt{|t|}$, such that $\scal \leq r^{-2}$ in $B(x,t,r)$
$$\vol_{g(t)}(B(x,t,r)) \geq \kappa r^n.$$
In particular, under both assumptions (RF1-2), since there is $a(n)>0$ such that in $B(x,t, \sqrt{|t|})$
$$\scal_{g(t)} \leq \frac{a(n,C_I)}{|t|},$$
it follows that there is a constant $\kappa_{nc}=\kappa_{nc}(n,C_I,\Lambda)>0$ such that
$$\vol_{g(t)}(B(x,t,\sqrt{|t|})\geq \kappa_{nc} |t|^{n/2}.$$
\end{remark}

\begin{remark}\label{rmk:chk}
By Proposition 2.7 and Proposition 2.8 in \cite{MM15} we know that under (RF1-2) there is a constant $H=H(n,C_I,\Lambda)<+\infty$ such that (RF3) holds.
\end{remark}

\begin{remark}
Note that condition (RF4) here is stronger than the analogous condition in \cite{G25}. This will be needed for the results of the following Subsection \ref{sec:distortion}.
\end{remark}

\begin{remark}\label{rmk:RF1_RF4}
If $(M,g(t))_{t\in I}$, $\sup I=0$ satisfies (RF1), then there is a $K=K(n,C_I)$ such that (RF4) holds. One can see this by applying Perelman's distance distortion Lemma I.8.3(b) in \cite{Per02}. Suppose that $x,y\in M$ satisfy $d_{g(t_0)}(x,y) \geq 2\sqrt{|t_0|}$, for some $t_0\in I$. Since, by (RF1), $\ric_{g(t_0)}\leq C(n,C_I) |t_0|^{-1}$ in $B(x,t_0,\sqrt{|t_0|})\cup B(y,t_0,\sqrt{|t_0|})$, applying that lemma gives that
\begin{equation}\label{eqn:d_dist_root}
\left.\frac{d}{dt}\right|_{t=t_0} d_{g(t)}(x,y) \geq -\frac{A}{\sqrt{|t_0|}}
\end{equation}
for some constant $A=A(n,C_I)$. On the other hand, if $d=d_{g(t_0)}(x,y)<2\sqrt{|t_0|}$, we know by (RF1) that along a minimizing unit speed geodesic $\gamma$ connecting $x$ with $y$
$$\int_{0}^d \ric_{g(t_0)}(\gamma'(s),\gamma'(s)) ds\leq C(n,C_I) |t_0|^{-1} d < \frac{2C(n,C_I)}{\sqrt{|t_0|}},$$
which again gives \eqref{eqn:d_dist_root}. Integrating \eqref{eqn:d_dist_root} gives (RF4).
\end{remark}

\subsection{Smooth convergence of sequences of Ricci flows} Throughout the paper we will be considering sequences of smooth Ricci flows equipped with conjugate heat flows. In our setting, in particular assumptions (RF1-3), it is quite easy to extract convergent subsequences, in the smooth Hamilton--Cheeger--Gromov topology. In order to clarify our setup, we quickly state the necessary definitions and the compactness result that we will rely on, and refer the reader to \cite{G25} for more details.

\begin{definition}
Consider a sequence $I_j\subset \mathbb R$ of intervals that converge to the interval $I_\infty\subset \mathbb R$ in the sense that for any closed interval $J\subset I_\infty$, $J\subset I_j$ for large $j$.
\begin{enumerate}
\item We say that a pointed sequence $(M_j,g_j(t),p_j)_{t\in I_j}$ of Ricci flows converges to a pointed Ricci flow $(M_\infty,g_\infty(t),p_\infty)_{t\in I_\infty}$ if for some, and thus for any,  $t_0\in I_\infty$ the following holds: for every $R<+\infty$ there are smooth maps $F_j: B(p_\infty,t_0,R)\rightarrow M_j$, diffeomorphisms onto their image, such that $F_j(p_\infty) = p_j$ and $F_j^* g_j$ converges to $g_\infty$ in the smooth and uniform in compact subsets of $B(p_\infty,t_0,R)\times I_\infty$ topology.
\item Given a convergent sequence $(M_j,g_j(t),p_j)_{t\in I_j}\rightarrow (M_\infty,g_\infty(t),p_\infty)_{t\in I_\infty}$, we will say that $x_j\in M_j$ converges to $x_\infty\in M_\infty$, and denote by $x_j\rightarrow x_\infty$, if there is $R<+\infty$ and $F_j :B(p_\infty,t_0,R)\rightarrow M_j$ as in (1) such that $F_j^{-1}(x_j)\rightarrow x_\infty$.
\item Given a convergent sequence $(M_j,g_j(t),p_j)_{t\in I_j}\rightarrow (M_\infty,g_\infty(t),p_\infty)_{t\in I_\infty}$, we will say that a sequence of functions $f_j\in C^\infty(M_j \times I_j)$ converges to $f_\infty \in C^\infty(M_\infty\times I_\infty)$  if for any $R<+\infty$ there is $F_j :B(p_\infty,t_0,R)\rightarrow M_j$ as in (1) such that $F_j^* f_j$ converges smoothly uniformly in compact subsets of $B(p_\infty,t_0,R)\times I_\infty$ to $f_\infty$.
\item Given a convergent sequence $(M_j,g_j(t),p_j)_{t\in I_j}\rightarrow (M_\infty,g_\infty(t),p_\infty)_{t\in I_\infty}$, we will say that a sequence $(\nu_{j,t})_{t\in I_j}$ of conjugate heat flows converges to a conjugate heat flow $(\nu_{\infty,t})_{t\in I_\infty}$ if the associated solutions $u_j\in C^\infty(M_j \times I_j)$ to the conjugate heat equation converge to $u_\infty\in C^\infty(M_\infty\times I_\infty)$, associated to the conjugate heat flow $\nu_{\infty,t}$.
\end{enumerate}
\end{definition}

\begin{proposition}[Compactness of Ricci flows paired with conjugate heat flows]\label{prop:compactness_rf}
Let $(M^n_j,g_j(t),p_j)_{t\in I}$, $\sup I = 0$,  be a sequence of smooth compact Ricci flows satisfying (RF1-2). Moreover, suppose that  there is sequence $(\nu_{j,t})_{t\in I}$ of conjugate heat flows that satisfy (CHF) with respect to $p_j\in M_j$, for a uniform constant $H<+\infty$. Then there is a pointed smooth complete Ricci flow $(M_\infty,g_\infty(t),p_\infty)_{t\in I}$ satisfying (RF1), and a conjugate heat flow $(\nu_{\infty,t})_{t\in I}$ on $(M_\infty,g_\infty(t))_{t\in I}$, satisfying (CHF) with respect to $p_\infty$, such that a subsequence of $(M_j,g_j(t),p_j)_{t\in I}$ converges to $(M_\infty,g_\infty(t),p_\infty)_{t\in I}$ and $(\nu_{j,t})_{t\in I}$ converges to $(\nu_{\infty,t})_{t\in I}$.
\end{proposition}

\subsection{Distance distortion} \label{sec:distortion}

In this subsection we turn our attention into various consequences of the a priori assumption (RF4), the lower distance distortion estimate. The results in this subsection are crucial for the rest of the paper. Roughly speaking, given a compact Ricci flow $(M,g(t))_{t\in I}$, $\sup I=0$, satisfying (RF4), we will define an \textit{almost} distance function on $M$, in the sense that the triangle inequality is almost satisfied.

\begin{lemma}\label{lemma:balls_inclusion}
Let $I\subset \mathbb R$ be an interval with $\sup I=0$ and $(M,g(t))_{t\in I}$ be a Ricci flow satisfying (RF4). Then,  if $\rho \geq K$ for every $s,t\in I$ with $t<s$, 
\begin{equation*}
B\left(x,s,\rho \sqrt{|s|} \right) \subset B\left(x,t, \rho\sqrt{|t|}\right).
\end{equation*}

In particular, given $x,y\in M$ the set 
$$I^\rho_{x,y}=\left\{r\in \left(0,\sqrt{-\inf I} \right), d_{g(-r^2)}(x,y) <r\rho\right\} =\left\{r\in \left(0,\sqrt{-\inf I}\right), y\in B(x,-r^2,r\rho) \right\},$$
if not empty, it is an interval of the form $\left(s,\sqrt{-\inf I}\right)$. 

\end{lemma}

\begin{proof}
Let $s,t\in I$, with $t\leq s<0$ and $y\in B(x,s,\rho \sqrt{|t|})$. Then by (RF4) we know that
\begin{equation*}
d_{g(t)}(x,y)  \leq d_{g(s)}(x,y) + K(\sqrt{|t|} -\sqrt{|s|})\leq (\rho-K)\sqrt{|s|}+ K\sqrt{|t|} \leq \rho\sqrt{|t|},
\end{equation*}
which proves the result.
\end{proof}

\begin{definition}[$\rho$-scale]\label{def:r_scale}
Let $I\subset\mathbb R$ be an interval with $\sup I=0$ and $(M,g(t))_{t\in I}$ be a Ricci flow satisfying (RF4). Given $\rho\geq K$, we define the $\rho$-scale $D_\rho:M\times M\rightarrow [0,+\infty)$ between two points  $x,y\in M$ as
\begin{equation*}
D_\rho(x,y)=\inf I_{x,y}^\rho.
 \end{equation*}
If the set $I^\rho_{x,y}$ is empty, the definition above is understood to mean that $D_\rho(x,y)=+\infty$.
\end{definition}

In the following proposition, we observe that under (RF4), if $\rho\geq 2K$ the $\rho$-scale $D_\rho$ behaves almost like a distance function on $M$.

\begin{proposition}\label{prop:D_metric}
Let $I\subset \mathbb R$ be an interval with $\sup I=0$, $(M,g(t))_{t\in I}$ be a Ricci flow satisfying (RF4) and let $\rho\geq 2K$. Then the following hold:

\begin{enumerate}
\item For any $x,y\in M$, $D_\rho(x,y)=D_\rho(y,x)\geq 0$.

\item For any $x,y\in M$  and $0<\bar r<\sqrt{-\inf I}$, $D_\rho(x,y)<\bar r$ if and only if $d_{g(-\bar r^2)}(x,y) <\bar r \rho$. In particular, $D_\rho(x,y)=0$ if and only if $d_{g(-r^2)}(x,y) <r \rho$ for every $r\in \left(0,\sqrt{-\inf I}\right)$.

\item If $\rho_i\geq \rho-K>K$, $i=1,2$, then for any $x,y,z\in M$ 
\begin{equation}
(\rho - K) D_\rho(x,y) \leq \rho_1 D_{\rho_1}(x,z) + \rho_2 D_{\rho_2}(z,y).
\end{equation}
In particular,
\begin{equation}\label{eqn:D_triangle}
D_\rho(x,y) \leq \frac{\max(\rho_1,\rho_2)}{\rho-K} (D_{\rho_1}(x,z) + D_{\rho_2}(z,y)).
\end{equation}
and if $\rho_1=\rho_2=\rho-K$, 
\begin{equation}
D_{\rho}(x,y)\leq D_{\rho-K}(x,z)+D_{\rho-K}(z,y).
\end{equation}
\item If $K\leq \rho_1\leq \rho_2$ then
\begin{equation}\label{eqn:D_compared}
D_{\rho_2}(x,y) \leq D_{\rho_1} (x,y)\leq \frac{\rho_2 -K}{\rho_1 - K} D_{\rho_2}(x,y).
\end{equation}
\item If $\rho_i\geq 2K$, $i=1,2$, then for any $x,y,z\in M$
\begin{equation}
D_{\rho_1+\rho_2} (x,z) \leq \max(D_{\rho_1}(x,y),D_{\rho_2}(y,z)).
\end{equation}
\end{enumerate}
\end{proposition}
\begin{proof}
The symmetry and the non-negativity of $D_\rho$  follow directly from Definition \ref{def:r_scale}.

To prove Assertion 2 of the proposition, first suppose that $D_\rho(x,y)<\bar r <\sqrt{-\inf I}$ for two points $x,y\in M$. Therefore,
\begin{equation}\label{eqn:inf0}
D_\rho(x,y) = \inf I^\rho_{x,y}=\inf\left\{ r\in \left(0,\sqrt{-\inf I} \right], d_{g(-r^2)}(x,y)<r\rho\right\}<\bar r.
\end{equation}
 By Lemma \ref{lemma:balls_inclusion}, we know that $I^\rho_{x,y}=\left(D_\rho(x,y),\sqrt{-\inf I}\right]$, hence $\bar r\in I^\rho_{x,y}$. This gives $ d_{g(-\bar r^2)}(x,y) <\bar r\rho$.
 
 On the other hand, if $d_{g(-\bar r^2)} (x,y) <\bar r \rho$, then $\bar r \in I^\rho_{x,y}$ hence $D_\rho(x,y) = \inf I^\rho_{x,y}< \bar r$.

To prove Assertion 3, fix any $\varepsilon>0$ and let $x,y,z\in M$. Take $r_1,r_2>0$ such that 
\begin{align*}
  D_{\rho_1}(x,z)&<r_1< D_{\rho_1}(x,z)+\varepsilon,\\
D_{\rho_2}(y,z)&<r_2< D_{\rho_2}(y,z)+\varepsilon.
\end{align*}
By Assertion 2, since $\rho_i \geq 2K$, it follows that
\begin{align*}
d_{g(-r_1^2)}(x,z)&< r_1 \rho_1,\\
d_{g(-r_2^2)}(y,z)& <r_2 \rho_2.
\end{align*}

By (RF4) we obtain
\begin{align*}
d_{g(-r_1^2)}(x,z) &\geq d_{g(-(r_1+r_2)^2)}(x,z) - Kr_2,\\
d_{g(-r_2^2)}(y,z)& \geq d_{g(-(r_1+r_2)^2)}(y,z) - Kr_1.
\end{align*}
The triangle inequality then gives
\begin{align*}
d_{g(-(r_1+r_2)^2)}(x,y) &\leq d_{g(-(r_1+r_2)^2)}(x,z) +d_{g(-(r_1+r_2)^2)}(z,y), \\
&\leq d_{g(-r_1^2)}(x,z) +d_{g(-r_2^2)}(z,y) +K(r_1+r_2),\\
&< r_1\rho_1+r_2\rho_2 +K(r_1+r_2).
\end{align*}

Now, since $\rho_i\geq \rho-K$ and   $s=\frac{\rho_1}{\rho - K}r_1+\frac{\rho_2}{\rho - K}r_2\geq r_1+r_2$, applying (RF4) once more we obtain
\begin{align*}
d_{g(-s^2)}(x,y)& \leq d_{g(-(r_1+r_2)^2)}(x,y) + K(s-(r_1+r_2))\\
&< r_1\rho_1+r_2\rho_2+K(r_1+r_2)+ K(s-(r_1+r_2))\\
&=r_1\rho_1+r_2\rho_2 + Ks\\
&=\rho s,
\end{align*}
Therefore, by Assertion 2 of the proposition,
\begin{align*}
D_{\rho}(x,y) &< s= \frac{\rho_1}{\rho - K}r_1+\frac{\rho_2}{\rho - K}r_2\\
&\leq (\rho - K)^{-1} \left(\rho_1 D_{\rho_1}(x,z) + \rho_2 D_{\rho_2}(z,y)\right)+\frac{\rho_1+\rho_2}{\rho-K}\varepsilon.
\end{align*}
Since $\varepsilon>0$ is arbitrary,  we obtain the triangle-like inequality
\begin{equation*}
(\rho - K) D_\rho (x,y)  \leq   \rho_1 D_{\rho_1}(x,z)+\rho_2 D_{\rho_2}(y,z).
\end{equation*}
The proof of Assertion 4 is similar. First observe that if $\rho_1\leq \rho_2$ it follows from the definition of $D_{\rho_1},D_{\rho_2}$ that $D_{\rho_2}(x,y)\leq D_{\rho_1}(x,y)$.
On the other hand, suppose that
$$D_{\rho_2}(x,y)<r<D_{\rho_2}(x,y)+\varepsilon.$$
Then, $d_{g(-r^2)}(x,y)<\rho_2 r$ and, for every $s\geq r$, (RF4) gives 
\begin{align*}
d_{g(-s^2)}(x,y)&\leq d_{g(-r^2)}(x,y)+K(s-r)\\
&< \rho_2 r + K(s-r)\\
&= (\rho_2 -K)r +Ks.
\end{align*}
Setting $s=\frac{\rho_2-K}{\rho_1-K} r\geq r$, we obtain $d_{g(-s^2)}(x,y) < \rho_1 s$.
Therefore, by Assertion 2 of the proposition,  
$$D_{\rho_1}(x,y) < s=\frac{\rho_2-K}{\rho_1-K} r <\frac{\rho_2-K}{\rho_1-K} \left(D_{\rho_2}(x,y)+\varepsilon\right).$$
Since $\varepsilon>0$ is arbitrary, this proves Assertion 4.

To prove Assertion 5, observe that if $r=D_{\rho_1}(x,y)=D_{\rho_2}(y,z)$ the result follows from the triangle inequality at $t=-r^2$. Thus, it suffices to assume that $D_{\rho_1}(x,y) < D_{\rho_2}(y,z)$, since otherwise we can switch the roles of $x$ and $z$.

Let $\varepsilon>0$ small enough and $r_i>0$, $i=1,2$, be such that
\begin{equation*}
D_{\rho_1}(x,y) <r_1<D_{\rho_1}(x,y)+\varepsilon<D_{\rho_2}(y,z) <r_2 <D_{\rho_2}(y,z)+\varepsilon.
\end{equation*}
By Assertion 2, it follows that
\begin{equation*}
d_{g(-r_2^2)}(x,y) <\rho_1 r_2 \quad \textrm{and}\quad
d_{g(-r_1^2)}(y,z) <\rho_2 r_2.
\end{equation*}
The triangle inequality at time $t=-r_2^2$ then gives
$$d_{g(-r_2^2)}(x,z) \leq (\rho_1+\rho_2)r_2$$
which, by Assertion 2, implies that $D_{\rho_1+\rho_2} (x,z) < r_2 <D_{\rho_2}(y,z)+\varepsilon$. Since $\varepsilon>0$ can be chosen arbitrarily small, we obtain that $D_{\rho_1+\rho_2} (x,z) \leq D_{\rho_2}(y,z)$, which proves the result.
\end{proof}

\begin{definition}\label{def:lip}
Let $I\subset\mathbb R$ be an interval with $\sup I=0$, $(M,g(t))_{t\in I}$ be a Ricci flow satisfying (RF4) and $\rho\geq 2K$. We will say that a function $f:M\rightarrow \mathbb R$ is $\rho$-scale $(\sigma,\delta)$-Lipschitz, where $\delta>0$ and $\sigma>1$, if for every $x,y\in M$
$$f(y)\leq \sigma f(x) + \delta D_\rho (x,y).$$
\end{definition}

\begin{remark}\label{eqn:Lip_distance}
In the setting of Definition \ref{def:lip}, for any $p\in M$ the function $d(x) = D_\rho (p,x)$ is $\rho$-scale $\left(\frac{\rho}{\rho-K},\frac{\rho}{\rho-K}\right)$-Lipschitz, since by Proposition \ref{prop:D_metric}
\begin{equation*}
d(x)=D_\rho(p,x) \leq \frac{\rho}{\rho - K}\left(D_\rho (p,y) + D_\rho(x,y)\right)=\frac{\rho}{\rho - K} d(y) +\frac{\rho}{\rho-K} D_\rho(x,y).
\end{equation*}
\end{remark}

In light of Proposition \ref{prop:D_metric}, it will prove very  convenient to introduce the following notation for the balls with respect to $D_\rho$:
\begin{equation}
\begin{aligned}
\tilde B_\rho(x,r) &= \{y\in M, D_\rho(x,y) <r\} = B(x,-r^2,\rho r)\\
\tilde B_\rho(x,0)&=\{y\in M, D_\rho(x,y) =0\}.
\end{aligned}
\end{equation}

\begin{lemma}[$10$-covering] \label{lemma:10_covering}
Let $(M,g(t))_{t\in [-10,0]}$ be a smooth Ricci flow on a compact manifold that satisfies (RF4), $S\subset M$, and suppose that $\rho\geq 2K$. Suppose that there is a function $0<r_x\leq 1$ for every $x\in S$ such that $\inf_{x\in S} r_x>0$. Then there is a maximal finite collection of mutually disjoint balls $\{\tilde B_\rho (x,r_x)\}_{x\in A}$, such that
\begin{equation}\label{eqn:10_covering}
\bigcup_{x\in S} \tilde B_\rho(x,r_x)\subset \bigcup_{x\in A} B_{\rho}(x,10r_x).
\end{equation}
\end{lemma}
\begin{proof}
For every $j\geq 1$ define
$$U_j=\{x\in S, 2^{-j} < r_x \leq 2^{-(j-1)} \}.$$
We will construct finite subsets $A_j\subset U_j$, $j\geq 1$, maximal with respect to the property that each ball $\tilde B_\rho (x,r_x)$, $x\in A_j$ does not intersect any of the balls $\tilde B_\rho(y,r_y)$, $y\in A_1\cup\cdots \cup A_j$. 

The set $A_1\subset U_1$ is constructed by choosing any arbitrary, possibly finite, maximal sequence $A_1=\{x_i^1\}_{i\geq 1}\in U_1$ such that the balls $\tilde B_\rho (x,r_x)$, $x\in A_1$ are mutually disjoint. Then, we can construct each set $A_j\subset U_j$ by choosing $A_j=\{x^j_i\}_{i\geq 1}$, where $\{x^j_i\}_{i\geq 1}$ is a, possibly finite, maximal sequence such that  each ball $\tilde B_\rho(x^j_l,r_{x^j_l})$ does not intersect any of the balls $\tilde B_\rho(y,r_y)$, $y\in A_1\cup\cdots\cup A_{j-1}\cup \bigcup_{k=1}^{l-1} \{x_k^j\}$. By construction, it is clear that the balls $\tilde B_\rho(x,r_x)$, $x\in A_1\cup\cdots\cup A_j$ are mutually disjoint.

Each set $A_j$ constructed above is finite. Otherwise, by the compactness of $M$, we would be able to find a sequence in $A_j$ which is Cauchy with respect to the metric $d_{g(-4^{-j})}$. This  would contradict the requirement that the balls $\tilde B_\rho(x,r_x)$, $x\in A_j$ are mutually disjoint, by the distance distortion assumption (RF4) and Lemma \ref{lemma:balls_inclusion}, since $r_x\geq 2^{-j}$ in $U_j$.  The finiteness of $A_j$  implies that it is maximal with respect to the required property.

Moreover, by the assumption $\inf_{x\in S} r_x>0$, there is only a finite number of $j\geq 1$ such that $U_j$ is non-empty. It follows that $A=\bigcup_{j\geq 1} A_j$ is a finite set, as a finite union of finite sets.

We will show that $A$ is maximal with the property that the balls $\tilde B_\rho(x,r_x)$, $x\in A$ are mutually disjoint. Take any $z\in U_j$, for some $j\geq 1$. Since $A_j$ is maximal, we know that $\tilde B_\rho(z,r_z)$ must  intersect one of the balls $\tilde B_\rho(x,r_x)$, for some $x\in  A_1\cup \cdots \cup A_j \subset A$. 

It remains to prove \eqref{eqn:10_covering}. We will first show that
\begin{equation}\label{eqn:5_cover}
S\subset \bigcup_{x\in A} \tilde B_\rho(x,5r_x).
\end{equation}
Take any $z\in U_j$. It is then clear that if $x\in A_j\subset U_j$, then $r_z \leq 2r_x$. On the other hand, if $x\in A_i$ for some $1\leq i<j$, it is clear that $r_z<r_x$. 

By the maximality of $A$ we know that there is $x\in A$ and a $w\in\tilde B_\rho(x,r_x) \cap \tilde B_\rho(z,r_z)$. Applying the triangle inequality of Proposition \ref{prop:D_metric} we obtain
\begin{equation*}
D_\rho(x,z) \leq \frac{\rho}{\rho-K} (D_\rho(x,w) + D_\rho(w,z))\leq \frac{\rho}{\rho-K} (r_x + r_z)\leq \frac{3\rho}{\rho-K} r_x < 5r_x,
\end{equation*}
if $\rho\geq \rho(K)$. This proves that $z\in \tilde B_\rho(x,5r_x)$, showing \eqref{eqn:5_cover}. 

To prove \eqref{eqn:10_covering} we just need to prove that $\tilde B_\rho(z,r_z)\subset \tilde B_\rho(x,10r_x)$, where $x,z$ are as above. For this, note that applying the triangle inequality of Proposition \ref{prop:D_metric} once more, for any $y\in \tilde B_\rho(z,r_z)$, we obtain
\begin{equation*}
D_\rho(x,y) \leq \frac{\rho}{\rho-K} (D_\rho(x,z) + D_\rho (z,y)) \leq \frac{\rho}{\rho-K}(5r_x + r_z)\leq \frac{7\rho}{\rho-K} r_x <10r_x,
\end{equation*}
if $\rho\geq \rho(K)$. This proves that $\tilde B_\rho(z,r_z)\subset \tilde B_\rho(x,10r_x)$, and thus \eqref{eqn:10_covering}.

\end{proof}

We conclude this section we the following lemma.

\begin{lemma}\label{lemma:tilde_balls_inclusion}
Let $(M, g(t))_{t\in [-1,0]}$ be a smooth complete Ricci flow satisfying (RF4), and $p\in M$. If $R\geq 12 K$ then, for every $r_i=2^{-i}\leq 2^{-5}$, the following inclusions hold.
\begin{enumerate}
\item For every $z\in \tilde B_{\frac{3R}{2}}(p,1)$, $\tilde B_{2R}(z,r_i)\subset \tilde B_{\frac{5R}{3}}(p,1)$.
\item For every $z\in\tilde B_{\frac{5R}{3}}(p,1)$, $\tilde B_{2R}(z,2r_i) \subset \tilde B_{2R}(p,1).$
\end{enumerate}

\end{lemma}
\begin{proof}

Recall that, by the distance distortion estimate, for every $t\in [-1,0]$ and $x,y\in M$, we have
\begin{equation*}
d_{g(t)}(x,y) \geq d_{g(-1)}(x,y) - K(1-\sqrt{|t|}),
\end{equation*}
therefore for any $z\in M$
\begin{equation*}
B(z,-r_i^2, 2R r_i) \subset B(z,-1,2 R r_i + K).
\end{equation*}
The claim will follow by the triangle inequality at $t=-1$, once $K\leq \frac{R}{12}$ and $r_i\leq 2^{-5}$, since
\begin{align*}
\frac{3R}{2} + 2Rr_i + K <\frac{5 R}{3}  \quad\Longleftrightarrow \quad 2Rr_i + K< \frac{1}{6}R,\\
\frac{5R}{3}+4R r_i +K <2R \quad \Longleftrightarrow\quad 4Rr_i +K< \frac{1}{3}R.
\end{align*}

\end{proof}

\subsection{Almost selfsimilarity}

We will say that a pointed, smooth, complete Ricci flow $(M,g(t),p)_{t<0}$ is $k$-selfsimilar if $(M,g(t))=(M'\times\mathbb R^k, g'(t)\oplus g_{\mathbb R^k})_{t<0}$ is the selfsimilar Ricci flow induced by a gradient shrinking Ricci soliton and $p\in\mathcal S_{\textrm{point}}$. Note that for a $k$-selfsimilar Ricci flow, $k$ need not be maximal.

\begin{definition}[$(k,\delta)$-selfsimilar]\label{def:almost_ss}
A pointed smooth complete Ricci flow $(M,g(t),p)_{t\in (-2\delta^{-1} r^2,0)}$ is said to be $(k,\delta)$-selfsimilar at scale $r>0$ if $g_r(t)= r^{-2} g(r^2 t)$, $t\in (-2\delta^{-1},0)$ satisfies the following: there is a $k$-selfsimilar Ricci flow $(\tilde M,\tilde g(t),q)_{t<0}$ and a diffeomorphism onto its image $F: \tilde B_{\delta^{-1}}(q,1) \rightarrow M$ such that $F(q)=p$ and
\begin{align*}
|\tilde \nabla^l(F^* g_r - \tilde g)|_{\tilde g} <\delta,\\
|d_{\tilde g(t)}(q, \cdot) - d_{g_r(t)} (p,F(\cdot))|<\delta,
\end{align*}
in $\tilde B_{\delta^{-1}}(q,1)\times [-\delta^{-1},-\delta]$, for every $0\leq l <\delta^{-1}$.

We define $\mathcal L_{p,r} = F(\mathcal S_{\textrm{point}}\cap \tilde B_{\delta^{-1}} (q,1))$ and say that $(M,g(t),p)_{t\in (-2\delta^{-1} r^2,0)}$ is $(k,\delta)$-selfsimilar at scale $r>0$ with respect to $\mathcal L_{p,r}$.

We will also say that $(M,g(t),p)_{t\in (-2\delta^{-1} r^2,0)}$ is $\delta$-close at scale $r>0$ to $(\tilde M,\tilde g(t),q)_{t<0}$.
\end{definition}

Below we collect two results from \cite{G25} which relate almost selfsimilarity with the pointed entropy \eqref{eqn:def_pointed_entropy}.

\begin{lemma}[Lemma 4.1 in \cite{G25}] \label{lemma:a_ss_to_ss}
Let $\delta_j\rightarrow 0$ and $(M_j,g_j(t),p_j)_{t\in (-2\delta_j^{-1},0)}$ be a pointed  sequence of complete $(k,\delta_j)$-selfsimilar at scale $1$ Ricci flows, converging smoothly to a complete Ricci flow $(M_\infty,g_\infty(t),p_\infty)_{t<0}$. Then $(M_\infty,g_\infty,p_\infty)_{t<0}$ is $k$-selfsimilar with spine $\mathcal S$. Moreover, if $M_j$ are compact and $g_j(t)$ is defined in $(-2\delta_j^{-1},0]$ and satisfies (RF2-3) then the conjugate heat flows $\nu_{(p_j,0)}$ converge to a conjugate heat flow $\nu_\infty\in\mathcal S$ and 
$$\mathcal W_{p_j}(|t|) \rightarrow \mu(g_\infty(-1))$$
for every $t<0$.
\end{lemma}

\begin{lemma}[Proposition 4.5 in \cite{G25}] \label{lemma:entropy_to_ss}
Let $(M_j, g_j(t),p_j)_{t\in (-2\delta_j^{-1},0]}$, $\delta_j\rightarrow 0$, be a sequence of compact Ricci flow satisfying (RF1-3) such that
$$\mathcal W_{p_j}(\delta_j)-\mathcal W_{p_j}(\delta_j^{-1}) <\delta_j.$$
Then, passing to a subsequence, we may assume that $(M_j,g_j(t),p_j)_{t\in (-2\delta_j^{-1},0]}$ converges to a smooth complete Ricci flow $(M_\infty,g_\infty(t),p_\infty)_{t<0}$ which is induced by a gradient shrinking Ricci soliton. Moreover, there is $D=D(n,H)<+\infty$ such that 
$$d_{g(t)}(p_\infty,\mathcal S_{\textrm{point}})\leq D\sqrt{|t|}$$
for every $t<0$.
\end{lemma}

\subsection{Almost splitting maps} In this subsection we review the notion of a $(k,\delta)$-splitting map that we will use, and recall some of their basic properties, established in \cite{G25}. 

\begin{definition}[$(k,\delta)$-splitting map]
\label{def:splitting_map}
Let $(M,g(t),p)_{t\in [-r^2,0]}$ be a smooth complete pointed Ricci flow and let $1\leq k\leq n$. Then $v=(v^1,\ldots,v^k): M\times (-r^2,0]\rightarrow \mathbb R^k$ is a $(k,\delta)$-splitting map around $p$ at scale $r$ if
\begin{enumerate}
\item Each $v^a$ solves the heat equation $\frac{\partial v^a}{\partial t} = \Delta_{g(t)} v^a$.
\item For every $a=1,\ldots,k$
\begin{equation}\label{eqn:def_splitting_map_hessian}
\int_{-r^2}^{-\delta r^2}\int_{M} | \hess_{g(t)} v^a|^2 d\nu_{(p,0),t} dt \leq\delta
\end{equation}
\item For any $a,b=1,\ldots,k$
\begin{equation}\label{eqn:def_splitting_map_gradient}
\int_{-r^2}^{-\delta r^2}\int_{M} \left| \langle\nabla v^a,\nabla v^b\rangle -\delta^{ab} \right|^2  d\nu_{(p,0),t} dt\leq\delta r^2.
\end{equation}
\end{enumerate}
In general, if $(M,g(t))_{t\in [-r^2,0)}$ is a smooth complete Ricci flow and $\nu=(\nu_t)_{t\in [-r^2,0)}$ is a conjugate heat flow, we will say that a solution $v: M\times (-r^2,0)\rightarrow \mathbb R^k$ to the heat equation is a $(k,\delta)$-splitting map at scale $r$ if \eqref{eqn:def_splitting_map_hessian} and \eqref{eqn:def_splitting_map_gradient} hold with respect to $\nu_t$.
\end{definition}

\begin{lemma}[Change of center and scale of splitting maps, Lemma 7.1 in \cite{G25}]\label{lemma:near_splitting}
 Let $(M,g(t))_{t\in [-1,0]}$ be a smooth compact  Ricci flow and let $v:M\times (-1,0)\rightarrow \mathbb R^k$ be a $(k,\delta)$-splitting map around $p$ at scale $1$, and fix $r_0 \in (0,1]$ and $\varepsilon>0$.
\begin{enumerate}
\item If $0<\delta\leq\delta(r_0|\varepsilon)$ then $v$ is also a $(k,\varepsilon)$-splitting map around $p$ at scale $r$ for every $r\in [r_0,1]$.
\item Fix $s>0$. If $(M,g(t))_{t\in [-1,0]}$ satisfies (RF1-3) then there is $\gamma=\gamma(n,H)\in (0,1]$ such that if $0<\delta\leq \delta(n,C_I,H|s,\varepsilon)$ then $v$ is a $(k,\varepsilon)$-splitting map around any $q\in B(p,-1,s)$ at scale $\gamma$.
\end{enumerate}
\end{lemma}

\begin{lemma}[Normalizing a splitting map, Lemma 7.3 in \cite{G25}]\label{lemma:normalize_splitting}
Let $(M,g(t),p)_{t\in [-1,0]}$ be a smooth complete Ricci flow and $v:M \times [-1,0]\rightarrow \mathbb R$ a $(k,\delta)$-splitting map around $p$ at scale $1$. Then for every  $r\in [2\sqrt \delta, 1]$ there is a unique lower triangular $k\times k$ matrix $T_r$ such that  $v_r=T_r v$, with $v_r^a=(T_r)^a_m v^m$, which satisfies for every $a,b=1,\ldots,k$
\begin{equation}\label{eqn:Tr_on}
\frac{4}{3r^{2}}\int_{-r^2}^{-r^2/4}\int_M \langle \nabla v_r^a,\nabla v_r^b \rangle  d\nu_{(p,0),t} dt =  \delta^{ab}.
\end{equation}
Moreover, $||T_r-I_k||\leq C(n) r^{-2}\sqrt\delta$, where $||\cdot||$ denotes the maximum norm of a $k\times k$ matrix.
\end{lemma}

\begin{proposition}[Compactness of splitting maps, Proposition 7.1 in \cite{G25}]\label{prop:s_map_compactness}
Let $(M_j,g_j(t),p_j)_{t\in (-1,0]}$ be a pointed sequence of smooth compact Ricci flows satisfying (RF1-3), a sequence $\varepsilon_j$ such that $\liminf_j \varepsilon_j = \varepsilon \in [0,1]$ and a sequence $v_j:M_j\times [-1,0]\rightarrow \mathbb R^k$ of $(k,\varepsilon_j)$-splitting maps around $p_j$ at scale $1$, normalized so that
\begin{equation}\label{eqn:compactness_average_assumption}
\int_{-1}^0\int_{M_j} v_j^a d\nu_{(p_j,0),t}dt = 0.
\end{equation}
Then, there is a pointed smooth complete  Ricci flow $(M_\infty,g_\infty(t),p_\infty)_{t\in (-1,0)}$ satisfying (RF1), a conjugate heat flow $\nu_{\infty,t}$, satisfying (CHF) with respect to $p_\infty$, and a $(k,\varepsilon)$-splitting map $v_\infty:M_\infty\times (-1,0)\rightarrow \mathbb R^k$ with respect to $\nu_\infty$ at scale $1$, which, after passing to a subsequence, are the limits of $(M_j,g_j(t),p_j)_{t\in (-1,0]}$, $\nu_{(p_j,0)}$ and $v_j$ respectively.

Moreover, $v_\infty$ satisfies
\begin{equation}\label{eqn:compactness_average_conclusion}
 \int_{M_\infty} v_\infty^a d\nu_{\infty,t} dt=0,
\end{equation}
for every $t\in (-1,0)$ and
\begin{equation}\label{eqn:compactness_B_convergence}
\lim_{j\rightarrow+\infty}\int_{-1}^{-1/4}\int_{M_j} \langle\nabla v_j^a,\nabla v_j^b\rangle d\nu_{(p_j,0),t} dt = \int_{-1}^{-1/4} \int_{M_\infty}\langle\nabla v_\infty^a,\nabla v_\infty^b\rangle d\nu_{\infty,t}dt.
\end{equation}
\end{proposition}

The following lemma and its consequence, Corollary \ref{cor:coverings} which is a covering result, will be very useful in the remaining sections of the paper.

\begin{lemma}\label{lemma:split_map_selfsimillar}
Fix a scale $r>0$, $s>0$, $\eta>0$, and $\varepsilon>0$. Let $(M,g(t),p)_{t\in (-2\delta^{-2}r^2,0]}$ be a  smooth compact Ricci flow satisfying (RF1-3) such that
$
\mathcal W_p(\delta)-\mathcal W_p(\delta^{-1})<\delta.
$
Suppose that there is $x\in \tilde B_{s}(p,r)$ such that $(M,g(t),x)_{t\in (-2\delta^{-2} r^2,0]}$  is $(k,\delta^2)$-selfsimilar at scale $r>0$ with respect to $\mathcal L_{x,r}$, but not $(k+1,\eta)$-selfsimilar at scale $r$ and that $v$ is a $(k,\theta)$-splitting map around $p$ at scale $r$ such that $v(p,0)=0$ and
\begin{equation}\label{eqn:n}
\frac{4r^2}{3}\int_{-r^2}^{-\frac{r^2}{4}}\int_M \langle \nabla v^a,\nabla v^b\rangle d\nu_{(p,0),t} dt = \delta^{ab}.
\end{equation}

Define $\tilde v: M\rightarrow \mathbb R^k$ by
$$\tilde v^a(x) := r^{-2}\int_{-r^2}^0\int_M v^a d\nu_{(x,0)} dt=v^a(x,0)$$
for each $a=1,\ldots,k$. If $0<\delta\leq \delta(n,C_I,\Lambda,H,\eta|s,\varepsilon)$ and $0<\theta\leq\theta(n,C_I,\Lambda,H,\eta|s,\varepsilon)$, then:
\begin{enumerate}
\item There is $E=E(n,H)<+\infty$ such that if $q_1,q_2\in \tilde B_{s}(p,r)$ satisfy $\mathcal W_{q_i}(r^2\delta) - \mathcal W_{q_i}(r^2\delta^{-1})<\delta$ and either $d_{g(-r^2)}(q_1,q_2)\geq \varepsilon r$ or $|\tilde v(q_1)-\tilde v(q_2)| \geq \varepsilon r$, then 
\begin{equation}\label{u_bilip}
\begin{aligned}
(1-\varepsilon)(d_{g(t)}(q_1,q_2))^2 &\leq |\tilde v(q_1) - \tilde v(q_2)|^2 +E^2 |t|,\\
|\tilde v(q_1) - \tilde v(q_2)| &\leq (1+\varepsilon) d_{g(t)}(q_1,q_2),
\end{aligned}
\end{equation}
for every $t\in [-\varepsilon^{-1} r^2,-\varepsilon r^2]$.

\item The function $\tilde v$ satisfies
$$B_{\frac{sr}{2}}(0)\subset B_{\varepsilon r}\left(\img \left(\tilde v|_{\tilde B_s(p,r)\cap \mathcal L_{x,r}}\right)\right).$$
\end{enumerate}
\end{lemma}

\begin{proof}
By rescaling we may assume that $r=1$. Consider any sequence $(M_j,g_j(t),p_j)_{t\in (-2\delta_j^{-2},0]}$ of smooth compact Ricci flows satisfying (RF1-3) and
\begin{equation}\label{eqn:ent_a_c}
\mathcal W_{p_j}(\delta_j)-\mathcal W_{p_j}(\delta_j^{-1})<\delta_j,
\end{equation}
with $\delta_j\rightarrow 0$, and suppose that there exist points $x_j\in \tilde B_{s}(p_j,1)$ such that $(M_j,g_j(t),x_j)_{t\in (-2\delta_j^{-2},0]}$ are $(k,\delta_j^2)$-selfsimilar with respect to $\mathcal L_{x_j,1}$ but not $(k+1,\eta)$-selfsimilar at scale $1$.

 Let $\nu_{(p_j,0)}$ denote the conjugate heat kernel  starting at $(p_j,0)$. 

Moreover, assume that there is a sequence $\theta_j\rightarrow 0$ and let $v_j=(v_j^1,\ldots,v_j^k):M_j\times [-1,0]\rightarrow\mathbb R^k$ be a sequence of $(k,\theta_j)$-splitting maps around $p_j$ at scale $1$, that satisfies  \eqref{eqn:n} and
\begin{equation*}
\tilde v_j^a(p_j)=\int_{-1}^0\int_{M_j} v_j^a d\nu_{(p_j,0),t} dt =v_j^a(p_j,0)=0
\end{equation*}
for every $a=1,\ldots,k$.

Observe that, given any $\theta_j'\rightarrow 0$, by Lemma \ref{lemma:near_splitting} there is $\gamma=\gamma(n,H)\in (0,1)$ such that, passing to a subsequence, we may assume that  $v_j:M_j\times [-4\gamma^2,0]\rightarrow \mathbb R^k$ is a $(k,\theta'_j)$-splitting map at scale $2\gamma$, around any $q\in \tilde B_s(p_j,1)$.

 Define $\tilde v_j=(\tilde v_j^1,\ldots,\tilde v_j^k):M_j\rightarrow \mathbb R^k$ by
$$\tilde v_j^a(q)= \int_{-1}^0\int_{M_j} v_j^a d\nu_{(q,0),t} dt = \int_{M_j} v_j^a d\nu_{(q,0),-\gamma^2}=v_j^a(q,0).$$

Assume, towards a contradiction, that one of the following holds for every $j$:
\begin{enumerate}
\item[I.] There are sequences $q_{1,j},q_{2,j}\in \tilde B_s(p_j,1)\subset M_j$, satisfying either $d_{g(-1)}(q_{1,j},q_{2,j})\geq \varepsilon$ or $|\tilde v(q_{1,j}) - \tilde v(q_{2,j})|\geq \varepsilon$, and $t_j \in [-\varepsilon^{-1},-\varepsilon]$ such that 
\begin{equation}\label{eqn:edj}
\mathcal W_{q_{i,j}}(\delta_j)-\mathcal W_{q_{i,j}}(\delta_j^{-1})<\delta_j,
\end{equation}
and for some $E<+\infty$ (which will be determined in the course of the proof to depend only on $n$ and $H$) either
\begin{equation}\label{eqn:caseii_a}
\left(1-\varepsilon\right)(d_{g(t_j)}(q_{1,j},q_{2,j}))^2 > |\tilde v_j(q_{1,j}) - \tilde v_j(q_{2,j})|^2 +E^2 |t_j|,
\end{equation}
or
\begin{equation}\label{eqn:caseii_b}
|\tilde v_j(q_{1,j}) - \tilde v_j(q_{2,j})| > \left(1+\varepsilon\right) d_{g(t_j)}(q_{1,j},q_{2,j}).
\end{equation}
\item[II.] There is $w_j\in B_{\frac{s}{2}}(0) \subset \mathbb R^k$ such that for any $y\in \tilde B_s(p_j,1)\cap \mathcal L_{x_j,1}$, we have that $|\tilde v_j(y) - w_j|\geq \varepsilon$.

\end{enumerate}

By Proposition \ref{prop:compactness_rf}, Lemma \ref{lemma:a_ss_to_ss}, Proposition \ref{prop:s_map_compactness} and assumption \eqref{eqn:ent_a_c}, passing to a subsequence we may assume that:
\begin{itemize}
\item[A.] $(M_j,g_j(t),p_j)_{t\in (-2\delta_j^{-1},0]}$ smoothly converges to a Ricci flow $(M_\infty,g_\infty(t),p_\infty)_{t\in (-\infty,0)}$, and there is an $x_\infty\in \overline{\tilde B_s (p_\infty,1)}$ such that $x_j\rightarrow x_\infty$ and $(M_\infty,g_\infty(t),x_\infty)_{t\in (-\infty,0)}$ is $k$-selfsimilar with spine $\mathcal S$ and point-spine $\mathcal S_{\textrm{point}}$, but not $k+1$-selfsimilar. Moreover, by \eqref{eqn:ent_a_c} and Lemma \ref{lemma:entropy_to_ss} we know that for every $t<0$
\begin{equation}\label{eqn:p_infty_S_point}
d_{g_{\infty}(t)}(p_\infty, \mathcal S_{\textrm{point}}) \leq D \sqrt{|t|},
\end{equation}
where $D=D(n,H)<+\infty$.
\item[B.] The conjugate heat flows $\nu_{(p_j,0)}$ smoothly converge to a conjugate heat flow $\nu_\infty\in \mathcal S$, that satisfies (CHF) with respect to $p_\infty$. Let $d\nu_{\infty,t}=(4\pi|t|)^{-n/2} e^{-f_\infty(\cdot,t)} d\vol_{g_\infty(t)}$.
\item[C.] $v_j$ smoothly converges to a map $v_\infty=(v_\infty^1,\ldots,v_\infty^k):M_\infty\times (-1,0)\rightarrow \mathbb R^k$ such that for any $a,b=1,\ldots,k$
\begin{equation}\label{eqn:vinfty}
\hess v_\infty^a =0\quad \textrm{and} \quad 
\langle \nabla v_\infty^a,\nabla v_\infty^b \rangle =\delta^{ab}
\end{equation}
in $M_\infty\times (-1,0)$ and
\begin{equation}\label{eqn:v_infty_av_z}
\int_{M_\infty} v_\infty^a d\nu_{\infty,-\gamma^2} = 0,
\end{equation}
for every $a=1,\ldots,k$.
\item[D.] If $q_j\in \tilde B_s(p_j,1)\subset M_j$, $q_j\rightarrow q_\infty \in M_\infty$, and the conjugate heat flows $\nu_{(q_j,0)}$ converge to a conjugate heat flow $\nu_{\{q_j\},\infty} $, then the $(k,\theta'_j)$-splitting maps $v_j$  around $q_j$ at scale $2\gamma$ smoothly converge to the map $v_{\infty}|_{M_\infty\times (-4\gamma^2,0)}$ and for every $a=1,\ldots,k$
$$\lim_{j\rightarrow +\infty} \int_{M_j} v_j^a d\nu_{(q_j,0),-\gamma^2}=\int_{M_\infty} v_\infty^a d\nu_{\{q_j\},\infty,-\gamma^2}.$$
To see this, apply Proposition \ref{prop:s_map_compactness} to the sequence $v_j - v_j(q_j,0)$ of almost splitting maps.
\end{itemize}
By (A) and (C) it follows that, up to isometry, $M_\infty=M'\times \mathbb R^k$, $g_{\infty}(t)=g'(t)\oplus g_{\mathbb R^k}$, for any $t<0$, and $(M',g'(t))_{t\in (-\infty,0)}$ does not split any Euclidean factors. Moreover, we may assume that the functions $v_\infty^a$, $a=1,\ldots,k$, are the corresponding Euclidean coordinate functions, for $t\in (-1,0)$. 

Now, since $\nu_\infty \in \mathcal S$ and satisfies (CHF) with respect to $p_\infty=(p'_\infty,a_\infty)$, by Proposition \ref{prop:spine}, it follows that for every $(q,z)\in M'\times \mathbb R^k$ and $t<0$
\begin{align*}
f_\infty(q,z,t)&= \frac{|z-a_\infty|^2}{4|t|}+f'(q,t),
\end{align*}
where $(M', g'(t), f'(\cdot,t))$ is a normalized gradient shrinking Ricci soliton at scale $|t|$.

By \eqref{eqn:v_infty_av_z} it follows that $a_\infty=0$, since for every $a=1,\ldots,k$,
\begin{equation}\label{eqn:ave_a_infty}
\begin{aligned}
0=\int_{M_\infty} v_\infty^a d\nu_{\infty,-\gamma^2} &= (4\pi \gamma^2)^{-n/2}\int_{-\infty}^{+\infty} \int_{M'} z^a  e^{-\frac{|z-a_\infty|^2}{4\gamma^2}-f'(\cdot,-\gamma^2)} d\vol_{g'(-\gamma^2)} d z^a\\
&=(4\pi \gamma^2)^{-n/2} \int_{M'} e^{-f'(\cdot,-\gamma^2)} d\vol_{g'(-\gamma^2)}  \int_{-\infty}^{+\infty} z^a  e^{-\frac{|z-a_\infty|^2}{4\gamma^2}}  d z^a\\
&=a_\infty^a.
\end{aligned}
\end{equation}
In particular, $p_\infty=(p_\infty',0)$.

In order to prove Assertion 1, suppose that (I) holds for every $j$. Then, passing to a further subsequence we may assume that for each $i=1,2$, $q_{i,j}$ converge to $q_{i,\infty}=(q'_i,a_i)\in \tilde B_{2s}(p_\infty,1)$, $q_{1,\infty}\not = q_{2,\infty}$, and  $t_j\rightarrow t_\infty \in [-\varepsilon^{-1},-\varepsilon]$. 

By \eqref{eqn:edj} and Lemma \ref{lemma:entropy_to_ss} 
$$d_{g_\infty(t_\infty)}(q_{i,\infty},\mathcal S_{\textrm{point}}) \leq D(n,H) \sqrt{|t_\infty|},$$
hence, by Proposition \ref{prop:spine}, we know that
\begin{equation}\label{eqn:dqi'}
d_{g_\infty(t_\infty)}(q_1',q_2') \leq D' \sqrt{|t_{\infty}|}, 
\end{equation}
where $D'=D(n,H)+A(n)$.

Moreover, by \eqref{eqn:edj}, passing to a subsequence we may assume that the conjugate heat flows $\nu_{(q_{i,j},0)}$ converge to the conjugate heat flows $\nu_{i,\infty}\in\mathcal S$. Therefore, by Proposition \ref{prop:spine},  $d\nu_{i,\infty,t} = (4\pi |t|)^{-n/2} e^{-f_{i,\infty}} d\vol_{g_\infty(t)}$ satisfies
\begin{equation*}
f_{i,\infty}(q,z,t) = \frac{|z-a_i|^2}{4|t|} + f'(q,t),
\end{equation*}
for every $t<0$.

Define  
 \begin{equation}
\tilde v_{\infty}^a(q_{i,\infty}):= \int_{M_\infty} v_\infty^a d\nu_{i,\infty,-\gamma^2},
\end{equation}
and note that (D) implies that
\begin{equation*}
\tilde v^a_\infty(q_{i,\infty})=\lim_{j\rightarrow+\infty} \int_{M_j} v_j^a d\nu_{(q_{i,j},0),-\gamma^2}=\lim_{j\rightarrow +\infty} \tilde v_j^a(q_{i,j}).
\end{equation*}
Then,  as in \eqref{eqn:ave_a_infty}, we have
\begin{equation}\label{eqn:ave_a_i}
\tilde v_\infty^a(q_{i,\infty})=\int_{M_\infty} v_\infty^a d\nu_{i,\infty,-\gamma^2}=a_i^a.
\end{equation}

Now, by Pythagoras' Theorem and \eqref{eqn:dqi'}
\begin{equation}\label{eqn:pythagoras}
\begin{aligned}
|\tilde v_\infty(q_{1,\infty}) - \tilde v_\infty(q_{2,\infty})|^2 &= |a_1 - a_2|^2 \\
&\leq |a_1 - a_2|^2 + (d_{g'(t_\infty)} (q'_1,q'_2))^2\\
&=(d_{g_\infty(t_\infty)}(q_{1,\infty},q_{2,\infty}))^2\\
&\leq |\tilde v_\infty(q_{1,\infty}) - \tilde v_\infty(q_{2,\infty})|^2 +(D')^2 |t_\infty|.
\end{aligned}
\end{equation}
On the other hand, by \eqref{eqn:caseii_a} and \eqref{eqn:caseii_b} we know that one of the following must hold
\begin{align*}
\left(1-\varepsilon\right)(d_{g_\infty(t_\infty)}(q_{1,\infty},q_{2,\infty}))^2 &\geq |\tilde v_\infty(q_{1,\infty}) - \tilde v_\infty(q_{2,\infty})|^2 +E^2 |t_\infty|^2 \\
|\tilde v_\infty(q_{1,\infty}) - \tilde v_\infty(q_{2,\infty})| &\geq \left(1+\varepsilon\right) d_{g_\infty(t_\infty)}(q_{1,\infty},q_{2,\infty}),
\end{align*}
and either one contradicts \eqref{eqn:pythagoras}, if we choose $E=2D'$. This proves Assertion 1 of the lemma.

To prove Assertion 2 of the lemma, suppose that (II) holds for every $j$. Then, passing to a further subsequence we may assume that $w_j\rightarrow w\in \overline{B_{\frac{s}{2}}(0)}\subset\mathbb R^k$. 

By \eqref{eqn:p_infty_S_point}, we know that there is $(x',0)\in \mathcal S_{\textrm{point}}$ such that $d_{g'(-1)}(p_\infty',x') \leq D$. Set $\tilde x_\infty = (x',w) \in \mathcal S_{\textrm{point}}$, so that

$$\tilde x_\infty=(x',w) \in  \left(\overline{B_{g'(-1)}(p'_\infty, D)}\times \overline{B_{\frac{s}{2}}(0)} \right) \cap \mathcal S_{\textrm{point}} \subset \mathcal S_{\textrm{point}} \cap \tilde B_s \left(p_\infty,1\right),$$
if $s>4D$.

Let $\tilde x_j\in \mathcal L_{x_j,1}\cap \tilde B_s(p_j,1)\subset M_j$ be any sequence such that $\tilde x_j\rightarrow \tilde x_\infty$. In particular, $(M_j,g_j(t),\tilde x_j)_{t\in (-\infty,0)}$ is $(0,\delta_j)$-selfsimilar by Proposition 4.4 in \cite{G25}. Then, by Proposition \ref{prop:compactness_rf} and Lemma \ref{lemma:a_ss_to_ss}, passing to a subsequence we may assume that the conjugate heat flows $\nu_{(\tilde x_j,0)}$ converge to a conjugate heat flow $\nu_{w} \in \mathcal S$ which satisfies (CHF) with respect to $\tilde x_\infty=(x',w)\in\mathcal S_{\textrm{point}}$. Therefore, by Proposition \ref{prop:spine}, we know that $d\nu_{w,t} = (4\pi |t|)^{-n/2} e^{-f_w(\cdot,t)} d\vol_{g(t)}$, where
\begin{equation*}
f_w(q,z,t) = \frac{|z-w|^2}{4|t|}+f'(q,t).
\end{equation*}
Thus, as in \eqref{eqn:ave_a_i},
\begin{equation*}
\lim_{j\rightarrow +\infty} \tilde v_j^a(\tilde x_j) = \lim_{j\rightarrow +\infty} \int_{M_j} v_j^a d\nu_{(\tilde x_j,0),-\gamma^2} = \int_{M_\infty} v_\infty^a d\nu_{w,-\gamma^2} = w^a.
\end{equation*}
It follows that for large $j$
$$|\tilde v_j^a(\tilde x_j) - w_j^a| <\varepsilon/2,$$
which contradicts (II).
\end{proof}

\subsection{Existence and small scale behaviour of almost splitting maps}
In this subsection we review the main results from \cite{G25} which we will need.

First, recall the concept of entropy pinching from \cite{G25}, which is fundamental for what will follow. It is a Ricci flow analogue of the entropy pinching that first appeared in \cite{CJN}.

\begin{definition}\label{def:entropy_pinching}
Let $T>1$, $r>0$, $I\subset \mathbb R$ and interval such that $ [-2Tr^2,0]\subset I$, and $(M^n,g(t))_{t\in I}$ a smooth compact Ricci flow. Given $\{x_i\}_{i=0}^k \subset M$, we define
\begin{equation*}
\mathcal E^k_r(\{x_i\}_{i=0}^k)=\sum_{i=0}^k \mathcal W_{x_i}(r^2)-\mathcal W_{x_i}(2Tr^2).
\end{equation*}
Moreover, if $p\in M$, $\alpha \in (0,1)$, $\delta>0$, $[-2Tr^2,0] 
\subset (-2\delta^{-1}r^2,0)\subset I$ and $R>>D'$, we define
\begin{equation*}
\begin{aligned}
\mathcal E^{(k,\alpha,\delta,R)}_r(p)= &\inf\left\{ \mathcal E^k_r(\{x_i\}_{i=0}^k),\{x_i\}_{i=0}^k \subset \tilde B_R(p,r)\textrm{ is $(k,\alpha, D'r)$-independent at $t=-r^2$}, \right.\\
&\left. \textrm{and for each $i=0,\ldots,k$, either } (M,g(t),x_i)_{t\in (-2\delta^{-2}r^2,0)} \textrm{is $(0,\delta)$-selfsimilar at scale $r$, or} \right.\\
&\left.\mathcal W_{x_i}(\delta r^2)-\mathcal W_{x_i}(\delta^{-1} r^2)<\delta.\right\}. 
\end{aligned}
\end{equation*}
We call $\mathcal E^{(k,\alpha,\delta,R)}_r(p)$ the $(k,\alpha,\delta,R)$-entropy pinching around $p$ at scale $r$.
\end{definition}

The following theorem from \cite{G25}, asserts that $T$ and $D'$ can be chosen so that the $(k,\alpha,\delta,R)$-entropy pinching  at each scale $r$ controls the Hessian of an almost splitting map at scale $r$.

\begin{theorem}[Theorem 8.1 in \cite{G25}]\label{thm:sharp_splitting}
Fix $\varepsilon>0$, $0<\alpha\leq \frac{1}{6}$. Let $(M,g(t),p)_{t\in(-2\delta^{-2},0]}$ is a smooth compact Ricci flow satisfying (RF1-4). Suppose that there is $q\in B(p,-1,R)$ such that $(M,g(t),q)_{t\in(-2\delta^{-2},0)}$ is $(k,\delta^2)$-selfsimilar around $p$ at scale $1$, and if $q\not = p$ suppose that $\mathcal W_p(\delta)-\mathcal W_p(\delta^{-1})<\delta$.

There are $D'(n,H)<+\infty$ and $T(n,H)<+\infty$ such that if in Definition \ref{def:entropy_pinching} we set $D'=D'(n,H)$, $T=T(n,H)$, $R\geq \frac{D'}{\alpha}$, and if $0<\delta\leq \delta(n,C_I,\Lambda,H,K|R,\alpha,\varepsilon)$, then there is a $(k,\varepsilon)$-splitting map $v:M\times [-1,0]\rightarrow \mathbb R^k$, $v=(v^1,\ldots,v^k)$, at scale $1$ around $p$ such that
\begin{equation*}
\int_M |\hess  v^a|^2 d\nu_{(p,0),t} \leq C(n,C_I,\Lambda,H|R,\alpha,\varepsilon) \mathcal E^{(k,\alpha,\delta,R)}_1(p),
\end{equation*}
for every $a=1,\ldots,k$ and $t\in [-1,-\varepsilon]$.
\end{theorem}

From now on, the constants $D'$ and $T$ in Definition \ref{def:entropy_pinching} are set to be the constants specified by Theorem \ref{thm:sharp_splitting}.

Given a $(k,\delta)$-splitting map $v$ around $p$ at scale $1$ it is not necessarily true that $v|_{M\times [-r^2,0]}$ is going to be a $(k,\varepsilon)$-splitting map around $p$ at an arbitrarily small scale $r\in (0,1)$, even if $\delta>0$ is small enough. However, according to the following theorem, if the Ricci flow remains sufficiently almost selfsimilar down to a small scale, then a linear transformation of $v$ is  a $(k,\varepsilon)$-splitting map at lower scales.

\begin{proposition}[Transformations, Proposition 9.1 in \cite{G25}]\label{prop:transformations}
Fix $\eta>0$, $\varepsilon>0$, $1\leq k\leq n$, $r\in (0,1)$ and $R<+\infty$, and 
let $(M,g(t),p)_{t\in (-2\delta^{-2},0]}$  be pointed smooth complete Ricci flow satisfying (RF1-3). Suppose that for every $s\in [r,1]$ there is $q_s\in \tilde B_R(p,s)$ such that $(M,g(t),q_s)_{t\in (-2\delta^{-2},0)}$  is $(k,\delta^2)$-selfsimilar but not $(k+1,\eta)$-selfsimilar, at scale $s\in [r,1]$ around $q_s$. Moreover, if $q_s\not = p$ suppose that $$\mathcal W_p(\delta s^2) -\mathcal W_p(\delta^{-1} s^2)<\delta.$$

Let $v$ be a $(k,\delta)$-splitting map around $p$ at scale $1$. If  $0<\delta \leq \delta(n,C_I,\Lambda,H,\eta|R,\varepsilon)$ then for each $s\in [r,1]$ there is a lower triangular $k\times k$ matrix $T_s$ such that
\begin{enumerate}
\item $v_s:=T_s v$, with $v_s^a= (T_s)^a_b v^b$ is a $(k,\varepsilon)$-splitting map at scale $s$.
\item For each $a,b=1,\ldots,k$,
 \begin{equation}\label{eqn:vr_normalization}
\frac{4}{3s^2}\int_{-s^2}^{-s^2/4}\int_M \langle \nabla v_s^a,\nabla v_s^b\rangle  d\nu_{(p,0),t} dt = \delta^{ab}.
\end{equation}
\item  $|| T_s\circ T_{2s}^{-1} - I_k||<C(n)\sqrt\varepsilon$ and whenever $r\leq s_1\leq s_2 \leq \frac{1}{2}$, we have 
\begin{equation}
||T_{s_1}\circ T_{s_2}^{-1}|| \leq \left(\frac{s_2}{s_1}\right)^{C(n)\sqrt{\varepsilon}} \quad \textrm{and}\quad ||T_{s}||\leq (1+\varepsilon) s^{-C(n)\sqrt\varepsilon},
\end{equation}
where $||\cdot||$ denotes the maximum norm of a $k\times k$ matrix.
\end{enumerate}
\end{proposition}

\begin{remark}\label{rmk:no_shrink}
Given $(M,g(t),p)_{t\in (-2\delta^{-2},0]}$ and $v: M\times [-1,0]\rightarrow \mathbb R^k$ a $(k,\delta)$-splitting map, as in Proposition \ref{prop:transformations}, for each $s\in [r,1]$ we may define an inner product $(\cdot,\cdot)_s$ on $\mathbb R^k$ by
$$(e_i,e_j)_s = \frac{4}{3s^2} \int_{-s^2}^{-s^2/4} \int_M \langle \nabla v^i,\nabla v^j\rangle d\nu_{(p,0),t} dt,$$
where $e_1,\ldots,e_k$ denote the standard basis of $\mathbb R^k$. In particular, denoting by $\cdot$ the standard inner product on $\mathbb R^k$, for any $\xi=\xi^a e_a\in\mathbb R^k$
$$(\xi,\xi)_s = \frac{4 }{3s^2} \int_{-s^2}^{-s^2/4} \int_M \langle \nabla (\xi\cdot v),\nabla (\xi \cdot v)\rangle d\nu_{(p,0),t} dt = \frac{4 }{3s^2} \int_{-s^2}^{-s^2/4} \int_M | \nabla (\xi \cdot v)|^2  d\nu_{(p,0),t} dt.$$
Since $(\frac{\partial}{\partial t} -\Delta) (\xi \cdot v)=0$, the Bochner formula gives that $(\frac{\partial}{\partial t} -\Delta)|\nabla (\xi \cdot v)|^2 \leq 0$. Therefore,
\begin{equation*}
(\xi,\xi)_s \leq \int_M |\nabla (\xi \cdot v)|^2 d\nu_{(p,0),-s^2} \leq \int_M |\nabla (\xi \cdot v)|^2 d\nu_{(p,0),-1/4} \leq (\xi,\xi)_1.
\end{equation*}

Thus, in the setting of Proposition \ref{prop:transformations}, we obtain that if $\delta$ is even smaller (depending only on $\varepsilon$) then for every $s\in [r,1]$
\begin{equation*}
\begin{aligned}
|\xi|^2&=\xi^a\xi^b \delta_{ab} = \xi^a\xi^b (T_s)^c_a (T_s)^d_b (e_c,e_d)_s, \\
&=( T_s(\xi), T_s(\xi))_s \leq (T_s(\xi),T_s(\xi))_1 \leq (1+\varepsilon) |T_s(\xi)|^2,
\end{aligned}
\end{equation*}
since
\begin{equation*}
\begin{aligned}
|(e_i,e_j)_1 - \delta_{ij}|&= \left| \frac{4}{3} \int_{-1}^{-1/4} \int_M ( \langle\nabla v^i,\nabla v^j\rangle -\delta_{ij})d\nu_{(p,0),t} dt  \right|,\\
&\leq \frac{4}{3} \int_{-1}^{-1/4} \int_M | \langle\nabla v^i,\nabla v^j\rangle -\delta_{ij} | d\nu_{(p,0),t} dt, \\
&\leq \left( \frac{4}{3} \int_{-1}^{-1/4} \int_M     | \langle\nabla v^i,\nabla v^j\rangle -\delta_{ij} |^{2} d\nu_{(p,0),t} dt \right)^{1/2},\\
&\leq \frac{2}{\sqrt 3}\delta^{1/2}.
\end{aligned}
\end{equation*}

\end{remark}

Finally, the following theorem asserts that if the sum of the entropy pinching over all scales is small then an almost splitting map at a large scale, remains an almost splitting map even at smaller scales.

\begin{theorem}[Non-degeneration, Theorem 10.1 in \cite{G25}]\label{thm:non_degen}
Fix $\varepsilon>0, \eta>0$, $0<\alpha\leq \frac{1}{6}$, $R\geq \frac{D'}{\alpha}$,  $1\leq k \leq n$, and let $(M,g(t),p)_{t\in (-2\delta^{-2},0]}$ be a pointed smooth compact Ricci flow satisfying (RF1-4). Let $r\in (0,1)$ and suppose that for every $r\leq s\leq 1$, there is $q_s \in \tilde B_R(p,s)$ such that  $(M,g(t),q_s)_{t\in (-2\delta^{-2},0)}$ is $(k,\delta^2)$-selfsimilar but not $(k+1,\eta)$-selfsimilar  at scale $s$, and if $q_s \not = p$ suppose that $\mathcal W_p(\delta s^2)- \mathcal W_p(\delta^{-1} s^2)<\delta$.

Moreover, let  $v$ be a $(k,\delta)$-splitting map $v$ around $p$ at scale $1$, and suppose that for $s_j=2^{-j}$ we have
\begin{equation}
\sum_{r \leq s_j\leq 1} \mathcal E^{(k,\alpha,\delta,R)}_{s_j}(p) <\delta.
\end{equation}
If $0<\delta\leq\delta(n,C_I,\Lambda,H,K,\eta|R,\alpha,\varepsilon)$, then for every $r\leq s\leq 1$, $v: M\times [-s^2,0]\rightarrow \mathbb R^k$ is a $(k,\varepsilon)$-splitting map around $p$ at scale $s$.
\end{theorem}

We finish this section with a simple application of Theorem \ref{thm:sharp_splitting}, which in combination with Lemma \ref{lemma:split_map_selfsimillar} allows for the construction of open covers by a uniformly controlled number of balls.

\begin{corollary}\label{cor:coverings}
Fix $\tau\in (0,1]$, $0<\zeta \leq 2^{-5}$, let $(M,g(t),p)_{t\in (-2\delta^{-2},0)}$ be a smooth compact Ricci flow satisfying (RF1-4) such that $\mathcal W_p(\delta)-\mathcal W_p (\delta^{-1})<\delta$, and let $R\geq R(n,H,K|\tau,\zeta)$. Suppose that there is $q\in \tilde B_{2R}(p,1)$ such that $(M,g(t),q)_{t\in (-2\delta^{-2},0)}$ is $(k,\delta^2)$-selfsimilar  but not $(k+1,\eta)$-selfsimilar at scale $1$. Let 
$$W \subset \{x\in \tilde B_{\frac{5R}{3}}(p,1), \mathcal W_x(\delta)-\mathcal W_x (\delta^{-1})<\delta\}.$$
Then, if $0<\delta\leq \delta(n,C_I,\Lambda,H,\eta|\zeta,R)$, there exists a maximal subset $\{x_i\}_{i=1}^N \subset  W$ such that the balls $\tilde B_{\frac{1}{2}\tau^2 R}(x_i,\zeta)$ are mutually disjoint, $\tilde B_{2R}(x_i,\zeta)\subset\tilde B_{2R}(p,1)$,
$$W\subset \bigcup_{i=1}^N \tilde B_R (x_i,\zeta)$$
and $N\leq C(n,\tau) \zeta^{-k}$. 

Moreover, if $W \subset \tilde B_R(p,1)$ does not contain any $(k,\alpha,D')$-independent subsets in $\tilde B_{R}(p,1)$ at time $t=-1$, and if $0<\alpha\leq\alpha(n,\zeta)$, $R\geq R(n,H,K|\tau,\zeta,\alpha)$ and $0<\delta\leq \delta(n,C_I,\Lambda,H,\eta|\alpha, \zeta,R)$, then $N\leq C(n,\tau) \zeta^{-k+1}$.
\end{corollary}

\begin{proof}
Applying Theorem \ref{thm:sharp_splitting} and Lemma \ref{lemma:split_map_selfsimillar}, we obtain that for any $0<\varepsilon\leq \frac{1}{2}$, if $0<\delta\leq \delta(n,C_I,\Lambda,H,K,\eta|\varepsilon,\zeta)$, then there is a constant $E=E(n,H)<+\infty$ and a function $\tilde v: M\rightarrow \mathbb R$ satisfying $\tilde v(p)=0$ and
\begin{equation}\label{eqn:lip}
\begin{aligned}
\frac{1}{2}( d_{g(t)}(q_1,q_2))^2 &\leq \frac{1}{1+\varepsilon} ( d_{g(t)}(q_1,q_2))^2 \leq |\tilde v(q_1)-\tilde v(q_2)|^2 + E^2 |t|, \\
|\tilde v(q_1)-\tilde v(q_2)|&\leq (1+\varepsilon) d_{g(t)} (q_1,q_2) \leq \frac{3}{2} d_{g(t)}(q_1,q_2),
\end{aligned}
\end{equation}
for any $q_1,q_2 \in W$ with $d_{g(-1)}(q_1,q_2)\geq \zeta$ and $t\in [-1,-\zeta^2]$.

Choose a maximal collection of points $\{x_i\}_{i=1}^N \subset W$, with $p=x_1$, such that the balls $\tilde B_{\frac{1}{4}\tau^2 R}(x_i,\zeta)$ are disjoint and $\tilde B_R (x_i,\zeta)$ cover $W$. Note that, by Lemma \ref{lemma:tilde_balls_inclusion}, if $R\geq 12K$, we know that for each $i=1,\ldots,N$, we have the inclusion $\tilde B_{2R}(x_i,\zeta)\subset \tilde B_{2R}(p,1)$. It remains to estimate $N$.

Since the balls $\tilde B_{\frac{1}{4}\tau^2 R}(x_i,\zeta)$ are disjoint we know that $d_{g(-\zeta^2)}(x_i,x_j) \geq \frac{1}{2}\tau^2 R \zeta$, for $i\not = j$, hence by (RF4) 
$$d_{g(-1)}(x_i,x_j) \geq d_{g(-\zeta^2)}(x_i,x_j) - K\geq \frac{1}{2} \tau^2 R \zeta -K\geq \zeta,$$
if $\tau^2 R\zeta\geq 4K$.

Now, applying  \eqref{eqn:lip},   we obtain that for every $i=2,\ldots,N$
$$|\tilde v(x_i)| \leq \frac{3}{2} d_{g(-1)}(p,x_i)\leq 3R,$$
since $\tilde v(p)=0$.

On the other hand, \eqref{eqn:lip} gives, for  $i\not = j$,
$$\frac{1}{8} \tau^4 \zeta^2 R^2\leq \frac{1}{2} (d_{g(-\zeta^2)}(x_i,x_j))^2 \leq |\tilde v(x_i)-\tilde v(x_j)|^2 +E^2 \zeta^2,$$
which implies that $|\tilde v(x_i)-\tilde v(x_j)|\geq \frac{1}{4} \tau^2  R \zeta$ if $R\geq R(n,H,\tau)$.
It follows that the balls $B_i=B_{\frac{1}{8} \tau^2 R \zeta}(\tilde v(x_i))\subset B_{4R}(0)\subset \mathbb R^k$ are disjoint, so $N\leq C(n,\tau) \zeta^{-k}$.

If $W\subset \tilde B_R(p,1)$ does not contain any $(k,\alpha,D')$-independent subsets in $\tilde B_R(p,1)$ the improved estimate on $N$ will follow once we show that 
\begin{equation}\label{eqn:V}
V=\{\tilde v(x_i)\}_{i=1}^N \subset B_{3R}(0)\cap B_\zeta (L),
\end{equation}
 where $L$ is $k-1$ dimensional affine subspace of $\mathbb R^k$.  
 
 To prove \eqref{eqn:V}, by Remark \ref{rmk:k_ind} it suffices to show that $V$ does not contain any $(k,\alpha)$-independent subsets in $B_{5R}(0)$, for some $0<\alpha\leq \alpha(n,\zeta)\leq \frac{1}{2}$.

Take any $k+1$ subset $\{\tilde v(x_{i_a})\}_{a=0}^k$ of $V$.  Since $W$ does not contain any $(k,\alpha,D')$-independent subsets in $\tilde B_R(p,1)$ at time $t=-1$, we know that there exists a metric space $(\mathcal K,d_{\mathcal K})$ with $\diam (\mathcal K)\leq D'$, and  points $z_j =(f_j,a_j)\in \mathcal K\times \mathbb R^{k-1}$, such that, denoting by $d$ the product metric on $\mathcal K\times \mathbb R^{k-1}$,
\begin{equation}\label{eqn:zx_ind}
|d(z_a,z_b) - d_{g(-1)}(x_{i_a},x_{i_b})|<D' + \alpha R.
\end{equation}
It follows that, if $R\geq R(D')=R(n,H)$,
\begin{equation}\label{eqn:av_bounds}
\begin{aligned}
|a_a - a_b| \leq d(z_a,z_b) \leq d_{g(-1)} (x_{i_a},x_{i_b})+D'+\alpha R\leq 3R,\\
|\tilde v(x_{i_a})-\tilde v(x_{i_b}) | \leq  (1+\varepsilon) d_{g(-1)}(x_{i_a}, x_{i_b}) \leq 3R
\end{aligned}
\end{equation}
since $W\subset \tilde B_R(p,1)$, by \eqref{eqn:zx_ind} and \eqref{eqn:lip} respectively.

Then, choosing $\varepsilon=\frac{\alpha}{2}$ and using \eqref{eqn:lip}, \eqref{eqn:zx_ind} and \eqref{eqn:av_bounds}, we have
\begin{align*}
 | \tilde v(x_{i_a}) - \tilde v(x_{i_b})| &\leq (1+\alpha) d_{g(-1)} (x_{i_a}, x_{i_b})\\
 &<(1+\alpha) \left( d(z_a,z_b) + D' +\alpha R\right)\\
 &\leq (1+\alpha) \left( \sqrt{d(f_a,f_b)^2 + |a_a- a_b|^2} + D' +\alpha R \right),\\
 &\leq (1+\alpha) \left( \sqrt{(D')^2 + |a_a-a_b|^2} + D' +\alpha R\right),\\
 &\leq (1+\alpha) \left( |a_a-a_b|+2 D' +\alpha R \right),\\
 &\leq |a_a - a_b| + 3 D' +\frac{9}{2} \alpha R,\\
 &\leq |a_a-a_b| + 5 \alpha R,
\end{align*}
if $3D'\leq \frac{1}{2}\alpha R$. Similarly,
\begin{align*}
|a_a-a_b| &\leq d(z_a,z_b) \\
&< d_{g(-1)} (x_{i_a},x_{i_b}) + D' +\alpha R\\
&\leq \left(1+\frac{\alpha}{2}\right) \sqrt{ |\tilde v(x_{i_a}) - \tilde v(x_{i_b})|^2 +E^2} + D' +\alpha R,\\
&\leq \left(1+\frac{\alpha}{2}\right) |\tilde v(x_{i_a}) - \tilde v(x_{i_b})| + 2 E + D' +\alpha R\\
&\leq |\tilde v(x_{i_a}) - \tilde v(x_{i_b})|  + \frac{3}{2} \alpha R+2E + D' +\alpha R,\\
&\leq |\tilde v(x_{i_a}) - \tilde v(x_{i_b})| + 5\alpha R,
\end{align*}
if $2E+D'\leq \alpha R$. It follows that  $V$ does not contain any $(k,\alpha)$-independent subsets in $B_{5R}(0)$, which proves the result.
\end{proof}

\section{Neck regions}

\subsection{$(k,\delta,\eta)$-neck regions}

\begin{definition}\label{def:neck_region}
Fix $R<+\infty$, $s>0$ and a non-negative integer $0\leq k\leq n-2$. Let $(M,g(t))_{t\in (-\delta^{-3}s^2,0)}$ be a smooth compact Ricci flow, where $n=\dim M$. Let $C\subset \tilde B_{2R}(p,s)$ be a non-empty closed subset and  $r_{\cdot}=r(\cdot)$ be a continuous function $r: C\rightarrow [0,+\infty)$. We will say that
 $$\mathcal N=\tilde B_{2R}(p,s)\setminus \bigcup_{x\in C}\overline{ \tilde B_R(x,r_x)}$$   is a $(k,\delta,\eta)$-neck region at scale $s$, with set of centers $C$ and radius function $r_x$, 
if the following hold: there is $0<\tau\leq \frac{1}{100}$ such that
\begin{itemize}
\item[(n1)] For all $x\in C$ the balls $\tilde B_{\tau^2 R}(x,r_x)$ are mutually disjoint.
\item[(n2)] For all $x\in C$, $\mathcal W_x(\delta r_x^2) - \mathcal W_x(\delta^{-1}s^2) <\delta$.
\item[(n3)] For each $x\in C$ and $r\in [ r_x,\delta^{-1}s^2]$,  there is an $x_r\in \tilde B_{\tau^2 R}(x,r)$ such that $(M,g(t),x_r)_{(-2\delta^{-2}s^2,0]}$ is, at scale $r$, $(k,\delta^2)$-selfsimilar with respect to $\mathcal L_{x_r,r}$, but not $(k+1,\eta)$-selfsimilar.
\item[(n4)] For each $x\in C$ and $r\in [r_x, s]$ such that $\tilde B_{2R}(x, r)\subset \tilde B_{2R}(p,s)$,  we have the inclusion
\begin{align*}
\mathcal L_{x_r,r} \cap \tilde B_R(x,r) &\subset \tilde B_{\tau R}(C,r).
\end{align*}
\item[(n5)] For each $x,y\in C$, and 
$r_y\leq 1.01  r_x + \delta D_R(x,y)$. In other words, $r_x$ is $R$-scale $(1.01,\delta)$-Lipschitz.
\end{itemize}
We will call $s^{-1}\inf_{x\in C} r_x$ the depth of the neck-region. Moreover, a $(k,\delta,\eta)$-neck region will be called smooth, if $(M,g(t))_{t\in(-2\delta^{-3}s^2,0]}$ is smooth. 
\end{definition}

\begin{remark}\label{rmk:infrx}
If (RF4) holds, $\tau^2 R\geq K$ and a neck region is smooth then $\inf_{x\in C} r_x >0$ and $C$ is finite. 

To see this, first note that if $\inf_{x\in C} r_x =0$,  there are sequences $x_j\in C$ and $r_j\rightarrow 0$ with $r_j>0$ with $r_j\geq r_{x_j}$. Consider the sequence of pointed Ricci flows $(M,r_{j}^{-2}g(r_{j}^2 t),x_j)$. Since the Ricci flow of the neck region is smooth up to $t=0$, we obtain that  $(M,r_{j}^{-2}g(r_{j}^2 t),x_j)$ converges smoothly to the pointed static Ricci flow $(\mathbb R^n,g_{\mathbb R^n},0)$. It follows that for large $j$, $(M,g(t),x_j)$ is $(n,\eta)$-selfsimilar at scale $r_j$, which is a contradiction.

Now, since $\bar r=\inf_{x\in C} r_x>0$ if we take a sequence of different points $x_j\in C$, we know that, by the definition of the neck region and Lemma \ref{lemma:balls_inclusion}, the balls $\tilde B_{\tau^2 R}(x_j,\bar r)\subset \tilde B_{\tau^2 R}(x_j,r_{x_j}) $ are pairwise disjoint. However, by the compactness of $M$, passing to a subsequence we may assume that the sequence $x_j$ is Cauchy with respect to the metric $d_{g(-\bar r^2)}$, which leads to a contradiction.

\end{remark}

\begin{remark}\label{rmk:neck_inside_neck}
Let $\mathcal N = \tilde B_{2R}(p,1) \setminus \bigcup_{x\in C} \overline{\tilde B_R(x,r_x)}$ be a $(k,\delta,\eta)$-neck region at scale $1$ of depth $\bar r$ in the Ricci flow $(M,g(t))_{t\in (-2\delta^{-3},0]}$, and let $s>0$ and $q\in C$ such that $\tilde B_{2R}(q,s)\subset \tilde B_{2R}(p,1)$. Define $C_{q,s}=C\cap \tilde B_{2R}(q,s)$.  Then $$\mathcal N_{q,s}=\tilde B_{2R}(q,s)\setminus \bigcup_{x\in C_{q,s}} \overline{\tilde B_R(x,r_x)}$$
a $(k,\delta,\eta)$-neck region at scale $s$ of depth $\bar r s^{-1}$. To see this, it suffices to check that $\mathcal N_{q,s}$ satisfies (n4), since the remaining properties (n1-n3) and (n5) are evident.  In particular, we just need to show that if for some $x \in C$ and $r_x\leq r\leq 1$ we know that $\tilde B_{2R}(x,r)\subset \tilde B_{2R}(p,1)$ and there is $z\in C$ and $w\in \tilde B_{\tau R}(z,r)\cap \tilde B_R(x,r)$, then $z\in C\cap\tilde B_{\frac{3R}{2}}(x,r)$. This, however, follows by the triangle inequality at time $t=-r^2$, since
$$d_{g(-r^2)}(x,z) \leq d_{g(-r^2)}(x,w)+d_{g(-r^2)}(w,z) \leq (R+\tau R)r\leq \frac{3R}{2}r.$$
\end{remark}

\begin{remark}\label{rmk:sigma}
Let $x,y\in \tilde B_{2R}(p,1)$. By the triangle inequality we know that
\begin{equation*}
d_{g(-1)}(x,y)\leq d_{g(-1)}(p,x) + d_{g(-1)}(p,y)< 4R.
\end{equation*}
It follows that $D_{4R}(x,y) < 1$.

On the other hand, since by Defintion \ref{def:neck_region} we know that for any $x,y\in C\subset \tilde B_{2R}(p,1)$, $x\not = y$
$$\tilde B_{\tau^2 R}(x,r_x) \cap \tilde B_{\tau^2 R}(y,r_y)=\emptyset,$$
it follows that 
\begin{equation*}
d_{g(-r_x^2)}(x,y) \geq \tau^2 R r_x\quad\textrm{and}\quad
d_{g(-r_y^2)}(x,y) \geq \tau^2 R r_y,
\end{equation*}
hence
$$D_{\tau^2 R}(x,y)\geq \max(r_x,r_y).$$
\end{remark}
Thus, by continuity, there is $\sigma=\sigma_{x,y}\in [\tau^2 R,4R]$ such that
$$\max(r_x,r_y)\leq D_\sigma(x,y)<1.$$

\begin{definition}
Define the packing measure $\mu$ of a smooth $(k,\delta,\eta)$-neck region by
\begin{equation*}
\mu= \sum_{x\in C} (R r_x)^k \delta_x.
\end{equation*}
\end{definition}

The main result of the section is the following theorem.

\begin{theorem}[smooth neck structure theorem] \label{thm:neck_structure}
Let $(M,g(t),p)_{t\in (-2\delta^{-3},0]}$ be a smooth Ricci flow satisfying (RF1-4). If $R\geq R(n,C_I,\Lambda,H,K|\tau)$, $0<\delta\leq\delta(n,C_I,\Lambda,H,K,\eta|R,\tau)$ and
\begin{equation*}
\mathcal N=\tilde B_{2R}(p,1) \setminus \bigcup_{x\in C} \overline{\tilde B_R(x,r_x)}
\end{equation*}
is a smooth $(k,\delta,\eta)$-neck region  at scale $1$, then for every $x\in C$ and $r\geq r_x$ with $\tilde B_{2R}(x,r)\subset \tilde B_{2R}(p,1)$ the estimate
\begin{equation*}
L^{-1} R^k r^k \leq \mu (\tilde B_R(x,r))\leq L R^k r^k
\end{equation*}
holds, for some $L=L(n,\tau)<+\infty$.
\end{theorem}

\subsection{Neck regions under a priori bounds} 
Suppose that $(M,g(t))_{t\in (-2\delta^{-3},0]}$ is a compact Ricci flow satisfying (RF1-4). We will say that a $(k,\delta,\eta)$-neck region in $\tilde B_{2R}(p,1)$ with set of centers $C$ and radius function $r_x$ satisfies the a priori assumptions if
\begin{enumerate}
\item $p\in C$ and $r_x\leq 2^{-5}$ for every $x\in C$.
\item There is map $v: M\times [-1,0]\rightarrow \mathbb R^k$ with $v(p,0)=0$, which is a $(k,\delta')$-splitting map around any $x\in C$ at scale $1$. 
\item There is a $B<+\infty$  with the following significance: for every $x\in C$ and $r\in [r_x,1]$ such that $\tilde B_{2R}\left(x,r\right)\subset \tilde B_{2R}(p,1)$ we have
\begin{equation}\label{weak_ahlfors}
B^{-1}  (R r)^k \leq \mu\left(\tilde B_{R}\left(x,r\right)\right)\leq B (R r)^k.
\end{equation}
\end{enumerate}

Given a $(k,\delta,\eta)$-neck region satisfying the a priori assumptions, we define $\tilde v:C \rightarrow \mathbb R^k$ by $\tilde v(x)=v(x,0)$.

Moreover, by Proposition \ref{prop:transformations} and Remark \ref{rmk:no_shrink}, given any $\theta>0$, we know that if $0<\max(\delta,\delta')\leq \delta(n,C_I,\Lambda,H,\eta|\theta)$ then for every $x\in C$ and $r\in [r_x,1]$ there is a lower triangular $k\times k$ matrix $T_{x,r}$ such that $T_{x,r} v$ is a $(k,\theta)$-splitting map at scale $r$ and
\begin{equation}\label{eqn:Tv_normalization}
\frac{4}{3r^2} \int_{-r^2}^{-r^2/4} \int_M \langle \nabla (T_{x,r} v)^a,\nabla (T_{x,r} v)^b\rangle d\nu_{(x,0),t} dt =\delta^{ab},
\end{equation}
for every $a,b=1,\ldots,k$. The maps $T_{x,r}$ also satisfy the growth estimates
\begin{equation}\label{eqn:T_ratio}
|| T_{x,r_1} T_{x,r_2}^{-1} || \leq \left(\frac{r_2}{r_1} \right)^{C(n) \sqrt\theta},
\end{equation}
for any $r_x\leq r_1\leq r_2\leq \frac{1}{2}$ and
\begin{equation}\label{eqn:T_estimates}
||T_{x,r} ||\leq (1+\theta) r^{-C(n) \sqrt\theta} \quad \textrm{and}\quad
|\xi|^2 \leq (1+\theta) |T_{x,r} (\xi)|^2,
\end{equation}
for every $r_x\leq r\leq 1$ and $\xi\in\mathbb R^k$.

\begin{lemma}\label{lemma:apriori_Ahlfors_larger_balls}
If $(M,g(t),p)_{t\in (-2\delta^{-3},0]}$ is a smooth compact Ricci flow satisfying (RF1-4) and
$$\mathcal N = \tilde B_{2R}(p,1) \setminus \bigcup_{x\in C} \overline{\tilde B_R(x,r_x)}$$ 
is a $(k,\delta,\eta)$-neck region at scale $1$ that satisfies the a priori assumptions with
\begin{align*}
R&\geq R(n,H,K),\\
0<\delta&\leq \delta(n,C_I,\Lambda,H,\eta|R),
\end{align*}
then there is  $\hat B=\hat B(n,B)<+\infty$ such that
\begin{equation}\label{eqn:5R3_ahlfors}
\hat B^{-1}R^k\leq \mu(\tilde B_{\frac{5R}{3}}(p,1))\leq \hat BR^k.
\end{equation}
\end{lemma}

\begin{proof}
Set $\zeta=2^{-5}$. By Corollary \ref{cor:coverings}, if $R\geq R(n,H,K)$ and $0<\delta\leq \delta(n,C_I,\Lambda,H,K,\eta|R)$, we can cover $\tilde B_{\frac{5R}{3}}(p,1)$ by a maximal collection of balls $\{\tilde B_R (x_i,\zeta)\}_{i=1}^N$, with $x_i\in \tilde B_{\frac{5R}{3}}(p,1)\cap C$, such that, for each $i$, $\tilde B_{2R}(x_i,\zeta)\subset \tilde B_{2R}(p,1)$, the balls $\{\tilde B_{\frac{1}{2} R\zeta}(x_i,\zeta)\}_{i=1}^N$ are pairwise disjoint, and $N\leq C(n) \zeta^{-k}$.

Since by the a priori assumptions $r_x \leq 2^{-5}$, the a priori upper Ahlfors bound gives for every $i=1,\ldots,N$,
\begin{equation*}
\mu(\tilde B_R(x_i,\zeta))\leq B (R \zeta)^k,
\end{equation*}
therefore
\begin{equation*}\label{eqn:5R3_upper}
\mu(\tilde B_{\frac{5R}{3}}(p,1))\leq \sum_{i=1}^N \mu(\tilde B_R(x_i,\zeta))\leq C(n)  B R^k.
\end{equation*}
On the other hand, since $p\in C$, the a priori lower Ahlfors bound gives
\begin{equation*}\label{eqn:5R3_lower}
B^{-1} R\leq \mu(\tilde B_R(p,1)) \leq \mu(\tilde B_{\frac{5R}{3}}(p,1)).
\end{equation*}
These suffice to prove the result.
\end{proof}

\begin{remark}\label{rmk:instead_p_C}
In the proof of Lemma \ref{lemma:apriori_Ahlfors_larger_balls} the assumption $p\in C$ for the neck-region is only used to ensure that 
\begin{equation}\label{eqn:instead_p_C}
\mathcal W_p(\delta)-\mathcal W_p(\delta^{-1})<\delta
\end{equation}
 which is needed to apply Corollary \ref{cor:coverings}. Thus, \eqref{eqn:5R3_ahlfors} still holds if we replace the assumption $p\in C$ with \eqref{eqn:instead_p_C}. 
 
Moreover, note that we didn't make any explicit use of the a priori assumption (2), although in the proof of Corollary \ref{cor:coverings} we do need to construct an almost splitting map.
\end{remark}

Recall that $T=T(n,H)<+\infty$ and $D'=D'(n,H)<+\infty$ are determined  by Theorem \ref{thm:sharp_splitting} and are involved in Definition \ref{def:entropy_pinching} of $k$-entropy pinching. Set
\begin{equation}\label{eqn:short}
E_y(r_i)=\mathcal W_y(r_i^2)-\mathcal W_y(2 T r_i^2).
\end{equation}

\begin{lemma}\label{lemma:small_average}
For every $\varepsilon>0$, if $(M,g(t),p)_{t\in (-2\delta^{-3},0]}$ is a smooth compact Ricci flow satisfying (RF1-4) and
$$\mathcal N = \tilde B_{2R}(p,1) \setminus \bigcup_{x\in C} \overline{\tilde B_R(x,r_x)}$$ 
is a $(k,\delta,\eta)$-neck region that satisfies the a priori assumptions with $R\geq R(n,H,K)$ and $0<\delta\leq\delta(n,C_I,\Lambda,H,\eta,B|R,\varepsilon)$ then for  $r_i=2^{-i}$ we have that
\begin{equation*}
\frac{1}{\mu(\tilde B_{\frac{3R}{2}}(p,1))}\int_{\tilde B_{\frac{3R}{2}}(p,1)} \left( \sum_{ r_x \leq r_i \leq 2^{-5}}\frac{1}{\mu(\tilde B_R(x,r_i))}\int_{\tilde B_R(x,r_i)} E_y(r_i)d\mu (y)\right) d\mu(x) < \varepsilon.
\end{equation*}
\end{lemma}

\begin{proof}
By Lemma \ref{lemma:tilde_balls_inclusion}, if $R\geq 12K$ then for any $x\in C\cap\tilde B_{\frac{3R}{2}}(p,1)$ we have that 
\begin{equation*}
\tilde B_{2R}(x,r_i)\subset \tilde B_{\frac{5R}{3}}(p,1)\subset \tilde B_{2R}(p,1)
\end{equation*}
 for every $i\geq 5$. Therefore, if $r_x\leq r_i\leq 2^{-5}$, the a priori Ahlfors bound gives
\begin{equation*}
B^{-1}(R r_i)^k\leq \mu(\tilde B_R(x,r_i)).
\end{equation*}
This allows us to estimate
\begin{equation*}
\begin{aligned}
&\int_{\tilde B_{ \frac{3R}{2} }(p,1)} \left( \sum_{ r_x \leq r_i \leq 2^{-5}}\frac{1}{\mu(\tilde B_R(x,r_i))}\int_{\tilde B_R(x,r_i)} E_y(r_i)d\mu (y)\right) d\mu(x)\\
&\leq \frac{B}{R^k}\int_{\tilde B_{\frac{3R}{2}}(p,1)}\left( \sum_{ r_x\leq r_i \leq 2^{-5}} r_i^{-k}\int_{\tilde B_{\frac{5R}{3}}(p,1)} \chi_{\{D_R(x,y) \leq r_i\}}E_y(r_i) d\mu(y) \right) d\mu(x)\\
&\leq \frac{B}{ R^k}\int_{\tilde B_{\frac{5R}{3}}(p,1)} \int_{\tilde B_{\frac{5R}{3}}(p,1)} \sum_{ r_x\leq r_i \leq 2^{-5}} r_i^{-k}\chi_{\{D_R(x,y) \leq r_i\}}E_y(r_i) d\mu(y)  d\mu(x).
\end{aligned}
\end{equation*}

Therefore, applying Fubini we obtain
\begin{equation}\label{eqn:characteristic_Fubini}
\begin{aligned}
&\int_{\tilde B_{\frac{3R}{2}}(p,1)} \left( \sum_{ r_x\leq r_i \leq 2^{-5}}\frac{1}{\mu(\tilde B_R(x,r_i))}\int_{\tilde B_R(x,r_i)} E_y(r_i)d\mu (y)\right) d\mu(x) \leq\\
&\leq \frac{B}{R^k}\int_{\tilde B_{\frac{5R}{3}}(p,1)} \int_{\tilde B_{\frac{5R}{3}}(p,1)} \sum_{ r_x\leq r_i \leq 2^{-5}} r_i^{-k}\chi_{\{D_R(x,y) \leq r_i\}}E_y(r_i) d\mu(x)  d\mu(y).
\end{aligned}
\end{equation}

Now, whenever $x,y\in C$, $ r_x\leq r_i$ and  $D_R(x,y)\leq r_i$, (n5) gives  
$$r_y \leq 1.01 r_x+ \delta D_R(x,y)<2r_i=r_{i-1}.$$
as long as $\delta\leq\frac{1}{6}$. Thus,
$$
\left\{r_i,  r_x\leq r_i\leq 2^{-5} \textrm{ and } D_R(x,y)\leq r_i\right\}\subset \left\{r_i,  r_y \leq 2r_i\leq  2^{-4} \right\},
$$
which implies
\begin{equation}\label{eqn:xtoy}\sum_{ r_x\leq r_i \leq 2^{-5}} r_i^{-k} \chi_{\{D_R(x,y) \leq r_i\}}E_y(r_i) \leq \sum_{ r_y\leq 2 r_i \leq 2^{-4}} r_i^{-k} \chi_{\{D_R(x,y) \leq r_i\}} E_y(r_i).
\end{equation}

Then, using \eqref{eqn:xtoy} in \eqref{eqn:characteristic_Fubini}, we can estimate
\begin{equation}\label{eqn:Lip_bound_estimate}
\begin{aligned}
&\int_{\tilde B_{\frac{3R}{2}}(p,1)} \left( \sum_{ r_x\leq r_i \leq 2^{-5}}\frac{1}{\mu(\tilde B_R(x,r_i))}\int_{\tilde B_R(x,r_i)} E_y(r_i)d\mu (y)\right) d\mu(x) \leq\\
&\leq \frac{B}{R^k}\int_{\tilde B_{\frac{5R}{3}}(p,1)} \int_{\tilde B_{\frac{5R}{3}}(p,1)} \sum_{ r_y\leq 2 r_i \leq 2^{-4}} r_i^{-k}\chi_{\{D_R(x,y) \leq r_i\}}E_y(r_i) d\mu(x)  d\mu(y)\\
&\leq \frac{B}{ R^k}\int_{\tilde B_{\frac{5R}{3}}(p,1)} \sum_{ r_y\leq 2r_i \leq 2^{-4}} \mu ( \tilde B_R(y,r_i))r_i^{-k}  E_y(r_i)  d\mu(y).
\end{aligned}
\end{equation}

Now, by Lemma \ref{lemma:tilde_balls_inclusion} we know that $\tilde B_{2R}(y,r_i)\subset \tilde B_{2R}(y, 2r_i) \subset \tilde B_{2R}(p,1)$, for any $i\geq 5$ and $y\in \tilde B_{\frac{5R}{3}}(p,1)$. Therefore, when $2r_i\geq r_y$, the a priori Ahlfors bound implies, by Lemma \ref{lemma:balls_inclusion}, that
$$\mu(\tilde B_R(y,r_i))\leq \mu(\tilde B_R(y,2r_i))\leq 2^kB (Rr_i)^k,$$ so \eqref{eqn:Lip_bound_estimate} becomes
\begin{equation}\label{eqn:total_estimate}
\begin{aligned}
&\int_{\tilde B_{\frac{3R}{2}}(p,1)} \left( \sum_{ r_x\leq r_i \leq 2^{-5}}\frac{1}{\mu(\tilde B_R(x,r_i))}\int_{\tilde B_R(x,r_i)} E_y(r_i)d\mu (y)\right) d\mu(x) \leq\\
&\leq 2^kB^2 \int_{\tilde B_{\frac{5R}{3}}(p,1)} \sum_{ r_y\leq 2 r_i \leq 1} E_y(r_i)  d\mu(y)\\
&\leq 2^kC(T) B^2 \int_{\tilde B_{\frac{5R}{3}}(p,1)} \left( \mathcal W_y(\delta r_y^2) - \mathcal W_y(\delta^{-1}) \right) d\mu(y)\\
&\leq 2^kC(T)B^2  \mu(\tilde B_{\frac{5R}{3}}(p,1)) \delta,
\end{aligned}
\end{equation}
where $C(T)=C(n,C_I)$ is a large constant, depending only on $T(n,C_I)$, such that every $t\in [-\delta^{-1},-\delta r_y^2]$ belongs to at most $C(T)$ intervals of the form $[-2Tr_i^2,-r_i^2]$, with $r_y\leq 2r_i\leq 1$, $i\geq 5$.

By Lemma \ref{lemma:apriori_Ahlfors_larger_balls} we know that if $R\geq R(n,H,K)$ and $0<\delta\leq \delta(n,C_I,\Lambda,H,\eta|R)$ then $\mu(\tilde B_{\frac{5R}{3}}(p,1))\leq \hat B R^k$, where $\hat B=\hat B(n,B)$. Using this, and the a priori Ahlfors lower bound for the ball $\tilde B_R(p,1)$, \eqref{eqn:total_estimate} becomes
\begin{equation}
\begin{aligned}
&\int_{\tilde B_{\frac{3R}{2}}(p,1)} \left( \sum_{ r_x\leq r_i \leq 2^{-l}}\frac{1}{\mu(\tilde B_R(x,r_i))}\int_{\tilde B_R(x,r_i)} E_y(r_i)d\mu (y)\right) d\mu(x) \\
&\leq 2^kC(T) B^2 \hat B R^k\delta,\\
&
\leq 2^kC(T) B^3 \hat B \mu(\tilde B_R(p,1)) \delta,\\
&\leq 2^nC(T) B^3 \hat B\mu(\tilde B_{\frac{3R}{2}}(p,1)) \delta,
\end{aligned}
\end{equation}
from which the result follows, if we further assume $0<\delta<(2^nC(T)B^3\hat B)^{-1}\varepsilon=\delta(n, C_I,H,B|\varepsilon)$.
\end{proof}

\begin{lemma}\label{lemma:not_degenerate}
Fix $\varepsilon>0$. Let $(M,g(t),p)_{t\in (-2\delta^{-2},0]}$ be a smooth compact Ricci flow satisfying (RF1-4) and
$$\mathcal N = \tilde B_{2R}(p,1) \setminus \bigcup_{x\in C} \overline{\tilde B_R(x,r_x)}$$ 
be a $(k,\delta,\eta)$-neck region that satisfies the a priori assumptions with
\begin{align*}
R&\geq R(n,H,K,B),\\
0<\delta &\leq \delta(n,C_I,\Lambda,H,K,\eta,B|R,\varepsilon)\\
 0<\delta' &\leq\delta'(n,C_I,\Lambda,H,K,\eta,B|R,\varepsilon)\\
 0<\delta'' &\leq\delta''(n,C_I,\Lambda,H,K,\eta,B|\varepsilon).
 \end{align*}
 If $x\in C \cap \tilde B_{\frac{3R}{2}} (p,1)$ satisfies
\begin{equation*}
\sum_{ r_x \leq r_i \leq 2^{-5}}\frac{1}{ \mu(\tilde B_R(x, r_i))} \int_{\tilde B_R (x, r_i)} \left(\mathcal W_y (r_i^2) - \mathcal W_y (2 T r_i^2)\right) d\mu(y) <\delta'',
\end{equation*}
then $v$ is a $(k,\varepsilon)$-splitting map around $x$ at every scale $r\in [r_x, 1]$.\end{lemma}

\begin{proof}
By the a priori assumptions and Lemma \ref{lemma:near_splitting}, if $0<\delta'\leq\delta'(\tilde\delta)$ is small enough then $v$ is a $(k,\tilde\delta)$-splitting map around any $x\in \tilde B_{2R}(p,1)$ and any scale $r\in [2^{-5},1]$. Thus it suffices to assume that $x\in C\cap \tilde B_{\frac{3R}{2}}(p,1)$ satisfies $r_x\leq 2^{-5}$ and choose $\tilde \delta\leq \varepsilon$ small enough.

\noindent \textbf{Claim:} There is a non-negative integer $m=m(n,B)$ and a small constant $\alpha=\alpha(n,B)>0$, such that when $R\geq R(n,H,K,B)$ and $0<\delta\leq\delta(n,C_I,\Lambda,H,\eta,B|R)$, the following holds: for every $x\in C \cap \tilde B_{\frac{3R}{2}} (p,1)$ with $2^{m +2}r_x\leq 2^{-5}$ and $r_i\in[2^{m+2}r_x,2^{-5}]$ we can find $\{x_j\}_{j=0}^k\subset \tilde B_R (x,r_i)\cap C$,  such that $\{x_j\}_{j=0}^k$ is $(k,\alpha,D' r_i)$-independent in $\tilde B_R(x,r_i)$ at time $t=-r_i^2$, and
\begin{equation}
\begin{aligned}
\mathcal E^{(k,\alpha,\delta,R)}_{r_i} (x) &\leq \sum_{i=0}^k \mathcal W_{x_j}(r_i^2) - \mathcal W_{x_j}(2T r_i^2) \\
&\leq \frac{C(n)}{\mu(\tilde B_R(x,r_i))}\int_{\tilde B_R(x,r_i)} \mathcal W_{y}(r_i^2) - \mathcal W_{y}(2T r_i^2) d\mu(y).
\end{aligned}
\end{equation}
\begin{proof}[Proof of the claim]
By Lemma \ref{lemma:tilde_balls_inclusion}, if $R\geq 12K$ then for every $y\in \tilde B_{\frac{3R}{2}}(p,1)$ and $r_i\leq 2^{-5}$ we have the inclusion $\tilde B_{2R}(y,r_i)\subset \tilde B_{2R}(p,1)$. Let $m\geq0$ be an integer which will be specified during the proof of the claim, and set $\zeta=2^{-m}$.

Let $x\in C\cap \tilde B_{\frac{3R}{2}}(p,1)$ be such that $2^{m+2} r_x \leq 2^{-5}$. By the weak Ahlfors assumption, for each $i$ with $r_i\in [2^{m+2} r_x,2^{-5}]$ there exists $C_{r_i,x} \subset C\cap \tilde B_R(x,r_i)$ such that 
\begin{equation}\label{eqn:mu_C}
\mu(C_{r_i,x}) \geq (2B)^{-1} (R r_i)^k
\end{equation}
 and for each $z\in C_{r_i,x}$
\begin{equation}\label{eqn:drop_average}
\mathcal W_{z}(r_i^2) - \mathcal W_{z}(2T r_i^2) \leq  \frac{C(n)}{\mu(\tilde B_R(x,r_i))} \int_{\tilde B_R(x,r_i)}( \mathcal W_{y}(r_i^2) - \mathcal W_{y}(2T r_i^2) )d\mu(y).
\end{equation}
By the a priori assumptions, for any $y\in C$ and $r\in [r_y,1]$, we know that
\begin{equation*}
\mathcal W_y(\delta r^2) - \mathcal W_y(\delta^{-1} r^2) <\delta.
\end{equation*}

Hence,
\begin{equation}\label{eqn:e_pinch_W}
\mathcal E^{(k,\alpha,\delta,R)}_{r_i}(x)\leq \sum_{j=0}^{k} \mathcal W_{x_j}(r_i^2) - \mathcal W_{x_j}(2Tr_i^2) ,
\end{equation}
for any subset $\{x_j\}_{j=0}^k \subset \tilde B_R(x,r_i)\cap C$, which is $(k,\alpha,D' r_i)$-independent at time $t=-r_i^2$. We will see below that $C_{r_i,x}$ contains such subset, hence the claim will follow from \eqref{eqn:drop_average} and \eqref{eqn:e_pinch_W}.

 Suppose this doesn't happen and that every subset $\{x_j\}_{j=0}^k\subset C_{r_i,x}$ is not $(k,\alpha,D' r_i)$-independent  in $\tilde B_R(x,r_i)$ at time $t=-r_i^2$. By Corollary \ref{cor:coverings}, if $0<\alpha\leq\alpha(n,\zeta)$, $R\geq R(n,H,K|\alpha,\zeta)$,  and $0<\delta\leq \delta(n,C_I,\Lambda,H,\eta|\alpha,\zeta,R)$,  then there is a subset $\{y_l\}_{l=1}^N \subset C_{r_i,x}$, with $N\leq C(n) \zeta^{-k+1}$, such that
 $$C_{r_i,x} \subset \bigcup_{j=l}^N \tilde B_R(y_l,\zeta r_i).$$
Now, since $y_l\in \tilde B_R(x,r_i)$, if $\zeta r_i\geq 2r_x $,  and  $\delta\leq\frac{\zeta}{4}$, (n5) gives
$$r_{y_l}\leq 1.01 r_x + \delta D_R(x,y_l) \leq \zeta r_i.$$
 Hence, by the a priori Ahlfors bound and \eqref{eqn:mu_C}  we obtain
\begin{equation*}
(2B)^{-1} R^k r_i^k\leq \mu(C_{r_i,x}) \leq \sum_l \mu(\tilde B_{R}(y_l,\zeta r_i)) \leq C(n) B \zeta^{-k+1} R^k r_i^k \zeta^k\leq  C(n) B \zeta R^k r_i^k.
\end{equation*}
Choosing  the largest $\zeta=2^{-m} \leq  (4C(n)B^2)^{-1}$  we arrive at a contradiction.

\end{proof}

Now, by the claim, if  $R\geq R(n,H,K,B)$, $0<\delta\leq \delta(n,C_I,\Lambda,H,\eta,B|R)$, $m=m(n,B)$, and $\alpha=\alpha(n,B)$,
\begin{equation*}
\begin{aligned}
\sum_{2^{m+2}r_x\leq r_i\leq 2^{-5}} \mathcal E^{(k,\alpha,\delta,R)}_{r_i}(x) &\leq \sum_{2^{m+2}r_x \leq r_i\leq 2^{-5}} \frac{C(n)}{\mu(\tilde B_R(x,r_i))}\int_{\tilde B_R(x,r_i)} \left( \mathcal W_{y}(r_i^2) - \mathcal W_{y}(2T r_i^2) \right)d\mu(y) \\
&< C(n)\delta''.
\end{aligned}
\end{equation*}

Then, by the non-degeneration Theorem \ref{thm:non_degen}, if $0<\max(\delta,\tilde\delta,C(n)\delta'')\leq \delta_0(n,C_I,\Lambda,H,K,\eta|R,\alpha,\tilde\varepsilon)$ it follows that $v$ is a $(k,\tilde\varepsilon)$-splitting map at any scale $r\in [2^{m+2}r_x,2^{-5}]$.

Applying Lemma \ref{lemma:near_splitting} once more, we obtain that if $0<\tilde\varepsilon\leq \tilde\varepsilon(m|\varepsilon)=\tilde\varepsilon(n,B|\varepsilon)$ then $v$ is a $(k,\varepsilon)$-splitting map also at scales $r\in [r_x,2^{m+2}r_x]$, which proves the result.
\end{proof}

\subsection{The bi-H\"older and bi-Lipschitz structure of the set of centres}

\begin{proposition}\label{prop:bilip}
Fix $\varepsilon>0$ and  suppose that $(M,g(t),p)_{t\in (-2\delta^{-3},0]}$ is a smooth compact Ricci flow satisfying (RF1-4) and
$$\mathcal N = \tilde B_{2R}(p,1) \setminus \bigcup_{x\in C} \overline{\tilde B_R(x,r_x)}$$ 
is a $(k,\delta,\eta)$-neck region that satisfies the a priori assumptions with
\begin{equation} \label{eqn:prop_bi_Lip_assumptions}
\begin{aligned}
R &\geq R(n,H,K,\tau,B|\varepsilon),\\
0<\delta&\leq \delta(n,C_I,\Lambda,H,\eta,B|R,\varepsilon),\\
 0<\delta'&\leq \delta'(n,C_I,\Lambda,H,\eta,B|R,\varepsilon),
 \end{aligned} 
 \end{equation}
 Then there is $C_\varepsilon \subset C\cap \tilde B_{\frac{3R}{2}}(p,1)$ such that
 
\begin{enumerate}
\item $\mu(C_\varepsilon) \geq (1-\varepsilon) \mu (C\cap \tilde B_{\frac{3R}{2}}(p,1))$
\item For any $x\in C_\varepsilon$ and $r\in[r_x,1]$, 
$v$ is a $(k,\varepsilon)$-splitting map around $x$ at scale $r$.
\item For every $x,y\in C_\varepsilon$ and $\tau^2 R\leq \sigma \leq 4R$ such that $r_x\leq D_\sigma(x,y) <1$
\begin{equation}\label{eqn:bi_lipschitz}
(1-\varepsilon) \sigma D_{\sigma}(x,y)\leq |\tilde v(x)-\tilde v(y)|\leq (1+\varepsilon) \sigma D_{\sigma}(x,y). 
\end{equation}

\item For every $x,y\in C\cap \tilde B_{\frac{3R}{2}}(p,1)$ and $\tau^2 R\leq \sigma \leq 4R$ such that $r_x\leq D_\sigma(x,y) <1$
\begin{equation}\label{eqn:bi_holder}
(1-\varepsilon) \sigma D_{\sigma}(x,y)^{1+\varepsilon} \leq|\tilde v(x)-\tilde v(y)|\leq (1+\varepsilon) \sigma D_{\sigma}(x,y).
\end{equation}
\end{enumerate}

\end{proposition}
\begin{proof}
First, we prove the following claim.\\

\noindent \textbf{Claim:} For any $0<\varepsilon<1$, if 
\begin{align*}
R&\geq R(n,H,K,B)\\
0<\delta&\leq\delta(n,C_I,\Lambda,H,\eta,B|R,\varepsilon)\\
0<\delta' &\leq\delta'(n,C_I,\Lambda,H,\eta,B|R)
\end{align*}
 then there is $C_{\varepsilon}\subset C\cap \tilde B_{\frac{3R}{2}}(p,1)$  and $E=E(n,H)<+\infty$ such that
\begin{enumerate}
\item $\mu(C_{\varepsilon})\geq (1-\varepsilon)\mu(C\cap \tilde B_{\frac{3R}{2}}(p,1))$
\item For every $x\in C_{\varepsilon}$, $r\in [r_x,1]$ and $y\in C \cap \overline{\tilde B_{4R}(x,r)}$ with $d_{g(-r^2)}(x,y)\geq \tau^2 R r$
\begin{equation}\label{eqn:lip1}
\begin{aligned}
&\left(1-\frac{\varepsilon}{2}\right) (d_{g(-r^2)}(x,y) )^2\leq |\tilde v(x) - \tilde v(y)|^2 + E r^2\\
&|\tilde v(x) - \tilde v(y)| \leq \left(1+\frac{\varepsilon}{2}\right) d_{g(-r^2)} (x,y).
\end{aligned}
\end{equation}
\item For every $x\in C\cap \tilde B_{\frac{3R}{2}}(p,1)$, $r\in[r_x,1]$ 
 and $y\in C\cap \overline{\tilde B_{4R}(x,r)}$ with $d_{g(-r^2)}(x,y)\geq \tau^2 R r$
\begin{equation}\label{eqn:lip2}
\begin{aligned}
\left(1 - \frac{\varepsilon}{2}\right) (d_{g(-r^2)}(x,y))^2 &\leq r^{-2\varepsilon} |\tilde v(x) - \tilde v(y)|^2 +E r^2,\\
|\tilde v(x) - \tilde v(y)| &\leq \left(1+\frac{\varepsilon}{2}\right) d_{g(-r^2)}(x,y).
\end{aligned}
\end{equation} 
\end{enumerate}
\begin{proof}[Proof of the claim]

For any $\theta>0$, by Lemma \ref{lemma:not_degenerate}, if 
\begin{align*}
R&\geq R(n,H,K,B),\\
0<\delta&\leq\delta(n,C_I,\Lambda,H,\eta,B|R,\theta),\\ 
0<\delta'&\leq\delta'(n,C_I,\Lambda,H,\eta,B|R,\theta),\\
 0<\delta'' &=\delta''(n,C_I,\Lambda,H,\eta,B|\theta),
 \end{align*}
 and $x\in C$ is such that 
\begin{equation}\label{eqn:Cepsilon}
\sum_{r_x \leq r_i \leq 2^{-5}}\frac{1}{ \mu(\tilde B_R(x, r_i))} \int_{\tilde B_R (x, r_i)} [\mathcal W_y (r_i^2) - \mathcal W_y (2T r_i^2)] d\mu(y) <\delta'',
\end{equation}
then  $v$ is a $(k,\theta)$-splitting map around $x$ at every scale $r\in [r_x,1]$.

We will define $C_{\varepsilon}\subset C \cap \tilde B_{\frac{3R}{2}}(p,1)$ such that $x\in C_{\varepsilon}$ if and only if
\begin{equation*}
\sum_{r_x \leq r_i \leq 2^{-5}}\frac{1}{ \mu(\tilde B_R(x, r_i))} \int_{\tilde B_R (x, r_i)} [\mathcal W_y (r_i^2) - \mathcal W_y (2T r_i^2)] d\mu(y) <\delta''=\delta''(n,C_I,\Lambda,H,\eta,B|\theta)
\end{equation*}
for an appropriate choice of $\theta$. 

We will first show that $C_\varepsilon$ satisfies Assertion 1 of the claim. Recall that Lemma \ref{lemma:small_average} ensures that, if $R\geq R(n,H,K)$ and $0<\delta\leq\delta(n,C_I,\Lambda,H,\eta,B|\varepsilon \delta'')$, then
\begin{equation*}
\frac{1}{\mu(\tilde B_{\frac{3R}{2}}(p,1))}\int_{\tilde B_{\frac{3R}{2}}(p,1)} \left( \sum_{r_x\leq r_i \leq 2^{-5}}\frac{1}{\mu(\tilde B_R(x,r_i))}\int_{\tilde B_R(x,r_i)} E_y(r_i)d\mu (y)\right) d\mu(x) < \varepsilon \delta'',
\end{equation*}
where $E_y(r_i)$ is defined as in \eqref{eqn:short}.

Hence,
\begin{align*}
& \frac{\delta'' \mu( B_{\frac{3R}{2}}(p,1)\setminus C_{\varepsilon})}{\mu(\tilde B_{\frac{3R}{2}}(p,1))}\leq \\
&\leq \frac{1}{\mu(\tilde B_{\frac{3R}{2}}(p,1))} \int_{B_{\frac{3R}{2}}(p,1)\setminus C_{\varepsilon}} \left(\sum_{r_x \leq r_i \leq 2^{-5}}\frac{1}{ \mu(\tilde B_R(x, r_i))} \int_{\tilde B_R (x, r_i)} [\mathcal W_y (r_i^2) - \mathcal W_y (2 T r_i^2)] d\mu(y) \right) d\mu(x),\\
&\leq \frac{1}{\mu(\tilde B_{\frac{3R}{2}}(p,1))}\int_{\tilde B_{\frac{3R}{2}}(p,1)} \left( \sum_{r_x\leq r_i \leq 2^{-5}}\frac{1}{\mu(\tilde B_R(x,r_i))}\int_{\tilde B_R(x,r_i)} \mathcal [W_y (r_i^2) - \mathcal W_y (2T r_i^2)]d\mu (y)\right) d\mu(x),\\
&< \varepsilon \delta'',
\end{align*}
which suffices to prove Assertion 1.

To prove Assertion 2, recall that by Lemma \ref{lemma:split_map_selfsimillar} there is $E=E(n,H)<+\infty$ such that choosing 
 \begin{equation}\label{eqn:delta_theta_small}
\begin{aligned}
0<\delta&\leq\delta(n,C_I,\Lambda,H,\eta|R,\varepsilon),\\
0<\theta&\leq \theta(n,C_I,\Lambda,H,\eta|R,\varepsilon)\leq\varepsilon,
\end{aligned}
\end{equation}
then for any $x\in C_\varepsilon$, $r\in [r_x,1]$ and  $y\in C\cap \overline{\tilde B_{4R}(x,r)}$ with $d_{g(-r^2)}(x,y) \geq \tau^2 Rr \geq \varepsilon r$
\begin{equation}
\begin{aligned}
&\left(1-\frac{\varepsilon}{10}\right)(d_{g(-r^2)}(x,y))^2 \leq |\tilde v(x) - \tilde v(y)|^2 + E r^2,\\
&|\tilde v(x) - \tilde v(y)| \leq \left(1+\frac{\varepsilon}{10}\right) d_{g(-r^2)} (x,y),
\end{aligned}
\end{equation}
since, by the definition of $C_{\varepsilon}$, $v$ is a $(k,\theta)$-splitting map around $x$ at any scale $r\in [r_x,1]$.

To prove Assertion 3, recall that, if $0<\max(\delta,\delta') \leq \delta(n,C_I,\Lambda,H,\eta|\theta)$, then for any $x\in C$ the maps $T_{x,r} v$ are $(k,\theta)$-splitting maps around $x$ at scale $r$, for any $r\in [r_x,1]$. Moreover, 
$$||T_{x,r}||\leq (1+\theta) r^{-C(n)\sqrt\theta} \leq \left( 1+\theta\right) r^{-2\varepsilon},$$
if $\theta>0$ is small enough so that $C(n)\sqrt\theta<\varepsilon$, and  $|\xi|^2 \leq (1+\theta) |T_{x,r} (\xi)|^2$, for any $\xi \in \mathbb R^k$.

Thus, by Lemma \ref{lemma:split_map_selfsimillar}, if \eqref{eqn:delta_theta_small} hold, then for any $y\in C\cap \overline{ \tilde B_{4R}(x,r)}$
\begin{equation}
\begin{aligned}
\left(1-\frac{\varepsilon}{10}\right)(d_{g(-r^2)}(x,y))^2 &\leq || T_{x,r}||^2 |\tilde v(x) - \tilde v(y)|^2 + E r^2,\\
&\leq \left( 1+\theta\right)^2 \left( r^{-2\varepsilon} |\tilde v(x) - \tilde v(y)|^2 + E r^2 \right),\\
(1+\theta)^{-1/2}|\tilde v(x) - \tilde v(y)| &\leq |T_{x,r}\tilde v(x) - T_{x,r}\tilde v(y)| \leq \left(1+\frac{\varepsilon}{10}\right) d_{g(-r^2)} (x,y).
\end{aligned}
\end{equation}
Choosing $\theta>0$ small enough, we obtain the result.

\end{proof}

We already know from the claim that $C_\varepsilon$ satisfies  Properties (1) and (2) of Proposition \ref{prop:bilip}. We will show below that it also satisfies Property (3) and as well as the remaining Property (4).

\noindent \textbf{The bi-Lipschitz estimate on $C_{\varepsilon}$.}

Take any $x,y\in C_{\varepsilon}$, $x\not=y$, and recall that by Remark \ref{rmk:sigma}, $d_{g(-r_x^2)}(x,y)\geq \tau^2 R r_x$ and there is  $\tau^2 R\leq \sigma \leq 4R$ such that $r_x \leq D_\sigma(x,y) <1$. Set $r=D_{\sigma}(x,y)$, so that $d_{g(-r^2)}(x,y) = \sigma r$ and $y\in  \overline{\tilde B_{4R}(x,r)}$.

Applying \eqref{eqn:lip1} of the claim, we  that
\begin{align*}
&\left(1-\frac{\varepsilon}{2}\right)(\sigma D_{\sigma}(x,y))^2=\left(1-\frac{\varepsilon}{2}\right)( \sigma r )^2=\\
&=\left(1-\frac{\varepsilon}{2}\right)(d_{g(-r^2)}(x,y))^2 \leq |\tilde v(y) - \tilde v(x)|^2 +Er^2\\
&\Longleftrightarrow \quad \left(1-\frac{\varepsilon}{2}-\frac{E}{\sigma^2}\right) (\sigma D_{\sigma}(x,y) )^2\leq |\tilde v(y) - \tilde v(x)|^2
\end{align*}
and
\begin{align*}
|\tilde v(y) - \tilde v(x)| &\leq \left(1+\frac{\varepsilon}{2}\right)d_{g(-r^2)}(x,y),\\
 &=\left(1+\frac{\varepsilon}{2}\right) \sigma D_{\sigma}(x,y). 
\end{align*}
If we choose $R\geq R(n,H,\tau|\varepsilon)$ so that $\tau^2 R\geq \frac{2E}{\varepsilon}$, then $\frac{E}{\sigma^2}<\frac{\varepsilon}{2}$, and we obtain \eqref{eqn:bi_lipschitz}.

\noindent \textbf{The bi-H\"older estimate on $C$.}

Take $x,y\in C\cap B_{\frac{3R}{2}} (p,1)$, $x\not=y$, with $d_{g(-r_x^2)}(x,y)\geq \tau^2 R r_x$, and suppose that $\tau^2 R\leq \sigma < 4R$ is such that $r_x \leq D_\sigma(x,y) < 1$. Set $r=D_{\sigma}(x,y)$, so that $d_{g(-r^2)}(x,y) = \sigma r$ and $y\in  \overline{\tilde B_{4R}(x,r)}$.

Applying \eqref{eqn:lip2} of the claim we obtain
\begin{align*}
\left(1-\frac{\varepsilon}{2}\right)(d_{g(-r^2)}(x,y))^2&\leq  r^{-2\varepsilon} |\tilde v(x) - \tilde v(y)|^2 + E r^2.
\end{align*}
Hence, 
\begin{equation*}
\left(1-\frac{\varepsilon}{2}\right)\left(\sigma D_{\sigma}(x,y)^{1+\varepsilon} \right)^2\leq  |\tilde v(x) - \tilde v(y)|^2 + \frac{E}{\sigma^2} \left(\sigma D_{\sigma}(x,y)^{1+\varepsilon} \right)^2.
\end{equation*}
Thus, 
\begin{align*}
\left(1-\frac{\varepsilon}{2} - \frac{E}{\sigma^2}\right) \left(\sigma D_{\sigma}(x,y)^{1+\varepsilon}\right)^2 \leq |\tilde v(x) - \tilde v(y)|^2,
\end{align*}
which proves the left hand side of \eqref{eqn:bi_holder}, if we further assume that $R\geq R(n,H,\tau|\varepsilon)$, as in the proof of the bi-Lipschitz estimate.

On the other hand, by \eqref{eqn:lip2} we have
$$|\tilde v(x)-\tilde v(y)|\leq (1+\varepsilon) d_{g(-r^2)}(x,y)=(1+\varepsilon) \sigma  D_\sigma (x,y),$$
which proves the right hand side of \eqref{eqn:bi_holder}.
\end{proof}

\subsection{Ahlfors bound for $(k,\delta,\eta)$-neck regions that satisfy the a priori assumptions}

\begin{lemma}[Upper Ahlfors bound]\label{lemma:apriori_upper}
Suppose that $(M,g(t),p)_{t\in (-2\delta^{-3},0]}$ is a smooth compact Ricci flow satisfying (RF1-4) and
$$\mathcal N = \tilde B_{2R}(p,1) \setminus \bigcup_{x\in C} \overline{\tilde B_R(x,r_x)}$$ 
is a $(k,\delta,\eta)$-neck region that satisfies the a priori assumptions with 
\begin{align*}
R&\geq R(n,H,K,\tau,B),\\
0<\delta&\leq\delta(n,C_I,\Lambda,H,\eta,B|R),\\
0<\delta'&\leq\delta'(n,C_I,\Lambda,H,\eta,B|R).
\end{align*}  
 Then, there is $F=F(k,\tau)<+\infty$ such that
\begin{equation*}
\mu(\tilde B_R(p,1)) \leq F R^k,
\end{equation*}
\end{lemma}
\begin{proof}
Choose $\varepsilon= 1/2$ and apply Proposition \ref{prop:bilip}  to obtain the set $C_{1/2}\subset \tilde B_{\frac{3R}{2}}(p,1)$, such that \begin{equation}\label{eqn:doubling_mu}
\mu(\tilde B_R(p,1)) \leq 2\mu(C_{1/2}).
\end{equation}

Let $x,y\in C_{1/2}$, $x\not =y$. By Remark \ref{rmk:sigma}, let $\tau^2 R\leq \sigma \leq 4R$ be such that $r_x\leq D_\sigma (x,y) <1$.

The bi-Lipschitz estimate \eqref{eqn:bi_lipschitz} gives that
$$\tau^2 R r_x\leq \sigma D_\sigma(x,y) \leq \frac{3}{2}|\tilde v(x) -\tilde v(y)|,$$
under the assumptions of Proposition \ref{prop:bilip}.
It follows that the balls  $B_x:= B_{\frac{1}{3} \tau^2 R r_x}(x) \subset \mathbb R^k$, $x\in C_{1/2}$ are disjoint.

Now, we claim that for every $x\in C\cap \tilde B_{\frac{3R}{2}}(p,1)$, $|\tilde v(x)| \leq  6R$. To see this, let $x\in C\cap \tilde B_{\frac{3R}{2}}(p,1)$ and $\tau^2 R\leq \sigma\leq 4R$ be such that $r_p \leq D_\sigma (p,x)<1$, as guaranteed by Remark \ref{rmk:sigma}.
 
Estimate \eqref{eqn:bi_holder} of Proposition \ref{prop:bilip} then gives 
\begin{equation*}
|\tilde v(x)|= |\tilde v(x)-\tilde v(p)| \leq \frac{3}{2} \sigma D_{\sigma}(p,x)\leq 6 R.
\end{equation*}

Since $0<\tau\leq 1$, and  $r_x\leq 2^{-5}$, by  the a priori assumptions, we obtain 
\begin{equation}
\bigcup_{x\in C_{1/2}} B_x \subset B_{\left(6+\frac{\tau^2  r_x}{3} \right) R}(0)\subset B_{7R}(0).
\end{equation}
Therefore, denoting by $\vol_k$ the $k$-volume in $\mathbb R^k$, 
\begin{align*}
\mu(C_{1/2})&=\sum_{x\in C_{1/2}} (R  r_x)^k, \\
&\leq C(k,\tau) \sum_{x\in C_{1/2}} \vol_k(B_x),\\
&\leq C(k,\tau) \vol_k(B_{7R}(0))\\
&\leq \hat C(k,\tau)R^k.
\end{align*}
Therefore, \eqref{eqn:doubling_mu} gives that
$$\mu(\tilde B_R(p,1)) \leq 2\mu(C_{1/2}) \leq 2\hat C(k,\tau)R^k,$$
which proves the result.
\end{proof}

\begin{lemma}[Covering]\label{lemma:covering}
Suppose that $(M,g(t),p)_{t\in (-2\delta^{-3},0]}$ is a smooth compact Ricci flow satisfying (RF1-4) and
$$\mathcal N = \tilde B_{2R}(p,1) \setminus \bigcup_{x\in C} \overline{\tilde B_R(x,r_x)}$$ 
is a $(k,\delta,\eta)$-neck region with $\tau\leq \frac{1}{100}$ and  satisfies the a priori assumptions with \begin{align*}
R&\geq R(n,H,K,\tau,B),\\
0<\delta&\leq \delta(n,C_I,\Lambda,H,\eta,B|R,\tau),\\
0<\delta' &\leq \delta'(n,C_I,\Lambda,H,\eta,B|R,\tau).
\end{align*}

For each $x\in C\cap \tilde B_R(p,1)$, let $I_{x,r_x} := \tilde v(x)+  T_{x,r_x}^{-1}( \overline{B_{\frac{1}{2}Rr_x}(0)} )$. Then 
$$B_{2^{-7} R}(0)\subset \bigcup_{x\in C\cap \tilde B_R(p,1)} I_{x,r_x}.$$
\end{lemma}

\begin{proof}
Suppose that there is $w\in B_{2^{-7} R}(0)$ such that
\begin{equation}\label{eqn:not_in}
w\not \in \bigcup_{x\in C\cap \tilde B_R(p,1)} I_{x,r_x}.
\end{equation}
For each $x\in C\cap \tilde B_R(p,1)$ and $r\in [r_x,1]$ let 
$$I_{x,r} = \tilde v(x) + T_{x,r}^{-1} ( \overline{B_{\frac{1}{2}Rr}(0)}),$$
and define
$$s_x= \inf \left\{r\in [ r_x,1], w\in I_{x,r} \right\}.$$

Choose $\delta'>0$ small enough so that, by Lemma \ref{lemma:normalize_splitting}, 
$T_{p,2^{-5}}$ is close enough to the identity to achieve $I_{p,2^{-5}} =  T_{p,2^{-5}}^{-1}(B_{2^{-6} R}(0))\supset B_{2^{-7} R}(0)$. Thus we may assume that $s_p \leq 2^{-5}$.

Now, set
$$\bar s = s_{\bar x}= \min_{x\in C\cap \tilde B_R(p,1) }  s_x. $$
The minimum is achieved because $C$ is a finite set,  by Remark \ref{rmk:infrx}, since the neck region is smooth.

By \eqref{eqn:not_in} it follows that $r_{\bar x}<s_{\bar x}=\bar s \leq s_p \leq 2^{-5}$. Moreover, $w\in I_{\bar x,2\bar s}$, hence
\begin{equation}\label{eqn:covering_2}
|T_{\bar x,2\bar s} w - T_{\bar x,2\bar s}  \tilde v(\bar x)| < R \bar s.
\end{equation}
By \eqref{eqn:T_estimates}, we know that if $0<\max(\delta,\delta')\leq \delta'(n,C_I,\Lambda,H,\eta|R)$ then $|\xi| \leq 2 | T_{\bar x,2\bar s} (\xi)|$
hence
\begin{align*}
|\tilde v(\bar x)|&\leq |\tilde v(\bar x) - w|+|w|\\
&\leq 2 |T_{\bar x,2\bar s}(\tilde v(\bar x)- w)| +|w|,\\
&<  R\bar s + 2^{-7} R,\\
&< 2^{-5} R + 2^{-7} R,\\
&\leq 2^{-4} R.
\end{align*}
Since $\bar x\in \tilde B_R(p,1)$, we know that
$$D_R(p,\bar x) < 1.$$
If in addition $r_p\leq D_R (p,\bar x)$, applying Proposition \ref{prop:bilip} for $\varepsilon=\frac{1}{2}$, we know that if 
\begin{align*}
R&\geq R(n,H,K,\tau,B),\\
0<\delta&\leq \delta(n,C_I,\Lambda,H,\eta,B|R),\\
0<\delta'&\leq \delta'(n,C_I,\Lambda,H,\eta,B|R),
\end{align*}
the bi-H\"older estimate \eqref{eqn:bi_holder} holds, and gives
\begin{equation}\label{eqn:barx_close_p}
\frac{1}{2}  R (D_R(p,\bar x))^{3/2} \leq |\tilde v(\bar x)|\leq 2^{-4} R \quad\Longrightarrow\quad D_R(p,\bar x) \leq \left(\frac{1}{8}\right)^{2/3}=\frac{1}{4}.
\end{equation}
Therefore, since by the a priori assumptions $r_p \leq 2^{-5}$,
$$D_R(p,\bar x) \leq \max(r_p,\frac{1}{4})\leq \frac{1}{4}.$$

Moreover, by  Proposition \ref{prop:D_metric} and \eqref{eqn:barx_close_p} we can estimate, for any $y\in \tilde B_{2R}(\bar x,2\bar s)$
\begin{align*}
D_R(p,y) &\leq \frac{R}{R-K} (D_R(p,\bar x)+D_R(\bar x,y)),\\
&\leq \frac{R}{R-K} \left(D_R (p,\bar x) + \frac{2R}{R-K}D_{2R}(\bar x,y)\right),\\
&<\frac{R}{R-K} \left( \frac{1}{4} +6 \bar s\right),\\
&\leq \frac{R}{R-K} \left( \frac{1}{4} +\frac{1}{4}\right),\\
&< 1,
\end{align*}
as long as $R$ is large enough, depending on $K$, so that $\frac{2R}{R-K} \leq 3$, which proves that 
\begin{equation}\label{eqn:covering_b_inclusion}
\tilde B_{2R}(\bar x,2\bar s) \subset \tilde B_R(p,1).
\end{equation}

\noindent \textbf{Claim 1:} If 
\begin{align*}
0<\delta&\leq\delta(n,C_I,\Lambda,H,\eta|R,\tau),\\
0<\delta'&\leq\delta(n,C_I,\Lambda,H,\eta|R,\tau),
\end{align*}
then there is $y\in C\cap \tilde B_R(p,1)$ such that
\begin{equation}
|T_{\bar x,2\bar s} w - T_{\bar x,2\bar s} \tilde v(y)| <4\tau  R\bar s.
\end{equation}
\begin{proof}
Recall that by the definition of a neck region satisfying the a priori assumptions
\begin{itemize}
\item there is $\tilde x\in \tilde B_R(\bar x,2\bar s)$ such that $(M,g(t),\tilde x)_{t\in (-2\delta^{-1},0]}$ is $(k,\delta)$-selfsimilar, with respect to $\mathcal L_{\tilde x,2\bar s}$, but not $(k+1,\eta)$-selfsimilar at scale $2\bar s$. 
\item we know that
\begin{equation}\label{eqn:inclusion}
\mathcal L_{\tilde x,2\bar s} \cap \tilde B_R(\bar x,2\bar s) \subset \tilde B_{\tau R} (C,2\bar s),
\end{equation}
since $\tilde B_{2R}(\bar x,2\bar s)  \subset \tilde B_{2R}(p,1)$, by \eqref{eqn:covering_b_inclusion}.
\item $v$ is a $(k,\delta')$-splitting map at scale $1$ for any $x\in C$
\end{itemize}

Moreover, for any $\theta>0$, we know that if $0<\max(\delta,\delta')\leq\delta(n,C_I,\Lambda,H,\eta|R,\theta)$, then for every $x\in C\cap \tilde B_R(p,1)$, $T_{\bar x,2\bar s} v$ is a $(k,\theta)$-splitting map around $x$ at scale $2\bar s$, satisfying \eqref{eqn:Tv_normalization} and \eqref{eqn:T_estimates}.

We can then apply Lemma \ref{lemma:split_map_selfsimillar} at scale $2\bar s$, so that if  $0<\delta\leq \delta(n,C_I,\Lambda,H,\eta|R,\tau)$ and $0<\theta\leq\theta(n,C_I,\Lambda,H,\eta|R,\tau)$, we know that for every $a\in B_{R\bar s}(0)$
there is $z\in \tilde B_R(\bar x,2 \bar s)\cap \mathcal L_{\tilde x,2\bar s}$ such that
\begin{equation}
|a-(T_{\bar x,2\bar s} \tilde v(z) - T_{\bar x,2\bar s} \tilde v(\bar x))| < \tau R \bar s.
\end{equation}

Thus, since $|T_{\bar x,2\bar s} w - T_{\bar x,2\bar s} \tilde v(\bar x) | < R\bar s$, there is $z\in \tilde B_R(\bar x,2\bar s)\cap \mathcal L_{\bar x,2\bar s}$ such that
\begin{equation}
 |T_{\bar x,2\bar s} w-T_{\bar x,2\bar s} \tilde v(z)|=|T_{\bar x,2\bar s} w - T_{\bar x,2\bar s} \tilde v(\bar x) -(T_{\bar x,2\bar s} \tilde v(z) - T_{\bar x,2\bar s} \tilde v(\bar x))|<\tau R \bar s. 
\end{equation}

Moreover, by \eqref{eqn:inclusion} there is $y\in C$ such that 
\begin{equation}\label{eqn:close_to_C}
d_{g(-4\bar s^2)}(z,y) < 2\tau R \bar s.
\end{equation}
In particular, $z\in \tilde B_{\tau R}(y,2\bar s)\subset \tilde B_{4R}(y,2\bar s)$. Hence, by Lemma \ref{lemma:split_map_selfsimillar}, 
\begin{align*}
|T_{\bar x,2\bar s}\tilde v(z)- T_{\bar x,2\bar s} \tilde v(y)| \leq \max\left\{2\bar s \tau,\frac{3}{2} d_{g(-4\bar s^2)}(
z,y)\right\} <3\tau R \bar s.
\end{align*}

Putting everything together, we obtain 
\begin{align*}
|T_{\bar x,2\bar s} w -T_{\bar x,2\bar s}\tilde v(y)|&\leq |T_{\bar x,2\bar s} w -T_{\bar x,2\bar s}\tilde v(z)| + |T_{\bar x,2\bar s}\tilde v(z) -T_{\bar x,2\bar s}\tilde v(y)|\\
&\leq  \tau R \bar s + 3\tau R\bar s\\
&\leq 4\tau R \bar s.
\end{align*}

Moreover,
\begin{align*}
d_{g(-4\bar s^2)}(y,\bar x) &\leq d_{g(-4\bar s^2)}(z,\bar x) +d_{g(-4\bar s^2)}(y,z)\\
&\leq 2R\bar s + 2\tau R\bar s \leq 2(1+\tau)R\bar s,
\end{align*}
hence $y\in \tilde B_{2R}(\bar x , 2\bar s)\subset \tilde B_R(p,1)$, by \eqref{eqn:covering_b_inclusion}.

\end{proof}

\noindent \textbf{Claim 2:} If $0<\max(\delta,\delta')\leq \delta(n,C_I,\Lambda,H,\eta|R,\varepsilon)$ then
\begin{equation*}
\left|\left| T_{y,2\bar s} T_{\bar x,2\bar s}^{-1} -I \right|\right| <  C(n)\sqrt{\varepsilon}.
\end{equation*}
\begin{proof}
Let $\gamma=\gamma(n,H)\in (0,1]$ be the constant provided by Lemma \ref{lemma:near_splitting}. Assume that $2\gamma^{-1}\bar s\leq 1$. Since $T_{\bar x,2\gamma^{-1} \bar s} v$ is a $(k,\theta)$-splitting map at scale $2\gamma^{-1}\bar s$, if $0<\max(\delta,\delta')\leq \delta(n,C_I,\Lambda,H,\eta|\theta)$, and $y\in \tilde B_{2R}(\bar x, 2\bar s) \subset \tilde B_{2R}(\bar x,2\gamma^{-1}\bar s)$, it follows by Lemma \ref{lemma:near_splitting} that if $0<\theta \leq \theta(n,C_I,H|R,\varepsilon)$, then $T_{\bar x,2\gamma^{-1}\bar s} v$ is also a $(k,\varepsilon)$-splitting map around $y$ at scale $2\bar s$.

On the other hand, since $T_{y,2\bar s} v=T_{y,2\bar s} T_{\bar x,2\gamma^{-1}\bar s}^{-1} T_{\bar x,2\gamma^{-1} \bar s} v$ satisfies the normalization \eqref{eqn:Tv_normalization} at scale $2\bar s$, it follows that $||T_{y,2\bar s} T_{\bar x,2\gamma^{-1}\bar s}^{-1} -  I || \leq C(n) \sqrt{\varepsilon}$, by Lemma \ref{lemma:normalize_splitting}. Similarly, since $T_{\bar x,2\bar s} v=T_{\bar x,2\bar s} T_{\bar x,2\gamma^{-1}\bar s} ^{-1}T_{\bar x,2\gamma^{-1}\bar s} v$ satisfies the normalization \eqref{eqn:Tv_normalization}, we know  that $||T_{\bar x, 2\bar s} T_{\bar x,2\gamma^{-1}\bar s}^{-1} - I ||\leq C(n) \sqrt\theta \leq  C(n)\sqrt \varepsilon$. From this, the result follows, in this case.

If $2\gamma^{-1}\bar s>1$ then we can use the almost splitting map $v$ instead of $T_{\bar x, 2\gamma^{-1}\bar s} v$, in the above argument.

\end{proof}

Now, by Claim 2, for $0<\varepsilon\leq \varepsilon(n)$, $0<\theta\leq \theta(n,C_I,H|R,\varepsilon)$  and $0<\max(\delta,\delta')\leq \delta(n,C_I,\Lambda,H,\eta|\theta)$, and Claim 1, we have
\begin{align*}
|T_{y,2\bar s} w - T_{y,2\bar s}  \tilde v(y)|  \leq ||T_{y,2\bar s} T_{x,2\bar s}^{-1}|| \; |T_{\bar x,2\bar s} w - T_{\bar x,2\bar s} \tilde v(y)| \leq 10 \tau R \bar s.
\end{align*}

Choosing $0<\theta\leq \theta(n)$ in \eqref{eqn:T_ratio} and  $\tau\leq \frac{1}{100}$, we obtain for $a=\max(\frac{\bar s}{2},r_y)$
\begin{align*}
|T_{y,a} w - T_{y,a} \tilde v(y) | &= |T_{y,a}T_{y,2\bar s}^{-1} T_{y,2\bar s}  (w-\tilde v(y)) |  \\
&\leq ||T_{y,a}T_{y,2\bar s}^{-1}|| \;|(T_{y,2\bar s} (w -\tilde v(y)) | \\
&\leq 20\tau R \bar s\\
&<\frac{R\bar s}{4} \leq \frac{R}{2}a.
\end{align*}

Therefore, $w\in \tilde v(y) + T_{y,a}^{-1} (B_{\frac{1}{2}Ra}(0)) \subset I_{y,a}$. By \eqref{eqn:not_in}, it follows that $r_y\leq s_y<a=\frac{\bar s}{2}<\bar s$, which contradicts the definition of $\bar s$.

\end{proof}

We thus arrive at the following lemma.

\begin{lemma}[Lower Ahlfors bound]\label{lemma:apriori_lower}
Suppose that $(M,g(t),p)_{t\in (-2\delta^{-2},0]}$ is a smooth compact Ricci flow satisfying (RF1-4) and
$$\mathcal N = \tilde B_{2R}(p,1) \setminus \bigcup_{x\in C} \overline{\tilde B_R(x,r_x)}$$ 
is a $(k,\delta,\eta)$-neck region with $\tau\leq \frac{1}{100}$ that satisfies the a priori assumptions with 
\begin{align*}
R&\geq R(n,H,K,\tau,B),\\
0<\delta&\leq \delta(n,C_I,\Lambda,H,\eta,B|R,\tau),\\
0<\delta'&\leq \delta'(n,C_I,\Lambda,H,\eta,B|R,\tau)
\end{align*}
then there is $G=G(k)<+\infty$ such that
\begin{equation}
G^{-1} R^k\leq \mu(\tilde B_R(p,1)).
\end{equation}
\end{lemma}
\begin{proof}
To obtain the lower Ahlfors bound, we have by Lemma \ref{lemma:covering} that
\begin{equation*}
C(k) R^k\leq \sum_{x\in C\cap \tilde B_R(p,1)}  \vol_k(T_{x,r_x}^{-1}(\overline{B_{\frac{1}{2}Rr_x}(0) }).  
\end{equation*}
However, since $|\xi|\leq 2 |T_{x,r_x}(\xi)|$ if $0<\max(\delta,\delta')\leq \delta(n,C_I,\Lambda,H,\eta)$, by \eqref{eqn:T_estimates}, we know that
$$ T_{x,r_x}^{-1}( \overline{B_{\frac{1}{2}Rr_x}(0) }) \subset \overline{B_{Rr_x}(0)}.$$
Therefore, denoting by $\omega_k$ the volume of the unit ball in $\mathbb R^k$,
\begin{equation*}
\begin{aligned}
\sum_{x\in C\cap \tilde B_R(p,1)}  \vol_k(T_{x,r_x}^{-1}(\overline{B_{\frac{1}{2}Rr_x}(0)}) &\leq  \sum_{x\in C\cap \tilde B_R (p,1)} \vol_k(\overline{B_{Rr_x}(0)})\\
&\leq  \omega_k \sum_{x\in C\cap \tilde B_R (p,1)}  (Rr_x)^k\\
&\leq  \omega_k  \mu (\tilde B_R(p,1)),
\end{aligned}
\end{equation*}
from which the result follows.

\end{proof}

\subsection{The smooth neck structure theorem} Now we ready to finish the proof of Theorem \ref{thm:neck_structure}. We recall it below for the convenience of the reader.
\begin{theorem}
Let $(M,g(t),p)_{t\in (-2\delta^{-3},0]}$ be a smooth Ricci flow satisfying (RF1-4). If $\tau\leq \frac{1}{100}$, $R\geq R(n,C_I,\Lambda,H,K|\tau)$, $0<\delta\leq\delta(n,C_I,\Lambda,H,K,\eta|R,\tau)$ and
\begin{equation*}
\mathcal N=\tilde B_{2R}(p,1) \setminus \bigcup_{x\in C} \overline{\tilde B_R(x,r_x)}
\end{equation*}
is a smooth $(k,\delta,\eta)$-neck region,  then for every $x\in C$ and $r\geq r_x$ with $\tilde B_{2R}(x,r)\subset \tilde B_{2R}(p,1)$ the estimate
\begin{equation*}
L^{-1} R^k r^k \leq \mu (\tilde B_R(x,r))\leq L R^k r^k
\end{equation*}
holds, for some $L=L(k,\tau)<+\infty$.
\end{theorem}
\begin{proof}
Take any $x\in C$ and $r\geq r_x$ such that $\tilde B_{2R}(x,r)\subset \tilde B_{2R}(p,1)$. By Remark \ref{rmk:neck_inside_neck}, $\tilde B_{2R}(x,r)$ contains a neck-region at scale $r$.

Moreover, for every $y\in \tilde B_{2R}(x,r)\cap C$
$$r_y\leq 1.01r_x+\delta D_R(x,y) \leq 1.01r +\frac{2R-K}{R-K} \delta D_{2R}(x,y)\leq 2r,$$
if $R\geq 2K$ and $0<\delta\leq \frac{1}{4}$.

Thus, by rescaling, if suffices to prove that for some $L<+\infty$, 
\begin{equation}\label{eqn:statement}
L^{-1} R^k\leq \mu(\tilde B_R(p,1))\leq L R^k,
\end{equation}
 for all neck regions at scale  $1$ with $p\in C$, $r_p \leq 1$ and $r_x\leq 2$ for every $x\in C$. We will do this by  induction on the depth $\inf r_x$ of the neck regions.

\noindent \textbf{Estimate \eqref{eqn:statement} holds for neck regions with depth $\geq 2^{-10}$.} Suppose that the neck region,  at scale $1$,  satisfies $\inf r_x\geq 2^{-10}$. By Corollary \ref{cor:coverings}, there is a maximal $\{x_i\}_{i=1}^N \subset \tilde B_{R}(p,1)\cap C$ such that the balls $\tilde B_{\frac{1}{4}\tau^2 R}( x_i, 2^{-10})$ are mutually disjoint, the balls $\tilde B_R(x_i,2^{-10})$ cover $C\cap \tilde B_R(p,1)$, and $N\leq C(n,\tau) 2^{10k}$. 

Moreover, we have that $C\cap \tilde B_R(p,1)= \{x_i\}_{i=1}^N$. To see this, any $y\in C\cap \tilde B_R(p,1)$ with $y\not \in \{x_i\}_{i=1}^N$ would satisfy, by (n1), $d_{g(-r_y^2)}(y,x_i)\geq \tau^2 R r_y\geq \tau^2 R 2^{-10}$, for every $i$. By (RF4) and the fact that $2^{-10}\leq r_y\leq 2$ we obtain
$$d_{g(-4^{-10})} (y,x_i)\geq d_{g(-r_y^2)}(y,x_i) - 2K \geq \tau^2 R 2^{-10} -2K\geq \frac{1}{2}\tau^2 R 2^{-10},$$
if $R\geq R(\tau,K)$, which contradicts the maximality of $\{x_i\}_{i=1}^N$.

 It follows that
\begin{equation}\label{eqn:step1_ua}
\mu(\tilde B_R(p,1))= \sum_{i=1}^N R^k r_{x_i}^k \leq 2^k 2^{10k}C(n,\tau) R^k,
\end{equation}
since $r_{x_i}\leq 2$. On the other hand, $p\in C$ and $2^{-10}\leq r_p\leq 1$, so
\begin{equation}\label{eqn:step1_la}
2^{-10k} R^k \leq R^k r_p^k \leq \mu(\tilde B_R(p,1)).
\end{equation}
Estimates \eqref{eqn:step1_ua} and \eqref{eqn:step1_la} prove \eqref{eqn:statement} for all neck regions with $\inf r_x\geq 2^{-10}$, for $L\geq C(n,\tau)2^{11k}$.

\noindent \textbf{Induction step.} Suppose that \eqref{eqn:statement} holds for all $(k,\delta,\eta)$-neck regions at scale $1$ with $\inf r_x\geq 2^{-j}$, for some $j\geq 10$, and for some $L<+\infty$ which will be determined in the course of the proof. 

Suppose also that $\tilde B_{2R}(p,1)$ contains a $(k,\delta,\eta)$-neck region with $2^{-j}\geq \inf r_x \geq 2^{-(j+1)}$. Once we show that this neck region satisfies the a priori assumptions the result will follow from Lemma \ref{lemma:apriori_upper} and Lemma \ref{lemma:apriori_lower}.

\noindent \textit{Existence of the splitting map.}

Let $\gamma=\gamma(n,H)$ be the constant provided by Lemma \ref{lemma:near_splitting}. Since $p\in C$ and we may assume by making $\delta$ small enough that $r_p<\gamma^{-1}<\delta^{-1}$, there is $q\in \tilde B_R(p,1)$ such that $(M,g(t),q)_{t\in (-2\delta^{-3},0)}$ is $(k,\delta^2)$-selfsimilar at scale $\gamma^{-1}$, by (n3). By Theorem \ref{thm:sharp_splitting} it follows that there is a $(k,\tilde\delta)$-splitting map around $p$ at scale $\gamma^{-1}$ with $v(p,0)=0$, if $0<\delta\leq \delta(n,C_I,\Lambda,H,K|R,\tilde\delta)$. By Lemma \ref{lemma:near_splitting}, $v$ is then a $(k,\delta')$-splitting map around any $x\in C$ at scale $1$, provided that $0<\tilde \delta\leq \tilde\delta(n,C_I,H|R,\delta')$. We will choose $\delta'$ small enough, towards the end of the proof, so that Lemma \ref{lemma:apriori_upper} and Lemma \ref{lemma:apriori_lower} apply.

\noindent \textit{Bound on $r_x$.}

We will show that $r_x\leq 2^{-8}$ for every $x\in C$. For this, since $2^{-10}\geq 2^{-j}\geq \inf_{x\in C} r_x \geq 2^{-(j+1)}$,  there is $\tilde x\in C$ such that $r_{\tilde x}\leq 2^{-10}$, hence by property (n5) of a neck region, if  $\delta$ is small enough,
$$r_p\leq 1.01r_{\tilde x} +\delta D_R(p,\tilde x)\leq 1.01\cdot 2^{-10} +4\delta\leq 2^{-9},$$
since $D_R(p,\tilde x)\leq \frac{2R-K}{R-K} D_{2R}(p,\tilde x)<4$, by Proposition \ref{prop:D_metric}, if $R\geq R(K)$.

Thus, by (n5), for every $x\in C\cap \tilde B_{2R}(p,1)$
\begin{equation}\label{eqn:rx_small}
r_x\leq 1.01r_p + \delta D_R(p,x)\leq  1.01 \cdot 2^{-9} +4\delta <2^{-8},
\end{equation}
if $\delta$ is small enough.

\textit{The weak Ahlfors assumption.} 

Consider $x\in C$, $r\in [r_x,1]$  such that $\tilde B_{2R}(x,r)\subset \tilde B_{2R}(p,1)$.

Suppose $r\leq 1/2$. Then $\tilde B_{2R}(x,r)$ contains a $(k,\delta,\eta)$-neck region at scale $r$ with set of centres $\tilde C=C\cap \tilde B_{2R}(x,r)$. Moreover, the depth of this neck region is $\geq 2^{-j}$, hence we can apply the induction hypothesis to obtain
\begin{equation}\label{eqn:small_depth}
L^{-1} (Rr)^k\leq \mu(\tilde B_R(x,r))\leq L (Rr)^k.
\end{equation}

Now, let $r > 1/2$. Applying Corollary \ref{cor:coverings} for $m=5$, we obtain a collection of points $\{y_i\}_{i=1}^N\subset \tilde B_R(x,r)$, so that the balls $\tilde B_R(y_i,2^{-5} r)\subset \tilde B_{2R}(x,r)$ cover $\tilde B_R(x,r)$, and $N\leq C(n) 2^{5k}$.

Since $r\leq 1$, it follows that $\tilde B_{2R}(y, 2^{-5}r)$ contains a neck region with depth $\geq  2^{-j}$, so the induction hypothesis gives
\begin{equation}\label{eqn:ahlfors_r}
L^{-1} (2^{-5} R  r )^k\leq \mu(\tilde B_R(y,2^{-5}r))\leq L ( 2^{-5} R r)^k,
\end{equation}
since $r_y\leq 2^{-8} \leq 2^{-5}r$, by \eqref{eqn:rx_small}.

Therefore, by \eqref{eqn:ahlfors_r},
\begin{align*}
\mu(\tilde B_R(x,r))\leq \sum_{i=1}^N \mu(\tilde B_R (y_i,2^{-5} r))\leq L \; C(n) 2^{5k} ( R r)^k.
\end{align*}

 On the other hand, since $r>1/2$, $\tilde B_R(x,1/2)\subset \tilde B_R(x,r)$, $\tilde B_{2R}(x,1/2)\subset \tilde B_{2R}(x,r)\subset \tilde B_{2R}(p,1)$ and $\tilde B_{2R}(x,1/2)$ contains a neck region of depth $\geq 2^{-j}$. Therefore, by the induction hypothesis,
\begin{equation*}
\mu(\tilde B_R (x,r))\geq \mu(\tilde B_R(x,1/2)) \geq L^{-1} R^k 2^{-k}\geq 2^{-k} L^{-1} (Rr)^k.
\end{equation*}

Now, set $B=\max \left( 2^n L, LC(n) 2^{5k} \right)$. By Lemma \ref{lemma:apriori_upper} and Lemma \ref{lemma:apriori_lower}, 
we know that if $\tau\leq \frac{1}{100}$ and
\begin{align*}
R&\geq R(n,H,K,\tau,B),\\
0<\delta&\leq \delta(n,C_I,\Lambda,H,\eta,B|R,\tau),\\
0<\delta'&\leq \delta'(n,C_I,\Lambda,H,\eta,B|R,\tau),
\end{align*}
then 
\begin{equation*}
\tilde L^{-1} R^k\leq \mu(\tilde B_R(p,1)) \leq \tilde L R^k,
\end{equation*}
where $\tilde L=\max(F,G)$. 

Combining this with \eqref{eqn:step1_ua} and \eqref{eqn:step1_la} proves that \eqref{eqn:statement} holds for $L=\max\left( \tilde L, 2^{11k} C(n,\tau)\right)$.

\end{proof}

\section{Neck regions in 3d Ricci flow}
The aim of this section is to show that in a smooth compact 3d Ricci flow satisfying (RF1-4) high curvature regions are described in terms of neck regions. In fact, these neck regions will have the additional advantage that around any $x\in C$, the function $r_x$ bounds the curvature scale (see Definition \ref{def:curvature_scale}) from below. This will allow us to use the neck structure theorem \ref{thm:neck_structure} to prove Theorem \ref{thm:intro_bounded_diameter}.

Let's begin with the notion of curvature scale that will be convenient for our purposes.

\begin{definition}\label{def:curvature_scale}
Let $(M,g(t))_{t\in I}$, be a smooth compact Ricci flow, and $R<+\infty$ some constant. Let $t_0\in J=I\cup \{\sup I\}$. We define the $R$-curvature scale at $(x,t_0)\in M\times J$ as
$$r^R_{\riem}(x,t_0)=\sup\{r>0, \textrm{$t_0-r^2\in I$ and $|\riem|\leq r^{-2}$ in $B(x,t_0-r^2,Rr) \times [t_0-r^2,t_0)$}\}.$$
If $\sup I=0$, we set $r_{\riem}^R(x)= r_{\riem}^R (x,0)$.
\end{definition}
Observe that similar notions of a curvature scale in the literature set $R=1$, see for instance \cite{Bam21}. However, under assumption (RF4), it seems more natural to choose $R\geq 2K$, due to the following lemma.

\begin{lemma}
Let $I\subset \mathbb R$ be an interval with $\sup I=0$, and $(M,g(t))_{t\in I}$ be a smooth compact Ricci flow satisfying (RF4), and let $R\geq 2K$. Then $r^R_{\riem}$ is $R$-scale $\left(\frac{R}{R-K},\frac{R}{R-K}\right)$-Lipschitz.
\end{lemma}

\begin{proof}
Take any $x,y\in M$ and a small enough $\varepsilon>0$. The triangle inequality of Proposition \ref{prop:D_metric} gives that
$$\tilde B_R(x,r_{\riem}(x)+\varepsilon) \subset \tilde B_R(y,r),$$ 
for $r=\frac{R}{R-K} \left( D_R(x,y) + r_{\riem}(x) +\varepsilon\right)\geq  r_{\riem}(x) + \varepsilon$, since for any $z\in \tilde B(x,r_{\riem}(x)+\varepsilon)$ we have
$$D_R(y,z) <\frac{R}{R-K} \left( D_R(x,y) + r_{\riem}(x) \right).$$
By the definition of $r_{\riem}(x)$, we know that there is some $(y,s)\in \tilde B_R(x,r_{\riem}(x)+\varepsilon) \times ( -(r_{\riem}(x)+\varepsilon)^2,0]$ such that $|\riem|(y,s)> (r_{\riem}(x) +\varepsilon)^{-2}$. If $r_{\riem}(y) \geq r$ then $|\riem|\leq r^{-2}$ in $\tilde B_R(x,r_{\riem}(x)+\varepsilon) \times (-(r_{\riem}(x)+\varepsilon)^2,0]$.

Therefore, $r< r_{\riem}(x)+\varepsilon$, which a contradiction. This gives that
$$r_{\riem}(y)< r =\frac{R}{R-K} \left( D_R(x,y) + r_{\riem}(x) +\varepsilon\right).$$
Sending $\varepsilon\rightarrow 0$, gives the result.
\end{proof}

From now on, in order to simplify our notation, we will use the symbol $r_{\riem}$ in place of $r_{\riem}^R$.

\begin{lemma}\label{lemma:2ss_rm}
Let $(M,g(t),x)_{t\in (-2\eta^{-1},0)}$ be a smooth compact 3d Ricci flow. Suppose that $(M,g(t),x)_{t\in (-2\eta^{-1},0)}$ is $(2,\eta)$-selfsimilar at scale $1$. If $\eta=\eta(R)>0$ is small enough then $r_{\riem}(x) \geq 1$.
\end{lemma}
\begin{proof}
The proof is an easy consequence of Perelman's pseudolocality theorem \cite{Per02} and the fact that any 3d shrinking Ricci soliton splitting more than one Euclidean factors should be the Gaussian shrinking soliton, since this is the only flat shrinking Ricci soliton. See also \cite{G19}.
\end{proof} 

From now on, $\eta=\eta(R)$ will always refer to the constant provided by Lemma \ref{lemma:2ss_rm}, while $\mu_{\mathbb R^3}, \mu_{\mathbb S^2\times \mathbb R}, \mu_{\mathbb S^3}$ denote the entropies of the standard gradient shrinking Ricci solitons in $\mathbb R^3,\mathbb S^2\times \mathbb R,\mathbb S^3$ respectively, according to \eqref{eqn:entropy_soliton}.

The following lemma describes the behaviour of a Ricci flow around points with small $R$-curvature radius. It can be seen as a combination of the canonical neighbourhood theorem and the backwards stability of necks, from \cite{Per02} and \cite{KL17}, respectively.

\begin{lemma}\label{lemma:deep_cyl}
Let $(M,g(t))_{t\in [-Q^2,0]}$ be a smooth simply connected compact $3d$ Ricci flow that satisfies (RF1-4), and $R\geq 2K$. If $\xi>0$, $Q\geq Q(C_I,\Lambda,H|R,\xi)$ and $x\in M$ satisfies $r_{\riem}(x)\leq Q^{-2} $, then one of the following holds:
\begin{itemize}
\item[Case 1:] For every $r\in [Q^2 r_{\riem}(x),\xi^{-1}]$, $(M,g(t),y)_{t\in [-Q^2,0]}$ is $\xi$-close to the shrinking sphere at scale $r$ and
$$|\mathcal W_y(\tau)-\mu_{\mathbb S^3}|<\xi$$
for every $y\in \tilde B_{\xi^{-1}}(x,r)$ and $\tau \in ( (Q^2 r_{\riem}(x))^2 \xi,\xi^{-1})$.
\item[Case 2:] For every $r\in [Q^2 r_{\riem}(x),\xi^{-1}]$, $(M,g(t),y)_{t\in [-Q^2,0]}$ is $\xi$-close to the shrinking cylinder at scale $r$ and$$|\mathcal W_y(\tau) - \mu_{\mathbb S^2\times \mathbb R}|<\xi.$$
for every $y\in \tilde B_{\xi^{-1}}(x,r)$ and $\tau \in ( (Q^2 r_{\riem}(x))^2 \xi,\xi^{-1})$.
\end{itemize}

\end{lemma}

\begin{proof}
We will first need to prove the following claim.\\

\noindent \textbf{Claim:} Given any $\zeta\in (0,1)$, if $Q\geq Q(C_I,\Lambda,H|R,\zeta)$ and $(M,g(t))_{t\in [-Q^2,0]}$ is a smooth compact Ricci flow satisfying (RF1-3) and $x\in M$ satisfes $r_{\riem}(x)\leq 1$ then there is $r\in [1,Q]$ such that exactly one of the following holds:\begin{itemize}
\item There is a $\beta >0$ such that $\ric\geq \beta \scal g$ on $M\times [-1,0]$ and $|\mathcal W_x(r^2) - \mu_{\mathbb S^3}|<\zeta$.
\item $M\setminus \tilde B_{\zeta^{-1}}(x,1)\not = \emptyset$ and $|\mathcal W_x(r^2) - \mu_{\mathbb S^2\times \mathbb R}|<\zeta$.
\end{itemize}

\begin{proof}[Proof of the claim]
We may assume that $0<\zeta<\eta$. Fix $0<\zeta'<\zeta<\eta$,  and let $(M,g(t),x)_{t\in (-2\theta^{-1}r^2, 0)}$ be a smooth compact simply connected 3d Ricci flow that satisfies (RF1-3) and suppose that $\mathcal W_x (\theta r^2) - \mathcal W_x(\theta^{-1} r^2) <\theta$. By Corollary 4.2 in \cite{G25}, if $\theta=\theta(C_I,\Lambda,H|\zeta')=2^{-N}<\zeta'$ then there is  $x' \in \tilde B_{2D}(x,r)$, where $D=D(H)$, such that $(M,g(t),x')_{t\in (-2(\zeta')^{-2}r^2,0)}$ is $(0,(\zeta')^2)$-selfsimilar at scale $r>0$, with respect to $\mathcal L_{x',r}$. We will choose $\zeta'$, and consequently $\theta$, small enough according to the following considerations.

Since $M$ is simply connected, and thus orientable, the only possible three dimensional gradient shrinking Ricci solitons that can appear are the standard solitons on $\mathbb S^3$, $\mathbb S^2\times \mathbb R$, or the Gaussian soliton in $\mathbb R^3$. In particular, $\mathcal S_{\textrm{point}}$ is either $\mathbb S^3, \mathbb S^2\times \mathbb R$ or $\mathbb R^3$, respectively, therefore $\tilde B_{\zeta^{-1}}(x,r)\subset\mathcal L_{x',r} $, if $\zeta'$ is small enough, by Definition \ref{def:almost_ss}. Therefore, by Proposition 4.4 in \cite{G25}, $(M,g(t),x)_{t\in [-2(\zeta')^{-1}r^2,0]}$ is $(0,\zeta')$-selfsimilar at scale $1$.

Moreover, combining Lemma 4.1 in \cite{G25} with any easy contradiction argument, we can prove that if $0<\zeta'\leq \zeta'(C_I,\Lambda,H|\zeta)$ then $|\mathcal W_x (r^2) - \mu_{\textrm{sol}}|<\zeta$, where $\mu_{\textrm{sol}}\in\{\mu_{\mathbb R^3},\mu_{\mathbb S^3},\mu_{\mathbb S^2\times \mathbb R}\}$ is the $\mathcal W$ entropy of the corresponding gradient shrinking Ricci soliton.

If  $(M,g(t),x)_{t\in [-2(\zeta')^{-1}r^2,0]}$ is $\zeta'$-close to the shrinking sphere at scale $r$, and $\zeta'$ is small enough, then $M=\tilde B_{(\zeta')^{-1}}(x,r)$, $\mu_{\textrm{sol}}=\mu_{\mathbb S^3}$ and there exists a $\beta>0$ such that $\ric\geq \beta \scal g$ on $M\times [-r^2,0]$, by the standard pinching estimates of three dimensional Ricci flow, see \cite{RFintro} for instance.

On the other hand, if $(M,g(t),x)_{t\in [-2(\zeta')^{-1}r^2,0]}$ is, at scale $r$, $\zeta'$-close to the shrinking cylinder, then we know that $M\not = \tilde B_{\zeta^{-1}}(x,r)$, since in this case $\tilde B_{\zeta^{-1}}(x,r)\subset \tilde B_{(2\zeta')^{-1}}(x,r)$ is not a closed subset of $M$. Moreover,  $\mu_{\textrm{sol}}=\mu_{\mathbb S^2\times \mathbb R}$ and if $0<\zeta'\leq \zeta'(\beta)$  we know that $\ric\geq \beta \scal g$ fails at some point in $M\times [-(\zeta')^{-1}r^2,-\zeta' r^2]$.

Finally, we know that $(M,g(t),x)_{t\in [-2\zeta^{-1}r^2,0]}$ is not $\zeta'$-close to the Gaussian soliton, because this would contradict the assumption $r_{\riem}(x)\geq 1$, due to Lemma \ref{lemma:2ss_rm}, since $\zeta'<\eta$.

Now, let $(M,g(t),x)_{t\in [-Q^2,0]}$ be as in the statement of the claim, and set $r_{j}=\theta^{-j}=2^{Nj}$, $j\geq 0$, so that $[r_{j},r_{j+1}] =[\tilde r_j \theta^{1/2} ,\tilde r_j \theta^{-1/2}]$
for $\tilde r_{j}=  r_j \theta^{-1/2}$.

By the monotonicity of the pointed entropy and (RF2) it follows that if $J=J(\Lambda,\theta)$ is large enough and $Q\geq r_J$  then there exists a $r=\tilde r_{j_0}\in (1,Q)$, for some $j_0=0,\ldots,J-1$, so that
$$\mathcal W_x(r^2 \theta)-\mathcal W_x(r^2\theta^{-1})<\theta.$$
By the previous discussion, the claim follows.

\end{proof}
Since by assumption $r_{\riem}(x) \leq Q^{-2}$, we can apply the claim at scales $Q$ and $r_{\riem}(x)$ to obtain $r_{\riem}(x)\leq r_1 \leq Q r_{\riem}(x)\leq 1\leq Q\leq r_2 \leq Q^2$ such that, for $i=1,2$, $|\mathcal W_x(r_i^2) - \mu_{\textrm{sol},i}|<\zeta$, for some $\mu_{\textrm{sol},i}\in \{\mu_{\mathbb S^3},\mu_{\mathbb S^2\times \mathbb R}\}$.

Suppose that $\mu_{\textrm{sol},2}=\mu_{\mathbb S^3}$. By the claim, we know that $\ric\geq \beta \scal g$ on $M\times [-r_2^2,0]$, for some $\beta>0$, hence, again by the claim, $\mu_{\textrm{sol},1}=\mu_{\mathbb S^3}$.

If on the other hand $\mu_{\textrm{sol,2}}=\mu_{\mathbb S^2\times \mathbb R}$, we know that $M\setminus \tilde B_{\zeta^{-1}}(x,r_1)\supset M\setminus \tilde B_{\zeta^{-1}}(x,r_2) \not=\emptyset$, so the claim gives that $\mu_{\textrm{sol},1} =\mu_{\mathbb S^2 \times \mathbb R}$.

Therefore there is $\mu \in \{\mu_{\mathbb S^3}, \mu_{\mathbb S^2\times\mathbb R}\}$ such that, by monotonicity, $|\mathcal W_x(\rho^2)-\mu|<\zeta$, for any $\rho \in [Qr_{\riem}(x),Q]$ and
\begin{equation*}
\mathcal W_x(r^2\zeta) - \mathcal W_x(r^2 \zeta^{-1})<2\zeta
\end{equation*}
for any $r\in [Q^2 r_{\riem}(x),\xi^{-1}]$, if $Q$ is large enough, depending only on $\zeta$ and $\xi$.

Thus, choosing $0<\zeta\leq\zeta(C_I,\Lambda,H|\xi')$ small enough, by Corollary 4.2 in \cite{G25} we know that for any $r\in [Q^2r_{\riem}(x),\xi^{-1}]$ there is $x_r\in \tilde B_{2D}(x,r)$, $D=D(H)$,  such that $(M,g(t),x_r)_{t\in (-2(\xi')^{-1},0)}$ is $(0,\xi')$-selfsimilar at scale $r$ with respect to $\mathcal L_{x_r,r}$. Morever, if we denote by $\mu_{\textrm{sol},r}$ the $\mathcal W$ entropy of the corresponding soliton, then $\mu_{\textrm{sol},r}=\mu$. This can be proven by the same contradiction argument used in the proof of Corollary 4.2 in \cite{G25}, since along a sequence of counterexamples $|\mathcal W_{x_j} (r^2) - \mu|<\zeta_j \rightarrow 0$ and $\mathcal W_{x_j}(r^2) \rightarrow \mu_{\textrm{sol},r}$, by Proposition 4.2 in \cite{G25}.

Now, take any $r\in [Q^2 r_{\riem}(x),\xi^{-1}]$. As in the proof of the claim we have that  $\tilde B_{\xi^{-1}}(x,r)\subset\mathcal L_{x_r,r}$, if $\xi'$ is small enough, so by Proposition 4.4 in \cite{G25}, we know that $(M,g(t),y)_{t\in (-2\xi^{-1},0)}$ is $(0,\xi)$-selfsimilar at any scale $r$, if $0<\xi'\leq \xi^2$ and $y\in\tilde B_{\xi^{-1}}(x,r)$. Moreover, the proof of Proposition 4.4 in \cite{G25} also gives that $|\mathcal W_y(\tau)-\mu|<\xi$, for every $y\in \tilde B_{\xi^{-1}}(x,r)$ and $\tau\in (\xi r^2,\xi^{-1} r^2)$, if $0<\xi'\leq \xi(C_I,\Lambda,H|\xi)$, since $\mu_{\textrm{sol},r}=\mu$.

If $\mu=\mu_{\textrm{sol},r}=\mu_{\mathbb S^3}$, then $(M,g(t),y)_{t\in (-2\xi^{-1},0)}$ is  $\xi$-close to the shrinking sphere at each scale $r\in [Q^2r_{\riem}(x),\xi^{-1}]$ and at every $y\in \tilde B_{\xi^{-1}}(x,r)$. 

 If, on the other hand, $\mu=\mu_{\textrm{sol},r}=\mu_{\mathbb S^2\times \mathbb R}$, we know that, for every $y\in \tilde B_{\xi^{-1}}(x,r)$,  $(M,g(t),y)_{t\in (-2\xi^{-1},0)}$ is $\xi$-close at scale $r$ to the shrinking cylinder soliton in $\mathbb S^2\times \mathbb R$.
\end{proof}

\begin{lemma}\label{lemma:3d_neck_region}
Let $0<\tau\leq 40^{-4}$, and suppose that $R\geq R(\tau,K)$ and $0<\delta\leq \delta(\tau)$. Suppose that $(M,g(t),p)_{t\in [ -Q^2,0]}$ is a smooth compact 3d Ricci flow that satisfies (RF1-4), that $M\not = \tilde B_R(p,1)$, and that there exists $x\in \tilde B_{2R}(p,1)$ such that $Q^2 r_{\riem}(x)\leq \frac{1}{2}$. 

There exists $\eta'\in (0,1)$,  a finite subset $C\subset \tilde B_{2R}(p,1)$, and a positive function $r: C\rightarrow \mathbb R$ such that if  $Q\geq Q(C_I,\Lambda,H,K|R,\delta)$ then
\begin{equation*}
\mathcal N =\tilde B_{2R}(p,1)\setminus \bigcup_{x\in C} \overline{ \tilde B_R(x,r_x)}
\end{equation*}
is a smooth $(1,\delta,\eta')$-neck region with $\tau\leq 40^{-5}$, such that
\begin{equation}\label{eqn:neck_covering}
\tilde B_R(p,1) \subset \bigcup_{x\in C\cap \tilde B_{\frac{3R}{2}}(p,1)} \tilde B_R(x,r_x),
\end{equation}
with $r_{\riem}\geq Q^{-2} r_x$ in $\tilde B_R(x,r_x)$, for every $x\in C$. 

Moreover, if $R\geq R(C_I,\Lambda,H,K|\tau)$ and $Q\geq Q(C_I,\Lambda,H,K|R,\tau)$ then
\begin{equation}\label{eqn:neck_content}
\sum_{x\in C\cap \tilde B_{\frac{3R}{2}}(p,1)} r_x \leq C, 
\end{equation}
for some constant $C=C(\tau)<+\infty$
\end{lemma}

\begin{proof}
Let $R<+\infty$, $\xi>0$, $\tau \leq 40^{-5}$, and let $Q\geq Q(C_I,\Lambda,H|R,\xi)$ as in Lemma \ref{lemma:deep_cyl}. The constants $R$ and $\xi$ will be chosen large and small enough, respectively, in the course of the proof.

\noindent \textbf{Claim 1.} Set $\gamma=\frac{1}{2}$. There is a $J\geq 2$ such that for every $j=1,\ldots,J$, there is a subset $\{x_f^j\}_{f=1}^{N^j}$ such that 
\begin{align*}
\{x_f^j\}_{f=1}^{N^j}&=\{y_c^j\}_{c=1}^{N^j_c}\cup\{z_b^j\}_{b=1}^{N^j_b} \\
\{y_c^J\}&=\emptyset,
\end{align*}
and
\begin{itemize}
\item For every $j=1,\ldots,J-1$ and $c=1,\ldots, N_c^j$ we know that 
\begin{equation}\label{eqn:c_ball}
Q^2 \inf_{\tilde B_{2R}(y_c^j,\gamma^j)} r_{\riem} < \gamma^j.
\end{equation}
\item For every $j=1,\ldots,J$ and $b=1,\ldots,N_b^J$
\begin{equation}\label{eqn:b_ball}
r_{\riem} \geq Q^{-2} \gamma^j
\end{equation}
in $\tilde B_{2R}(z_b^j,\gamma^j)$.
\end{itemize}
Moreover,
\begin{itemize}
\item For every $j=1,\ldots,J$ the balls $\tilde B_{\tau^3 R} (x_f^j,\gamma^j)$ are mutually disjoint and
\begin{equation}
\label{eqn:covering_choice_claim}
x_f^j \in \bigcup_{c=1}^{N_c^{j-1}} \tilde B_{2R} (y_c^{j-1},\gamma^{j-1}) \setminus \bigcup_{i=1}^{j-1} \bigcup_{b=1}^{N_b^{j-1}} \tilde B_{\tau^3 R} (z_b^i,\gamma^i),
\end{equation}

\item  For every $2\leq j\leq J$, $1\leq i < j$, $f=1,\ldots, N^j$ and $b=1,\ldots, N_b^i$, 
\begin{equation}\label{eqn:covering_t4_disjoint}
\tilde B_{\tau^4 R}(x_f^j,\gamma^j) \cap \tilde B_{\tau^4 R}(z_b^i,\gamma^i) = \emptyset.
\end{equation}
\item For every $2\leq j\leq J$
\begin{equation}\label{eqn:covering_inclusion_j}
\tilde B_{2R}(p,1)\subset \left( \bigcup_{f=1}^{N^j} \tilde B_{\tau^2 R}(x_f^j,\gamma^j) \right) \cup \bigcup_{i=1}^{j-1} \bigcup_{b=1}^{N_b^i} \tilde B_{\tau^2 R}(z_b^i,\gamma^i)
\end{equation}
and
\begin{equation}\label{eqn:covering_inclusion_J}
\tilde B_{2R}(p,1)\subset \bigcup_{i=1}^J \bigcup_{b=1}^{N_b^i} \tilde B_{\tau^2 R}(z_b^i,\gamma^i).
\end{equation}
\end{itemize}

\begin{proof}[Proof of Claim 1]
We first prove that the claim holds for $j=1$.

Choose a maximal collection of points $x^1_f\in \tilde B_{2R}(p,1)$, $f=1,\ldots,N^1$ such that the balls $\tilde B_{\tau^3 R}(x^1_f,\gamma)$ are disjoint and 
\begin{equation}\label{eqn:covering_1}
\tilde B_{2R}(p,1) \subset \bigcup_{f=1}^{N^1} \tilde B_{\tau^2 R}(x^1_f,\gamma).
\end{equation}

We separate the set $\{x^1_f\}_{f=1}^{N^1}$ into two subsets $\{y^1_c\}_{c=1}^{N^1_c}$ and $\{z^1_b\}_{b=1}^{N^1_b}$ according to the following rule:
\begin{itemize}
\item For every $c$, there is $x\in \tilde B_{2R}(y^1_c,\gamma)$ such that $Q^2 r_{\riem}(x) < \gamma$.
\item $r_{\riem} \geq Q^{-2}\gamma$ in each $\tilde B_{2R}(z^1_b,\gamma)$.
\end{itemize}
Since by assumption $\inf_{\tilde B_{2R}(p,1)} Q^2 r_{\riem} <\frac{1}{2}=\gamma$, we know that $\{y_c^1\}$ is non-empty, due to \eqref{eqn:covering_1}.

Now, suppose that the claim holds for $j=k$ and that $\{y_c^k\}$ is non-empty. Choose a maximal collection of points
\begin{equation}\label{eqn:covering_choice}
x_f^{k+1} \in \bigcup_{c=1}^{N_c^k} \tilde B_{2R} (y_c^k,\gamma^k) \setminus \bigcup_{i=1}^k \bigcup_{b=1}^{N_b^k} \tilde B_{\tau^3 R} (z_b^i,\gamma^i),
\end{equation}
$f=1,\ldots,N^{k+1}$, such that the balls $\tilde B_{\tau^3 R} (x_f^{k+1},\gamma^{k+1})$ are mutually disjoint and 
$$\bigcup_{c=1}^{N_c^k} \tilde B_{2R} (y_c^k,\gamma^k) \setminus \bigcup_{i=1}^k \bigcup_{b=1}^{N_b^k} \tilde B_{\tau^3 R} (z_b^i,\gamma^i) \subset \bigcup_{f=1}^{N^{k+1}} \tilde B_{\tau^2 R} (x_f^{k+1},\gamma^{k+1}).$$

Notice that the set $\{x_f^{k+1}\}$ is non-empty, since otherwise we would have the inclusion
$$\bigcup_{c=1}^{N_c^k} \tilde B_{2R} (y_c^k,\gamma^k) \subset  \bigcup_{i=1}^k \bigcup_{b=1}^{N_b^k} \tilde B_{\tau^3 R} (z_b^i,\gamma^i),$$
and therefore \eqref{eqn:b_ball} would imply that $Q^2 r_{\riem} \geq \gamma^k$ in each $\tilde B_{2R}(y_c^k,\gamma^k)$, which contradicts \eqref{eqn:c_ball}.

To prove \eqref{eqn:covering_t4_disjoint} for $j=k+1$, notice that for any $1\leq i<k+1$, $\tilde B_{\tau^4 R}(x_f^{k+1},\gamma^{k+1})\subset \tilde B_{\tau^4 R} (x_f^{k+1},\gamma^i)$, and $B_{\tau^4 R} (x_f^{k+1},\gamma^i)$ is disjoint from any ball $\tilde B_{\tau^4 R} (z_b^i,\gamma^i)$, for $b=1,\ldots,N_b^i$, since $d_{g(-\gamma^{2i})} (x_f^{k+1},z_b^i) \geq \tau^3 R\gamma^i\geq 2\tau^4 R \gamma^i$, by \eqref{eqn:covering_choice}.

Moreover, since the claim holds for $j=k$, we obtain by \eqref{eqn:covering_inclusion_j} that
\begin{align*}
\tilde B_{2R}(p,1) &\subset \left( \bigcup_{f=1}^{N^k} \tilde B_{\tau^2 R}(x_f^k,\gamma^k) \right) \cup \bigcup_{i=1}^{k-1} \bigcup_{b=1}^{N_b^i} \tilde B_{\tau^2 R}(z_b^i,\gamma^i)\\
&=  \bigcup_{c=1}^{N_c^k} \tilde B_{\tau^2 R}(y_c^k,\gamma^k) \cup \bigcup_{b=1}^{N_b^k} \tilde B_{\tau^2 R}(z_b^k,\gamma^k) \cup \bigcup_{i=1}^{k-1} \bigcup_{b=1}^{N_b^i} \tilde B_{\tau^2 R}(z_b^i,\gamma^i)\\
&\subset \bigcup_{f=1}^{N^{k+1}} \tilde B_{\tau^2 R} (x_f^{k+1},\gamma^{k+1}) \cup \bigcup_{i=1}^k \bigcup_{b=1}^{N_b^i} \tilde B_{\tau^2 R}(z_b^i,\gamma^i),
\end{align*}
which proves \eqref{eqn:covering_inclusion_j} for $j=k+1$.

Now, we can decompose 
$$\{x_f^{k+1}\}_{f=1}^{N^{k+1}} =  \{y_c^{k+1}\}_{c=1}^{N_c^{k+1}} \cup \{z_b^{k+1}\}_{b=1}^{N_b^{k+1}}$$
so that
\begin{itemize}
\item For every $c=1,\ldots,N_c^{k+1}$, there is $x\in \tilde B_{2R}(y^{k+1}_c,\gamma^{k+1})$ such that $Q^2 r_{\riem}(x) < \gamma^{k+1}$.
\item $r_{\riem} \geq Q^{-2}\gamma^{k+1}$ in each $\tilde B_{2R}(z_b^{k+1},\gamma^{k+1})$.
\end{itemize}

Thus, starting from $j=1$ we can iterate this procedure, as long as $\{y_c^j\}$ is non-empty. However, since the Ricci flow is smooth up to $t=0$ we know that if the claim holds for $j=k$ then for every $c=1,\ldots,N_c^k$
$$0<\inf_M Q^2 r_{\riem} \leq \inf_{\tilde B_{2R}(y_c^k,\gamma^k)} Q^2 r_{\riem} <\gamma^k,$$
which is impossible for large $k$. Therefore, there is a $J\geq 1$ such that $\{y_c^J\}$ is empty and the iteration stops. Then \eqref{eqn:covering_inclusion_j} for $j=J$ gives \eqref{eqn:covering_inclusion_J}.

\end{proof}

\noindent \textbf{Some properties of $\tilde C= \{z_b^i, 1\leq i\leq J, b=1,\ldots,N_b^i\}$.} Define the function $\tilde r:\tilde C\rightarrow \mathbb R_{>0}$ by $\tilde r_{z_b^i} = \gamma^i$. By Claim 1, for every $\tilde x=z_b^i\in \tilde C$, there is an $\hat x \in \{y_c^{i-1}, c=1,\ldots,N_c^{i-1}\}$ such that 
\begin{equation}\label{eqn:tilde_hat_x}
\tilde x \in B_{2R} (\hat x,\gamma^{-1} \tilde r_{\tilde x}) \subset \tilde B_{2R}(\hat x,r),
\end{equation}
for every $r\geq \gamma^{-1} \tilde r_{\tilde x}$.

\noindent \textbf{Claim 2.} For every $\tilde x\in\tilde C$ and $\hat x$ as in \eqref{eqn:tilde_hat_x}, there is $q \in \tilde B_{2R}(\hat x,\gamma^{-1} \tilde r_{\tilde x})$  such that if $R<
+\infty$ is large enough, $0<\xi\leq \xi(R)$, $r\in [\gamma^{-1}\tilde r_{\tilde x},\xi^{-1}]$ and $y \in \tilde B_{\xi^{-1}}(q,r)$, then $(M,g(t),y)_{t\in (-2\xi^{-1},0)}$ is $\xi$-close to the shrinking cylinder at scale $r$. Moreover,
\begin{equation}
|\mathcal W_y(\tau) - \mu_{\mathbb S^2\times \mathbb R}|<\xi,
\end{equation}
for every $\tau \in \left( \xi r^2,\xi^{-1} r^2 \right)$. In particular, the conclusions of the claim hold for $y=p$ and $r=1$.

\begin{proof}[Proof of Claim 2]
By Claim 1 we know that there is a $q\in \tilde B_{2R}(\hat x,\gamma^{-1}\tilde r_{\tilde x})$
such that
$r_{\riem}(q)\leq Q^{-2} \gamma^{i-1}\leq Q^{-2}$ for which the conclusion of Lemma \ref{lemma:deep_cyl} holds.

If Case 1 of Lemma \ref{lemma:deep_cyl} were to hold, we would know that $(M,g(t),q)_{t\in (-2\xi^{-1},0)}$ is $\xi$-close to the shrinking sphere. Then, if $R$ is large (two times the diameter of the $3$-sphere of radius $\sqrt{4}$ would suffice) and $\xi$ is small enough, we know that $M=\tilde B_R(p,1)$, which contradicts the assumption $M\not= \tilde B_R(p,1)$ of the lemma.

Therefore, Case 2 of Lemma \ref{lemma:deep_cyl} gives that for every $r\in [\gamma^{-1} \tilde r_{\tilde x},\xi^{-1}]$ and $y\in \tilde B_{\xi^{-1}}(q,r)$, $(M,g(t),y)_{t\in (-2\xi^{-1},0)}$ is $\xi$-close to the shrinking cylinder at scale $r$, and
\begin{equation}
|\mathcal W_y(\tau) - \mu_{\mathbb S^2\times \mathbb R}|<\xi,
\end{equation}
for every $\tau \in (\xi r^2,\xi^{-1} r^2)$. 

It remains to prove that the conclusions of Claim 2 hold for $y=p$ and $r=1$, namely that $(M,g(t),p)_{t\in (-2\xi^{-1},0)}$ is $\xi$-close to the cylinder at scale $1$ and $|\mathcal W_p(\tau)-\mu_{\mathbb S^2\times \mathbb R}| <\xi$, for every $\tau\in (\xi,\xi^{-1})$. 

By \eqref{eqn:covering_inclusion_J}, we know that $p\in \tilde B_{\tau^2 R}(\tilde x,\tilde r_{\tilde x})$, for some $\tilde x\in \tilde C$, and let $\hat x$ be as in \eqref{eqn:tilde_hat_x} and $q \in \tilde B_{2R}(\hat x, \gamma^{-1} \tilde r_{\tilde x})$ such that $r_{\riem}(q) \leq Q^{-2}\gamma^{-1} \tilde r_{\tilde x}\leq Q^{-2}$, as guaranteed to exist by Claim 1. By Proposition \ref{prop:D_metric}, we obtain
\begin{equation*}
D_{3R} (p,\hat x) \leq \max(D_{\tau^2 R} (p,\tilde x), D_{2R} (\tilde x,\hat x) )\leq 1,
\end{equation*}
since $\gamma^{-1}\tilde r_{\tilde x}\leq 1$. Similarly, we estimate, if $\xi^{-1}>5R$,
\begin{equation*}
D_{\xi^{-1}}(p,q)\leq D_{5R}(p,q) \leq \max( D_{3R} (p,\hat x), D_{2R}(\hat x,q)) \leq 1.
\end{equation*}
Therefore, $p\in \tilde B_{\xi^{-1}}(q,1)$, so by Case 2 of Lemma \ref{lemma:deep_cyl} we conclude that $(M,g(t),p)_{t\in (-2\xi^{-1},0)}$ is $\xi$-close to the shrinking cylinder at scale $1$, and
$$|\mathcal W_p(\tau) - \mu_{\mathbb S^2\times \mathbb R}|<\xi$$
for every $\tau\in (\xi,\xi^{-1})$.

\end{proof}

\noindent \textbf{Construction of the neck-region.} Let $S=\bigcup_{i=1}^J \bigcup_{b=1}^{N_b^i} \tilde B_{\tau^2 R} (z_b^i , \gamma^i)= \bigcup_{\tilde x\in \tilde C} \tilde B_{\tau^2 R}(\tilde x, \tilde r_{\tilde x})$. Recall \eqref{eqn:covering_inclusion_J} of Claim 1 implies that
$$\tilde B_{2R}(p,1)\subset S.$$

For every $x\in S$ there is a $\tilde x\in \tilde C$ such that $D_{\tau^4 R}(x,\tilde x) = D_{\tau^4 R} (x,\tilde C)= \min_{\tilde y\in\tilde C} D_{\tau^4 R}(x,\tilde y)$, since $\tilde C$ is finite. Moreover, if $D_{\tau^4 R}(x,\tilde x)\leq \frac{1}{4} \tilde r_{\tilde x}$ then $D_{\tau^4 R} (x,\tilde C)=D_{\tau^4 R}(x,\tilde x)$ and $\tilde x$ is the unique element of $\tilde C$ with this property. 

To see this, take any $\tilde y\in \tilde C$ such that $\tilde y \not = \tilde x$.  Since $\tilde B_{\tau^4R}(\tilde x,\tilde r_{\tilde x}) \cap  \tilde B_{\tau^4R}(\tilde y,\tilde r_{\tilde y})=\emptyset$, by Claim 1, the triangle inequality of Proposition \ref{prop:D_metric} gives
\begin{align*}
\tilde r_{\tilde x} &\leq D_{\tau^4 R}(\tilde x,\tilde y) \leq \sigma D_{\tau^4 R}(\tilde x, x)+\sigma D_{\tau^4 R} (x,\tilde y)\\
&\leq \frac{\sigma}{4} \tilde r_{\tilde x} + \sigma D_{\tau^4 R} (x,\tilde y),
\end{align*}
where $\sigma=\frac{\tau^4 R}{\tau^4 R-K}>1$. This implies that
\begin{equation}\label{eqn:tilde_x_unique}
D_{\tau^4 R}(x,\tilde y)  \geq \left(1-\frac{\sigma}{4} \right) \sigma^{-1} \tilde r_{\tilde x} \geq 4 \left(1-\frac{\sigma}{4} \right) \sigma^{-1} D_{\tau^4 R}(x,\tilde x).
\end{equation}
Choosing $R\geq R(K,\tau)$ large enough so that $\sigma\leq \frac{4}{3}$ gives that 
$$D_{\tau^4 R}(x,\tilde y) \geq  2 D_{\tau^4 R}(x,\tilde x).$$
This in particular shows that $\tilde x$ is the unique minimizer of $D_{\tau^4 R}(x, \cdot)$ in $\tilde C$, since by \eqref{eqn:tilde_x_unique}
$$D_{\tau^4 R}(x,\tilde y) \geq \left(1-\frac{\sigma}{4} \right) \sigma^{-1} \tilde r_{\tilde x}>0.$$

Define $r:S\rightarrow \mathbb R_{>0}$ by
$$r_x = \left\{ 
\begin{array}{ll}
\frac{1}{4}\delta^2 \tilde r_{\tilde x}, & x\in \tilde B_{\tau^4 R}(\tilde x,\frac{1}{4}\tilde r_{\tilde x}),\\
\delta^2 D_{\tau^4 R}(x,\tilde C), & \textrm{otherwise}
\end{array}
\right.
$$

Observe that since, by the definition of $S$, for every $x\in S$ there is a $z\in \tilde C$ such that $D_{\tau^2 R}(x,z)<\tilde r_z$,  and if $R\geq R(K,\tau)$ we can use Proposition \ref{prop:D_metric} to estimate
\begin{align*}
D_{\tau^ 4 R}(x,\tilde C) \leq D_{\tau^4 R}(x,z) \leq \frac{\tau^2 R}{\tau^4 R - K} D_{\tau^2 R}(x,z) \leq 2 \tau^{-2}  \tilde r_z<2\tau^{-2}.
\end{align*}
This implies that 
\begin{equation}\label{eqn:r_x_small}
r_x \leq 2\delta^2 \tau^{-2}\leq 1,
\end{equation}
if $0<\delta\leq\delta(\tau)$.

We can then apply Lemma \ref{lemma:10_covering} to obtain a maximal subset $\widehat C\subset S$ such that $$\tilde B_{\tau^2 R}(x,r_x) \cap \tilde B_{\tau^2 R}(y,r_y)=\emptyset$$
for any $x\not = y$ in $\widehat C$, and 
\begin{equation}
\tilde B_{2R}(p,1) \subset S\subset \bigcup_{x\in \widehat C} \tilde B_{\tau^2 R} (x,10r_x).
\end{equation}

We will prove that if $C=\widehat C\cap \tilde B_{2R}(p,1)$ then
$$\mathcal N=\tilde B_{2R}(p,1) \setminus \bigcup_{x\in C} \overline{\tilde B_R(x,r_x)}$$
is a $(1,\delta,\eta')$-neck region, for some $\eta'>0$, if $R\geq R(\tau,K)$, $0<\delta\leq \delta(\tau)$ and $\xi$ is small enough, and thus $Q$ is chosen large enough.

\noindent \textbf{$\mathcal N$ is a $(1,\delta,\eta')$-neck region.} Note that property (n1) of the Definition \ref{def:neck_region} is automatically satisfied by construction.

We will first prove property (n5) of Definition \ref{def:neck_region}. In fact, we will prove that for any $x,y\in S$
\begin{equation}\label{eqn:r_lip_S}
r_x\leq 1.01 r_y + \delta D_R (x,y).
\end{equation}
\begin{enumerate}
\item[\underline{Case 1:}] Suppose that $x\in S \setminus \bigcup_{z\in \tilde C} \tilde B_{\tau^4 R}(z, \frac{1}{4}\tilde r_z)$ and $y\in S$. Then, by definition, $r_x= \delta^2 D_{\tau^4 R}(\tilde x,x) = \delta^2 D_{\tau^4 R}(x,\tilde C)$ and $r_y\geq \delta^2 D_{\tau^4 R}(y, \tilde C)$. Therefore, for any $z\in \tilde C$, using  the triangle inequality of Proposition \ref{prop:D_metric} we obtain
\begin{equation*}
r_x \leq \delta^2 D_{\tau^4 R}(x,z)\leq 1.01\delta^2 (D_{\tau^4 R}(x,y)+ D_{\tau^4 R}(y,z) ),
\end{equation*}
provided that $R\geq R(\tau,K)$. Taking the minimum with respect to $z$ and using again Proposition \ref{prop:D_metric}, we obtain, for $R\geq R(\tau,K)$ 
\begin{equation}\label{eqn:n5_case_1}
\begin{aligned}
r_x &\leq 1.01 \delta^2 D_{\tau^4 R}(y,\tilde C) + 2\delta^2 D_{\tau^4 R}(x,y)\\
&\leq 1.01 r_y + 2\delta^2 D_{\tau^4 R}(x,y)
\end{aligned}
\end{equation}
hence, if in addition $0<\delta\leq \delta(\tau)$,
\begin{equation*}
r_x\leq 1.01 r_y + 4\delta^2  \tau^{-4} D_R(x,y) \leq 1.01 r_y + \delta D_R(x,y),
\end{equation*}
which proves property (n5) in this case. 

\item[\underline{Case 2:}] Suppose that $x \in \tilde B_{\tau^4 R}(\tilde x,\frac{1}{4} r_{\tilde x})$, so that $r_x = \frac{\delta^2}{4} \tilde r_{\tilde x}$. Recall that $\tilde x\in \tilde C$ with this property is unique. If $y\in \tilde B_{\tau^4 R}(\tilde x,\frac{1}{4} \tilde r_{\tilde x})$, then $r_y =\frac{\delta^2}{4} \tilde r_{\tilde x} = r_x$ and there is nothing to prove. Similarly, if $y\not \in \tilde B_{\tau^4 R}(\tilde x,\frac{1}{4} \tilde r_{\tilde x})$ but $D_{\tau^4 R}(y,\tilde C) = D_{\tau^4 R}(y,\tilde x) \geq \frac{1}{4} \tilde r_{\tilde x}$ then by definition $$r_y \geq \delta^2 D_{\tau^4 R}(y,\tilde C) \geq \frac{\delta^2}{4} \tilde r_{\tilde x}=r_x,$$
which again proves proves (n5) in this case. So, suppose that $y\not \in \tilde B_{\tau^4 R}(\tilde x,\frac{1}{4} \tilde r_{\tilde x})$ and that $D_{\tau^4 R}(y,\tilde C)< D_{\tau^4 R}(y,\tilde x)$. Then, for any $z\in \tilde C$, $z\not = \tilde x$, by construction $\tilde B_{\tau^4 R}(\tilde x,\tilde r_{\tilde x})\cap \tilde B_{\tau^4 R}(z,\tilde r_z) =\emptyset$ and the triangle inequality of Proposition \ref{prop:D_metric}, we obtain
\begin{equation*}
r_x =\frac{\delta^2}{4} \tilde r_{\tilde x} \leq \frac{\delta^2}{4} D_{\tau^4 R}(z,\tilde x)\leq \frac{\delta^2}{3} (D_{\tau^4 R}(\tilde x,y) + D_{\tau^4 R}(y,z)),
\end{equation*}
if $R\geq R(\tau,K)$. Thus, taking the minimum with respect to $z\in \tilde C$, $z\not = \tilde x$, we obtain
\begin{equation*}
r_x \leq \frac{\delta^2}{3} D_{\tau^4 R}(\tilde x,y) + \frac{\delta^2}{3} D_{\tau^4 R}(y,\tilde C) \leq \frac{\delta^2}{3} D_{\tau^4 R}(\tilde x,y) + \frac{1}{3} r_y,
\end{equation*}
by the definition of $r_y$. Applying the triangle inequality of Proposition \ref{prop:D_metric} once more we have, for $R\geq R(\tau,K)$,
\begin{align*}
r_x &\leq \frac{\delta^2}{2} D_{\tau^4 R}(\tilde x, x) + \frac{\delta^2}{2}D_{\tau^4 R}(x,y) + \frac{1}{3} r_y\\
&\leq \frac{\delta^2}{8} \tilde r_{\tilde x} + \frac{1}{3} r_y +\frac{\delta^2}{2}D_{\tau^4 R}(x,y)\\
&=\frac{1}{2} r_x + \frac{1}{3} r_y + \frac{\delta^2}{2}D_{\tau^4 R}(x,y).
\end{align*}
Therefore,
\begin{equation*}
\frac{1}{2} r_x \leq \frac{1}{3} r_y + \frac{\delta^2}{2}D_{\tau^4 R}(x,y),
\end{equation*}
which, by Proposition \ref{prop:D_metric}, gives
\begin{equation*}
r_x \leq \frac{2}{3} r_y + \frac{\delta^2}{2} D_{\tau^4 R}(x,y) \leq 1.01 r_y + \delta^2 \tau^{-4} D_R(x,y)\leq 1.01 r_y + \delta D_R(x,y),
\end{equation*}
for $R\geq R(\tau,K)$ and $0<\delta\leq \delta(\tau)$.
\end{enumerate}

To prove properties (n2) and (n3) of Definition \ref{def:neck_region}, let $x\in \widehat C\cap \tilde B_{2R}(p,1)$. We separate two cases.
\begin{itemize}
\item[\underline{Case 1:}] Suppose that $x \in \tilde B_{\tau^4 R}(\tilde x,\frac{1}{4} \tilde r_{\tilde x})$, so that $r_x = \frac{1}{4}\delta^2 \tilde r_{\tilde x}$, and take any $r\in [\gamma^{-1}\tilde r_{\tilde x},\xi^{-1}]$.
Since we may assume that $R\geq R(K)$, we know by Lemma \ref{lemma:balls_inclusion} that 
\begin{equation}\label{eqn:x_tilde_x_c1}
x\in \tilde B_{\tau^4 R}(\tilde x, \frac{1}{4}\tilde r_{\tilde x}) \subset \tilde B_{2R}(\tilde x, r).
\end{equation}
We also know that there is a $q\in \tilde B_{2R}(\hat x,\gamma^{-1}\tilde r_{\tilde x}) \subset \tilde B_{2R}(\hat x,r)$, that satisfies the properties of Claim 2. By \eqref{eqn:tilde_hat_x}, \eqref{eqn:x_tilde_x_c1} and the triangle inequality at time $t=-r^2$ we see that $q \in \tilde B_{6R}(x,r)$.
Choosing $0<\xi\leq \xi(R)$ small enough we obtain that $x\in \tilde B_{\xi^{-1}}(q,r)$, thus, by Claim 2,  $(M,g(t),x)_{t\in (-2\xi^{-1},0)}$ is $\xi$-close to the shrinking cylinder, at all scales $r\in [\gamma^{-1} \tilde r_{\tilde x},\xi^{-1}]$,  and 
$$|\mathcal W_x (\tau)-\mu_{\mathbb S^2\times \mathbb R}|<\xi$$ 
for any $\tau \in ((\gamma^{-1} \tilde r_{\tilde x})^2\xi,\xi^{-1})$.
Therefore, if $0<\xi\leq \xi(\delta)$ and $\eta'$ is smaller than a universal constant (depending on the geometry of the shrinking cylinder Ricci flow), $(M,g(t),x)_{(-2\delta^{-3},0)}$ is $(1,\delta^2)$-selfsimilar but not $(2,\eta')$-selfsimilar at any scale $r\in [r_x,\delta^{-1}]$ and $\mathcal W_x(\delta r_x^2) - \mathcal W_x(\delta^{-1}) < \delta$, since we may assume that $r_x =\frac{1}{4}\delta^2 \tilde r_{\tilde x}\geq \xi^2 \gamma^{-1}\tilde r_{\tilde x}$, by choosing $\xi$ small enough.

\item[\underline{Case 2:}]

Now, if $r_x=\delta^2 D_{\tau^4 R}(x,\tilde C)=\delta^2 D_{\tau^4 R}(x,\tilde x)=\delta^2 \bar r_x\geq \frac{\delta^2}{4} \tilde r_{\tilde x}$, by \eqref{eqn:tilde_hat_x} and the triangle inequality of Proposition \ref{prop:D_metric} we obtain, for $\sigma_{2R}=\frac{2R}{2R-K}$,
\begin{align*}
D_{2R} (x,\hat x) &\leq \sigma_{2R} D_{\tau^4 R}(x,\tilde x) + \sigma_{2R} D_{2R}(\tilde x,\hat x) \\
&\leq \sigma_{2R} \bar r_x +2\sigma_{2R} \tilde r_{\tilde x} \\
&\leq 10 \bar r_x,
\end{align*}
if $R\geq R(K)$. It follows that $x\in \tilde B_{2R}(\hat x,10\bar r_x) \subset \tilde B_{2R}(\hat x,r)$, for every $r\geq 10\bar r_x$.

Moreover, since $x\in S$ we know that for some $z\in \tilde C$, $x\in \tilde B_{\tau^2 R}(z,\tilde r_{\tilde z})\subset \tilde B_{\tau^2 R}(z,1)$, since $\tilde r_{\tilde z}\leq 1$ by construction, and also that $r_z= \frac{\delta^2}{4} \tilde r_{\tilde z} \leq \frac{\delta^2}{4}$. Then, by \eqref{eqn:r_lip_S} and Proposition \ref{prop:D_metric} we obtain
\begin{align*}
r_x = \delta^2 \bar r_x &\leq 1.01 r_z + \delta D_{\tau^2 R} (x,z)\leq \frac{\delta^2}{2} + \delta \leq 2\delta.
\end{align*}
which gives that $10\bar r_x \leq 20\delta^{-1}$.

Let $q\in \tilde B_{2R} (\hat x,\gamma^{-1} \tilde r_{\tilde x})\subset \tilde B_{2R}(\hat x,r)$ be as in Claim 2. Proposition \ref{prop:D_metric} then gives
$$D_{4R}(x,q) \leq \max(D_{2R}(q,\hat x), D_{2R}(\hat x,x)) \leq \max(r,10\bar r_x),$$
so $q\in \tilde B_{4R}(x,r)$ for any $r\geq 20\bar r_x $, and thus $x\in \tilde B_{\xi^{-1}}(q,r)$ if $\xi\leq \xi(R)$.
 
 As in Case 1, since  $\gamma^{-1}\tilde r_{\tilde x} \leq 20\bar r_x \leq 40 \delta^{-1}$ and $r_x =\delta^2 \bar r_x$, it follows that  $(M,g(t),x)_{t\in (-2\delta^{-3},0)}$ is $(1,\delta^2)$-seflsimilar but not $(2,\eta')$-seflsimilar and $\mathcal W_x(\delta r_x^2)-\mathcal W_x(\delta^{-1})<\delta$, at any scale $r\in [r_x, \delta^{-1}]$.

\end{itemize}

Finally, we prove property (n4) of Definition \ref{def:neck_region}. Take any $x\in C$ and $r\in [r_x,1]$ such that $\tilde B_{2R}(x,r)\subset \tilde B_{2R}(p,1)$.

Since $\tilde B_{2R}(p,1)\subset S$, by the claim, we know that by the construction of the set $\widehat C$ that
\begin{equation}\label{eqn:n4_inc}
\tilde B_R(x,r) \subset \bigcup_{y\in \widehat C: \tilde B_{\tau^2 R}(y,10 r_y)\cap \tilde B_R(x,r)\not = \emptyset} \tilde B_{\tau^2 R}(y,10 r_y).
\end{equation}

We will first show that
\begin{equation}\label{eqn:C_intersect_inclusion}
\{y\in \widehat C: \tilde B_{\tau^2 R}(y,10r_y) \cap \tilde B_R(x,1) \not=\emptyset\} \subset \widehat C\cap \tilde B_{\frac{3R}{2}} (x,1) \subset C.
\end{equation}
and that for every $y\in \widehat C$ such that $\tilde B_{\tau^2 R}(y,10r_y) \cap \tilde B_R(x,r)\not= \emptyset$.
\begin{equation}\label{eqn:ry_less_4r}
r_y\leq 4r.
\end{equation}
Take any $w\in \tilde B_{\tau^2 R}(y,10r_y) \cap \tilde B_R(x,1)\subset \tilde B_{\tau^2 R}(y,1) \cap \tilde B_R(x,1)$. The latter inclusion holds by Proposition \ref{lemma:balls_inclusion}, if $R\geq K$, since we can assume that $10r_y\leq 1$ by choosing $0<\delta\leq \delta(\tau)$ small enough, by \eqref{eqn:r_x_small}. Therefore, the triangle inequality at time $t=-1$ gives
$$d_{g(-1)}(x,y) \leq d_{g(-1)}(x,w) + d_{g(-1)}(w,y) \leq R + \tau^2 R \leq \frac{3R}{2},$$
which proves \eqref{eqn:C_intersect_inclusion}.

To prove \eqref{eqn:ry_less_4r}, take any $y\in \widehat C$ such that there is a $z\in\tilde B_{\tau^2 R}(y,10 r_y)\cap \tilde B_R(x,r)$. The triangle inequality of Proposition \ref{prop:D_metric} gives
\begin{align*}
D_R(x,y) & \leq \sigma_R D_R(x,z) + \sigma_R D_R(z,y),\\
&\leq \sigma_R r + \sigma_R D_{\tau^2 R}(z,y),\\
&\leq \sigma_R r + 10 \sigma_R r_y,
\end{align*}
where $\sigma_R=\frac{R}{R-K}$.

Property (n5) then gives
\begin{align*}
r_y &\leq 1.01 r_x + \delta D_R(x,y)\\
&\leq 1.01 r_x +\delta \sigma_R r +10\delta \sigma_R r_y\\
&\leq 2r + \frac{1}{2}r_y,
\end{align*}
if $R\geq R(K)$ and $0<\delta\leq \frac{1}{20}$, which proves \eqref{eqn:ry_less_4r}.

Using \eqref{eqn:C_intersect_inclusion} and then \eqref{eqn:ry_less_4r}, inclusion \eqref{eqn:n4_inc} becomes
\begin{equation}\label{eqn:cover_BRxr}
\tilde B_R(x,r) \subset \bigcup_{y\in C\cap \tilde B_{\frac{3R}{2}}(p,1): \tilde B_{\tau^2 R}(y,10 r_y)\cap \tilde B_R(x,r)\not = \emptyset} \tilde B_{\tau^2 R}(y,10 r_y) \subset \bigcup_{y\in C} \tilde B_{\tau^2 R}(y, 40r).
\end{equation}

Now, observe that on the shrinking cylinder Ricci flow, for any $r>0$ and $\tau\leq 40^{-4}$
$$\tilde B_{\tau^{3/2} R}(x,40r)\subset \tilde B_{\tau^{5/4} R}(x,40\tau^{1/4} r)\subset \tilde B_{\tau^{5/4} R} (x,r).$$

Since for any $y\in C$, $(M,g(t),y)_{[-Q^2,0]}$ is $\delta$-close to the shrinking cylinder at scale $40r$ (provided that $\delta^{-1}>40$), if follows that if $\delta$ is small enough
\begin{equation}\label{eqn:inherited_inclusion}
\tilde B_{\tau^2 R}(y, 40r) \subset \tilde B_{\tau R}(x,r),
\end{equation}
since $\tau^2<\tau^{3/2}$ and $\tau^{5/4}<\tau$.

Therefore, \eqref{eqn:cover_BRxr} gives
\begin{equation*}
\tilde B_R(x,r)\subset \bigcup_{y\in C} \tilde B_{\tau R}(y,r)= \tilde B_{\tau R}(C,r),
\end{equation*}
which proves (n4).

\noindent\textbf{Covering \eqref{eqn:neck_covering} and curvature radius estimate in each $\tilde B_R(x,r_x)$.} 
By \eqref{eqn:C_intersect_inclusion} and inclusion \eqref{eqn:cover_BRxr}, for $r=1$ we obtain that

\begin{equation}\label{eqn:BRp1_covering_almost}
\tilde B_R(p,1) \subset \bigcup_{y\in C\cap \tilde B_{\frac{3R}{2}}(p,1)} \tilde B_{\tau^2 R}(y, 10 r_y).
\end{equation}
Arguing as in the proof of \eqref{eqn:inherited_inclusion}, we obtain that $\tilde B_{\tau^2 R}(y,10 r_y) \subset \tilde B_R(y,r_y)$, for any $y\in C$. This, combined with \eqref{eqn:BRp1_covering_almost}, gives \eqref{eqn:neck_covering}.

It remains to establish the lower bound on the curvature scale in each $\tilde B_R(x,r_x)$, $x\in C$. Take any $x\in C$. Since $C\subset S = \bigcup_{z \in \tilde C} \tilde B_{\tau^2 R} (z, \tilde r_{z})$, we know that $x\in \tilde B_{\tau^2 R} ( z,\tilde r_z)$ for some $z\in \tilde C$. Moreover, by the defining property of $\tilde C$, we know that
\begin{equation}\label{eqn:curv_radius_z}
r_{\riem}\geq Q^{-2} \tilde r_z
\end{equation}
in $\tilde B_{2R}(z,\tilde r_z)$. Thus, it suffices to prove that $\tilde B_R(x,r_x) \subset \tilde B_{2R}(z,\tilde r_z)$ and that $r_x\leq \tilde r_z$.

We will first prove that $r_x\leq \tilde r_z$. Since $z\in \tilde C$ we know that $r_z = \frac{\delta^2}{4} \tilde r_z$, so \eqref{eqn:r_lip_S} gives
$$r_x \leq 1.01 r_z + \delta D_R(x,z)\leq 1.01 \frac{\delta^2}{4} \tilde r_z + \delta \tilde r_z \leq 2 \delta \tilde r_z \leq \tilde r_z$$
It follows that for any $y\in \tilde B_R(x,r_x) \subset \tilde B_R(x, \tilde r_z)$,  we have
$d_{g(-\tilde r_z^2)} (x,y) < R\tilde r_z$. Then,  the triangle inequality at time $t=-\tilde r_z^2$ gives
\begin{equation*}
d_{g(-\tilde r_z^2)} (y,z) \leq d_{g(-\tilde r_z^2)}(x,y) + d_{g(-\tilde r_z^2)}(x,z)\leq R \tilde r_z + \tau^2 R \tilde r_z\leq 2R \tilde r_z,
\end{equation*}
thus $\tilde B_R(x,r_x) \subset \tilde B_{2R}(z,\tilde r_z)$.

Therefore, \eqref{eqn:curv_radius_z} gives that $r_{\riem} \geq Q^{-2} \tilde r_z \geq Q^{-2} r_x$
in $\tilde B_R(x,r_x)$, which proves the result.

\noindent \textbf{Proof of the content estimate \eqref{eqn:neck_content}.}
By the neck-structure Theorem \ref{thm:neck_structure}, if $R\geq R(C_I,\Lambda,H, K,\eta'|\tau) $ and $0<\delta\leq \delta(C_I,\Lambda,H,K,\eta'|R,\tau)$ we know that for every $x\in C$ and $r_x\leq r$ with $\tilde B_{2R}(x,r)\subset \tilde B_{2R}(p,1)$,
\begin{equation}\label{eqn:L_apriori}
L^{-1} Rr \leq \mu(\tilde B_R(x,r)) \leq L Rr
\end{equation}
for some constant $L=L(\tau)<+\infty$.

Moreover, by \eqref{eqn:r_x_small}, we can assume that $r_x\leq 2^{-5}$, for any $x\in C$, by making $0<\delta\leq \delta(\tau)$. We also know, from Claim 2, that $\mathcal W_p(\delta)-\mathcal W_p(\delta^{-1})<\delta$. Thus, by Remark \ref{rmk:instead_p_C}, we can apply Lemma \ref{lemma:apriori_Ahlfors_larger_balls}, under \eqref{eqn:L_apriori},
 to obtain that
 $$\sum_{x\in C\cap\tilde B_{\frac{3R}{2}}(p,1)} R r_x=\mu(\tilde B_{\frac{3R}{2}} (p,1) \leq \hat B R,$$
 for some $\hat B=\hat B(L)$.
 
 This suffices to prove \eqref{eqn:neck_content} when $R\geq R(C_I,\Lambda,H,K|\tau)$ and $Q\geq Q(C_I,\Lambda,H,K|R,\tau)$.
 
\begin{theorem}[Neck decomposition]\label{thm:neck_decomp}
Let $(M,g(t))_{t\in [-Q^2,0]}$ be a smooth compact simply connected Ricci flow satisfying (RF1-4). If $R\geq R(C_I,\Lambda,H,K)$ and $Q\geq Q(C_I,\Lambda,H,K|R)$ then exactly one of the following holds:
\begin{enumerate}
\item There is a $q\in M$ and $0<r\leq 1$ such that $M=\tilde B_R(q,r)$ and $r_{\riem}\geq \frac{1}{2} Q^{-2} r$ on $M$.
\item There is a cover
$$M= \bigcup_{c=1}^{N_c} \tilde B_{R}(y_c,1) \cup \bigcup_{b=1}^{N_b} \tilde B_{R}(z_b,1)$$
such that 
\begin{itemize}
\item $N_b+N_c\leq C(n,C_I,\Lambda,\vol_{g(-1)}(M))$.
\item For every $z_b=1,\ldots,N_b$, $r_{\riem} \geq \frac{1}{2}Q^{-2}$ in $\tilde B_{R}(z_b,1)$.
\item For every $y_c=1,\ldots,N_c$, $\inf Q^2 r_{\riem} <\frac{1}{2}$ and there is a cover
$$\tilde B_R(y_c,1) \subset \bigcup_{x\in C_c}\tilde B_R(x,r_x)$$
such that $r_{\riem} \geq Q^{-2} r_x$ in each ball $\tilde B_R(x,r_x)$ and 
\begin{equation*}
\sum_{x\in C_c} r_x \leq C,
\end{equation*}
for some universal constant $C<+\infty$.
\end{itemize}
\end{enumerate}
\end{theorem}
\begin{proof}
Let $\tilde B_{\frac{R}{2}}(x_i,1)$, $i=1,\ldots, N$, $x_i\in M$, be a maximal collection of mutually disjoint balls, so that  
$$M= \bigcup_{i=1}^N \tilde B_R(x_i,1).$$
By Remark \ref{rmk:nc}, we know that
\begin{align*}
N \kappa_{nc} \leq \sum_{i=1}^N \vol_{g(-1)}(\tilde B_R(x_i,1)) \leq \vol_{g(-1)}(M),
\end{align*}
thus $N\leq C(\kappa_{nc},\vol_{g(-1)}(M))=C(n,C_I,\Lambda,\vol_{g(-1)}(M)$, since $\kappa_{nc}$ depends only on $n,C_I$ and $\Lambda$.

Suppose that $N=1$, so that $M=\tilde B_R(x_1,1)$. If $r_{\riem} \geq \frac{1}{2} Q^{-2}$ on $M$, then there is nothing to prove. If on the other hand $\inf_M Q^2 r_{\riem} <\frac{1}{2}$, Case 1 of Lemma \ref{lemma:deep_cyl} at scale $1$, applies, since otherwise $M\not = \tilde B_R(x_1,1)$.

Therefore, for some $q\in M$, $(M,g(t),q)_{t\in [-Q^2,0]}$ is $\xi$-close to the shrinking sphere, at scale $r=Q^2 r_{\riem}(q)<1$,  such that
$$r_{\riem}(q) \leq 2 \inf_M r_{\riem} < Q^{-2}.$$
This implies that if $R$ is large enough and $\xi$ is small enough, then $M=\tilde B_R(q,r)$ and $r_{\riem}\geq \frac{1}{2} Q^{-2} r$ on $M$.

Suppose that $N>1$, so $M\not=\tilde B_{x_i}(p,1)$ for any $i=1,\ldots,N$. We can then separate the set $\{x_i\}_{i=1}^N$ into two subsets $\{y_c\}_{c=1}^{N_c}$ and $\{z_b\}_{b=1}^{N_b}$ so that
\begin{itemize}
\item For every $b$, $r_{\riem} \geq \frac{1}{2}Q^{-2}$ in $\tilde B_{2R}(z_b,1)$.
\item For every $c$, $\inf_{\tilde B_{2R}(y_c,1)} Q^2 r_{\riem}<\frac{1}{2}$.
\end{itemize}
The result then follows from Lemma \ref{lemma:3d_neck_region}, setting $\tau=40^{-4}$, provided that $R$ and $Q$ are chosen large enough.
\end{proof}

\begin{corollary}\label{cor:L1_bound_Q}
Let $(M,g(t))_{t\in [-Q^2,0]}$ be a smooth compact simply connected Ricci flow satisfying (RF1-4). Suppose also that for every $x \in M$, $0<r\leq 1$ and $R>0$ such that $1+R^2\leq Q^2$ we have that
\begin{equation}\label{eqn:non_inflating}
\vol_{g(t)}(\tilde B_R (x,r))\leq \kappa_{\textrm{infl}} (R r)^3,
\end{equation}
for some constant $\kappa_{\textrm{infl}}>0$.

 If $Q = Q(C_I,\Lambda,H,K)$, then there exists a constant $C=C(C_I,\Lambda,H,K,\kappa_{\textrm{infl}},\vol_{g(-1)}(M))<+\infty$ such that
\begin{equation*}
\int_M |\riem|(\cdot,0) d\vol_{g(0)}\leq \int_M r_{\riem}^{-2} d\vol_{g(0)} \leq C.
\end{equation*}
\end{corollary}
\begin{proof}
Apply Theorem \ref{thm:neck_decomp} to obtain $R=R(C_I,\Lambda,H,K)$ and $Q =Q(C_I,\Lambda,H,K)$, such that $4\leq 1+R^2\leq Q^2$ and the conclusion of Theorem \ref{thm:neck_decomp} holds. By \eqref{eqn:non_inflating} we also have
\begin{equation}
\vol_{g(-1)}(\tilde B_R(x,r))\leq \kappa_{\textrm{infl}} (Rr)^3,
\end{equation}
for any $0<r\leq 1$.

By the evolution equation for the scalar curvature and the maximum principle we obtain that for any $t\in [-1,0]$
$$\scal_{g(t)} \geq -\frac{3}{4}, $$
hence, for any $x\in M$ and $0<r\leq 1$
$$\frac{d}{dt} \vol_{g(t)}(\tilde B_R(x,r))\leq \frac{3}{4} \vol_{g(t)}(\tilde B_R(x,r)).$$
This implies that 
\begin{equation}\label{eqn:later_ni}
\vol_{g(0)}(\tilde B_R(x,r))\leq e^{\frac{3r^2}{4}}\vol_{g(-r^2)}(\tilde B_R(x,r))\leq e^{\frac{3}{4}} \kappa_{\textrm{infl}} (R r)^3.
\end{equation}

Now, if Case 1 of Theorem \ref{thm:neck_decomp}  applies, we know that $M=\tilde B_R(p,r)$ for some $r\in (0,1]$ and $r_{\riem}\geq \frac{1}{2}Q^{-2} r$ on $M$. Using \eqref{eqn:later_ni}, it follows that
\begin{align*}
\int_M |\riem|(\cdot,0) d\vol_{g(0)} & \leq \int_M r_{\riem}^{-2}(x) d\vol_{g(0)} (x)\\
&\leq 4 Q^4 \int_{\tilde B_R(p,r)} r^{-2} d\vol_{g(0)} \\
&\leq 4 Q^4 r^{-2} \vol_{g(0)} (\tilde B_R(p,r)) \\
&\leq 4Q^4 e^{\frac{3}{4}} \kappa_{\textrm{infl}} R^3. 
\end{align*}
On the other hand, if Case 2 of Theorem \ref{thm:neck_decomp} applies, we obtain, using \eqref{eqn:later_ni},
\begin{align*}
&\int_M |\riem|(\cdot,0) d\vol_{g(0)} \leq \int_M r_{\riem}^{-2} d \vol_{g(0)}\\
&\leq \sum_{c=1}^{N_c} \int_{\tilde B_R(y_c,1)} r_{\riem}^{-2}d\vol_{g(0)} + \sum_{b=1}^{N_b} \int_{\tilde B_R(z_b,1)} r_{\riem}^{-2} d\vol_{g(0)}\\
&\leq \sum_{c=1}^{N_c} \sum_{x\in C_c} \int_{\tilde B_R(x,r_x)} r_{\riem}^{-2} d\vol_{g(0)} + 4Q^4 \sum_{b=1}^{N_b} \vol_{g(0)} (\tilde B_R(z_b,1)),\\
&\leq\sum_{c=1}^{N_c} \sum_{x\in C_c} \int_{\tilde B_R(x,r_x)} r_{\riem}^{-2}(y) d\vol_{g(0)}(y) + 4Q^4 \sum_{b=1}^{N_b} \vol_{g(0)} (\tilde B_R(z_b,1)),\\
&\leq\sum_{c=1}^{N_c} \sum_{x\in C_c}  Q^4 r_x^{-2} \vol_{g(0)} (\tilde B_R(x,r_x)) + 4Q^4 \sum_{b=1}^{N_b} \vol_{g(0)} (\tilde B_R(z_b,1)),\\
&\leq Q^4 e^{\frac{3}{4}} \kappa_{\textrm{infl}} \sum_{c=1}^{N_c} \sum_{x\in C_c} r_x+ N_b e^{\frac{3}{4}} \kappa_{\textrm{infl}} R^3,\\
&\leq Q^4 e^{\frac{3}{4}} \kappa_{\textrm{infl}} N_c C R^3+ N_b 4Q^4 e^{\frac{3}{4}}\kappa_{\textrm{infl}} R^3,\\
&\leq C(C_I,\Lambda,H,K,\kappa_{\textrm{infl}},\vol_{g(-1)}(M)) R^3.
\end{align*}
\end{proof}

\begin{theorem}\label{thm:bounded_diameter}
Let $(M,g(t))_{t\in [0,T]}$ be a smooth Ricci flow on a compact 3d manifold and suppose that 
\begin{equation}\label{eqn:type_I_T}
|\riem|\leq \frac{C_I}{T-t}
\end{equation}
on $M\times [0,T)$. Then there is a $C=C(C_I,g(0),T)<+\infty$ such that $$\int_M r_{\riem}^{-2} (\cdot,T) d\vol_{g(0)}+\diam_{g(T)}(M)\leq C.$$
\end{theorem}
\begin{proof}
Passing to the universal cover, we may assume that $M$ is simply connected.

Given any $Q<+\infty$ we can choose a scale $\rho>0$ such that $Q^2  \rho^2=T$. The rescaled flow $\tilde g(t)=\rho^{-2} g(T+\rho^2 t)$ is then defined for all $t\in [-T \rho^{-2} ,0] = [-Q^2,0]$. It is clear that 
$$|\riem_{\tilde g}|\leq \frac{C_I}{|t|},$$
for every $t\in [-Q^2,0]$, hence there is $K=K(C_I)<+\infty$ such that $\tilde g$ satisfies the distance distortion estimate (RF4), by Remark \ref{rmk:RF1_RF4}.

Let $\Lambda<+\infty$ be a large positive constant such that $\nu(g(0),T)\geq -\Lambda$. By the scaling behaviour and the monotonicity of $\nu$ we then know that for any $t\in [-Q^2,0]$
\begin{align*}
\nu(\tilde g(t),|t|)&=\nu(\rho^{-2}g(T+\rho^2 t),|t|),\\
&=\nu(g(T+\rho^2 t),\rho^2|t|),\\
&\geq \nu(g(0),T) \geq -\Lambda,
\end{align*}
by the scaling behaviour and the monotonicity of $\nu$.

By Propositions 2.7 and 2.8 in \cite{MM15}, there is a  constant $H=H(g(0))<+\infty$ such that $\tilde g$ satisfies (RF3).

Moreover, there is positive constant $C_{\textrm{scal}}=C_{\textrm{scal}}(C_I,T)<+\infty$ such that $\scal_{g(0)}\geq - C_{\textrm{scal}}$. By the evolution equation of the scalar curvature under Ricci flow we then know that $\scal_{g(t)}\geq -C_{\textrm{scal}}$ for every $t\in [0,T]$ which implies that
$$\frac{d}{dt} \vol_{g(t)}(M) \leq C_{\textrm{scal}} \vol_{g(t)}(M),$$
so $\vol_{g(t)}(M) \leq e^{C_{\textrm{scal}} T} \vol_{g(0)}(M)$.

Therefore, we also know that
$$\vol_{\tilde g(-1)} (M) = \rho^{-n} \vol_{g(T-\rho^2)}(M)\leq e^{C_{\textrm{scal}} T}  T^{-n/2} Q^n \vol_{g(0)}(M).$$

Moreover, since $(M,g(t))_{t\in [0,T]}$  satisfies \eqref{eqn:type_I_T}, there is $a(n,C_I)<+\infty$ such that for every $t_0\in [0,T]$
$$\scal_{g(t)} \leq \frac{a}{t_0-t}$$
in $B(x_0,t_0,s)\times [t_0- s^2, t_0]$, for every $x_0\in M$ and $0<s\leq \sqrt{|t_0|}$.

For any $0<r\leq 1$ and $x\in M$
\begin{equation}\label{eqn:vol_g_to_tilde_g}
\vol_{\tilde g(-r^2)}(\tilde B_R(x,r)) =\rho^{-3}\vol_{g(T-\rho^2 r^2)} (B(x,T-\rho^2r^2, Rr\rho).
\end{equation}
Now, as long as $Q^2> 1+R^2$, we know that
$$T-\rho^2 r^2 - (Rr\rho)^2\geq T-(1+R^2) \rho^2 r^2=Q^2\rho^2 - (1+R^2)\rho^2=\left(Q^2-(1+R^2)\right) \rho^2>0,$$
hence Theorem 1.1 in \cite{Zhang12} gives that there is a positive constant $\kappa_{\textrm{infl}} =\kappa_{\textrm{infl}}(g(0),C_I,T)$ such that
$$\vol_{g(T-\rho^2 r^2)} (B(x,T-\rho^2r^2, Rr\rho)\leq \kappa_{\textrm{infl}} (Rr\rho)^3.$$
Applying this estimate in \eqref{eqn:vol_g_to_tilde_g} we obtain
\begin{equation*}
\vol_{\tilde g(-r^2)}(\tilde B_R(x,r)) \leq \kappa_{\textrm{infl}} (R r)^3,
\end{equation*}
for every $x\in M$ and $0<r\leq 1$.

We have thus shown that $(M,\tilde g(t))_{t\in [-Q^2,0]}$ satisfies the assumptions of  Corollary \ref{cor:L1_bound_Q}, namely the a priori assumptions (RF1-4) and the non-inflating estimate \eqref{eqn:non_inflating} hold. Therefore, if $Q=Q(C_I,\Lambda,H,K)$ is as in Corollary \ref{cor:L1_bound_Q}, there exists a constant $C=C(C_I,g(0),T)<+\infty$ such that
$$\int_M r_{\riem,\tilde g}^{-2} d\vol_{\tilde g(0)}\leq C.$$
Thus, since $g(T)= \rho^2 \tilde g(0)$ 
\begin{align*}
\int_M |\riem|(\cdot,T) d\vol_{g(T)}\leq \int_M r_{\riem}^{-2}(\cdot,T) d\vol_{g(0)}= \int_M \rho^{-2} r_{\riem,\tilde g}^{-2} \rho^3 d\vol_{\tilde g(0)}\leq C \rho =C T^{1/2} Q^{-1}.
\end{align*}
The result follows by applying Theorem 2.4 in \cite{Top05}.

\end{proof}

\begin{proof}[Proof of Theorem \ref{thm:intro_bounded_diameter}]
Given any $(M,g(t))_{t\in [0,T)}$ as in the statement of Theorem \ref{thm:intro_bounded_diameter}, and every $\bar t\in [T/2,T)$, set $g'(t)= g(t+\bar t-\frac{T}{2})$, $t\in [0,T/2]$. Then $(M,g'(t))_{t\in [0,T/2]}$ satisfies the assumptions of Theorem \ref{thm:bounded_diameter}, which implies the result for $t=\bar t$. On the other hand, the result for any $t\in [0,T/2)$ follows trivially from the Type I bound on the curvature and standard distance distortion estimates.
\end{proof}

\end{proof}

\providecommand{\bysame}{\leavevmode\hbox to3em{\hrulefill}\thinspace}
\providecommand{\MR}{\relax\ifhmode\unskip\space\fi MR }
\providecommand{\MRhref}[2]{%
  \href{http://www.ams.org/mathscinet-getitem?mr=#1}{#2}
}
\providecommand{\href}[2]{#2}


\end{document}